\documentclass{amsart}

\usepackage{amssymb}
\usepackage{hyperref}

\newtheorem{theorem}{Theorem}[section]
\newtheorem{lemma}[theorem]{Lemma}
\newtheorem{corollary}[theorem]{Corollary}
\newtheorem{proposition}[theorem]{Proposition}

\theoremstyle{definition}
\newtheorem{definition}[theorem]{Definition}

\theoremstyle{remark}

\newcommand\R{\mathbb{R}}
\newcommand\Z{\mathbb{Z}}
\newcommand\C{\mathbb{C}}
\newcommand\N{\mathbb{N}}
\newcommand\T{\mathbb{T}}
\newcommand{\HS}{{\mathfrak{I}_{2}}}

\newcommand{\ind}{\mathbf{1}}
\newcommand\eps{\varepsilon}
\newcommand\vk{\varkappa}

\newcommand{\qtq}[1]{\quad\text{#1}\quad}

\newcommand{\LL}{\mathcal L}
\newcommand{\Lax}{\mathbf{L}}
\newcommand{\LaxR}{\mathbf{R}}

\newcommand{\mKdV}{\text{\slshape mKdV}}

\newcommand{\Wspb}{W^{\smash[b]{s,p}}_b}

\newcommand{\loc}{\text{loc}}
\newcommand{\op}{\textup{op}}

\DeclareMathOperator{\tr}{tr}
\DeclareMathOperator{\dist}{dist}
\DeclareMathOperator{\supp}{supp}
\DeclareMathOperator{\sign}{sign}
\DeclareMathOperator{\E}{\mathbb{E}}
\DeclareMathOperator{\PP}{\mathbb{P}}

\newcommand{\slp}[1]{{\slshape(}#1{\slshape)}}  

\newcommand{\jb}[1]{\langle #1 \rangle}

\numberwithin{equation}{section}

\allowdisplaybreaks

\begin{document}

\title[Invariant measures for mKdV and KdV in infinite volume]{Invariant measures for mKdV and KdV\\in infinite volume}

\author[J.~Forlano]{Justin Forlano }
\address{
Justin Forlano, School of Mathematics\\ The University of Edinburgh\\ and The Maxwell Institute for the Mathematical Sciences\\ James Clerk Maxwell Building\\
The King's Buildings\\ Peter Guthrie Tait Road\\ Edinburgh\\ EH9 3FD\\ United Kingdom\\ and\\
Department of Mathematics\\ University of California\\ Los Angeles\\ CA 90095\\ USA}
\email{j.forlano@ed.ac.uk}

\author[R.~Killip]{Rowan Killip }
\address{
Rowan Killip, Department of Mathematics\\
University of California\\Los Angeles\\CA 90095\\USA}
\email{killip@math.ucla.edu}

\author[M.~Vi\c{s}an]
{Monica Vi\c{s}an }
\address{
Monica Vi\c{s}an, Department of Mathematics\\
University of California\\Los Angeles\\CA 90095\\USA}
\email{visan@math.ucla.edu}



\begin{abstract}
We construct dynamics for the defocusing real-valued (Miura) mKdV equation on the real line with initial data distributed according to Gibbs measure. We also prove that Gibbs measure is invariant under these dynamics.  On the way, we provide a new proof of the invariance of the Gibbs measure under mKdV on the torus.

Building on these results, we construct new measure-preserving dynamics for the KdV equation on the whole real line.  Samples from this family of measures exhibit the same local regularity as white-noise, but completely different statistics!
\end{abstract}

\maketitle
\tableofcontents

\section{Introduction}

The defocusing modified Korteweg--de Vries equation,
\begin{equation}\label{mkdv}\tag{mKdV}
\partial_t q=-\partial_x^{3}q+6 q^2 \partial_x q,
\end{equation}
describes the time evolution of a real-valued field $q(t,x)$ which may be defined either for $x\in\R$ or for $x\in\T_{L}:=\R/2L\Z$ with $L>0$.  The former will be referred to as the line or infinite-volume case, while the latter constitutes the finite-volume or periodic case. 

This equation was originally introduced by Miura \cite{MR0252825} through its connection to the Korteweg--de Vries equation.  It was subsequently realized that \eqref{mkdv} is an effective model for many physical scenarios.

We are interested in studying the \eqref{mkdv} dynamics at thermal equilibrium, particularly, in the infinite-volume setting.  This necessitates solving \eqref{mkdv} with initial data chosen at random from the Gibbs distribution (described below) and demonstrating that these solutions lead to a measure-preserving transformation of the Gibbs state.  This is inherently a low regularity problem, further compounded by the lack of spatial decay. Indeed, the Gibbs state is a stationary process.   The absence of spatial decay places this problem outside the scope of the much-studied well-posedness theory in $H^s$ spaces. 

The description of Gibbs states rests on the underlying Hamiltonian structure.  The \eqref{mkdv} dynamics is generated by the Hamiltonian
\begin{align}\label{mkdv hamil}
H_\mKdV(q)=\int \tfrac{1}{2}(\partial_x q)^2 + \tfrac{1}{2}q^4 \, dx
\end{align}
with respect to the Poisson structure 
\begin{align}\label{PB}
\{ F,G\}:=\int \tfrac{\delta F}{\delta q}(x)  \; \partial_x \big( \tfrac{\delta G}{\delta q}\big)(x) \,dx.
\end{align}
Here, functional derivatives are defined in the traditional fashion:
\begin{align}\label{functl deriv}
\frac{d}{ds}\bigg\vert_{s=0} F(q+sf) = dF\vert_{q}(f)=\int \frac{\delta F}{\delta q}(x) f(x)\,dx.
\end{align}

A second important conserved quantity associated with \eqref{mkdv} is 
\begin{align}\label{mkdv mass}
M(q)=\int \tfrac{1}{2} q^2 \, dx,
\end{align}
which may be regarded as either the mass or the momentum of the wave, depending on the physical context.  This functional generates translations under the Poisson structure \eqref{PB}.

The two conserved quantities combine naturally into a single free energy
\begin{align}\label{Free E}
H_\mKdV(q) - \mu M(q),
\end{align}
where $\mu\in\R$ plays the role of a chemical potential.  This in turn gives rise to the following two-parameter family of Gibbs states:
\begin{align}\label{formal Gibbs}
\tfrac{1}{Z}\exp\{ - \beta(H_\mKdV - \mu M) \} \, d(\text{phase volume}).
\end{align}
Here, $\beta>0$ denotes the inverse temperature and  $Z=Z(\beta,\mu)$ is the partition function, which is included in order to normalize this as a probability measure.

The natural notion of phase-space volume is induced by the underlying Poisson structure.  In view of \eqref{PB}, we ought to adopt product Lebesgue measure on the Fourier coefficients; however, such an infinite product measure does not exist.

We describe the rigorous construction of \eqref{formal Gibbs} in Section~\ref{S:Gibbs}.  For the moment however, we wish to give the reader an intuitive picture of samples from the Gibbs state \eqref{formal Gibbs}.  To this end, we rewrite the free energy density as
\begin{align}\label{Gibbs F E}
\tfrac12 (\partial_x q)^2 + U(q(x))  \qtq{with} U(y) = \tfrac12 y^4 - \tfrac\mu2 y^2  .
\end{align}
The first term penalizes rapid changes in $q(x)$, ensuring that it maintains the regularity of Brownian motion.  The second term endeavors to keep the process in regions where $U$ is small.   In Corollary~\ref{C:BM}, we give a precise meaning to this heuristic in a manner inspired by the It\^o SDE representation of the process:
\begin{align}\label{E:Ito}
dq = \tfrac{\psi_0'(q)}{\psi_0(q)} dt + dB
\end{align}
where $\psi_0$ denotes the ground-state of the Schr\"odinger operator \eqref{HGibbs}.


Building on this intuitive picture, we see that when $\mu$ is large and positive, the process $q(x)$ will remain near the two well-separated minima of $U(y)$ for the overwhelming fraction of $x$, with very occasional transits between them.   This leads us to regard the corresponding Gibbs state as representing a dilute soup of kinks evolving within a thermal background.  We remind the reader that kinks are the topological solitons indigenous to \eqref{mkdv}.  The desire to cover this physically important case was a major driver for including the chemical potential $\mu$ in our definition of the Gibbs state.

To truly fulfill its promise of describing thermal equilibrium,  a Gibbs state ought to be shown to be dynamically invariant.  For finite-dimensional systems, this is typically elementary.  However, it becomes a subtle matter even for the most basic infinite-dimensional systems, such as a gas in infinite volume or a classical field theory (even) in finite volume.  In both cases, the total energy is infinite.  The model treated here combines both infinities: the infinite energy density of a classical field (the famous Rayleigh--Jeans catastrophe) \emph{and} infinite volume.  Nevertheless, we are able to construct dynamics and prove that it conserves the Gibbs law:

\begin{theorem}\label{T:1.1}
For all values of temperature and chemical potential, solutions of \eqref{mkdv} define a measure-preserving group of transformations of the Gibbs state.
\end{theorem}

A more detailed formulation of this result can be found in Theorem~\ref{T:1936}.  In particular, we show that $q(t,x)$ is both a distributional and green solution of \eqref{mkdv} and that $\langle x\rangle^{-1}q(t,x)$ is an $L^2(\R)$-continuous function of time.  For the notion of green solution, see Definition~\ref{D:mkdv green}.  One of the virtues of this notion is that such solutions retain the conservation laws characteristic of \eqref{mkdv}.

As an offshoot of our analysis, we will also provide a new proof of invariance of the Gibbs state for \eqref{mkdv} in finite volume; see Theorem~\ref{T:4.3}.  This result was first demonstrated in \cite{MR1309539}.  The corresponding result for generalized KdV equations was subsequently proved in \cite{MR4504646,MR3489633}.  In each case, the approach was to develop a suitable deterministic well-posedness result.  This was then combined with finite-dimensional approximation to obtain preservation of measure, which was in turn used to construct global solutions (almost surely).  We note that the Fourier-based methods employed in all these works are ill-suited to the problem in infinite volume and so do not inform our analysis here.

There are two alternate approaches to the well-posedness of \eqref{mkdv} that have proven successful at the low regularity inherent to samples from the Gibbs state.  Using the construction of action/angle variables for the Korteweg--de Vries
equation
\begin{equation}\label{kdv}\tag{KdV}
\partial_t w=-\partial_x^{3}w + 6 w\partial_x w
\end{equation}
in \cite{MR2267286} and the link to \eqref{mkdv} provided by the Miura transformation
\begin{equation}\label{miuraI}
w=q'+q^2,
\end{equation}
the global well-posedness of \eqref{mkdv} in $L^2(\T)$ was proved in \cite{MR2131061}.  At this moment, there is no analogue of this approach demonstrating well-posedness in $L^2(\R)$.  This incompatibility with the infinite-volume regime leads us to regard this as a nonviable approach to proving Theorem~\ref{T:1.1}.

The second method is based on the commuting flows paradigm of \cite{KV}, which was applied to \eqref{mkdv} in \cite{Forlano,HGKV,2309.12773}.  This is the method that we will be applying.  One key part of this approach is the introduction of regularized flows that approximate \eqref{mkdv} but have much better high-frequency behavior.  This better behavior allows us to construct solutions to these flows by simple contraction mapping arguments.  Moreover, the weaker transport of these approximating flows aids in passing to the infinite-volume limit.  It is only in infinite volume that we will remove the regularization; this in turn is aided by the fact that the regularized flows commute with one another, as well as with \eqref{mkdv}.   

The foregoing discussion has focused on the well-posedness theory of \eqref{mkdv} because that theory both informs the methods we employ and embodies the strength of the results that we obtain, including, conservation laws and the group property.  By comparison, the existence of weak solutions (without uniqueness) to \eqref{mkdv} is much older; for example, this was shown for initial data in $L^2(\R)$ already in \cite{MR0759907}.  There is a corresponding theory of random weak solutions based on Prokhorov's compactness criterion for probability measures; this theory was developed in a manner applicable to \eqref{mkdv} in infinite-volume in the paper \cite{MR4077096}.

The final section of the paper is devoted to the construction of new invariant measures for \eqref{kdv} by applying the Miura map \eqref{miuraI} to the Gibbs-distributed solutions to \eqref{mkdv} described in Theorem~\ref{T:1.1}.  The culmination of these results may be summarized as follows:

\begin{theorem}\label{T:1.2}
The \eqref{kdv} evolution defines a measure-preserving group of transformations of the image of the Gibbs state under the Miura map.
\end{theorem}

As we have noted earlier, samples $q$ from the \eqref{mkdv} Gibbs state have the regularity of Brownian motion.  Correspondingly, the random distributions $w$ defined through \eqref{miuraI} have the regularity of white noise.  This informal assertion is made precise in Proposition~\ref{P:7.3}. Moreover, we show there that this process is not white noise: it is not Gaussian and its law is mutually singular to that of white noise!

The question of whether \eqref{kdv} preserves white noise has attracted significant attention.  In finite volume, this was settled via a variety of methods in \cite{MR2540076,MR2762393,MR3034397,MR2899812,MR2365449}.  In the whole-line setting, this was proved in \cite{KMV}.

At first glance, it is natural to wonder what invariant measure for \eqref{mkdv} might correspond (under the Miura map) to white noise in infinite volume.  The answer is none!  A characteristic property of the image of the Miura map is that Schr\"odinger operators with such potentials are positive semi-definite.  By comparison, the white noise potential leads to a Schr\"odinger operator whose spectrum is the whole real line.   

While it is not atypical for completely integrable systems to admit multiple Gibbs measures through the incorporation of higher conserved quantities into the construction of free energies (cf. \cite{MR3518561,MR1811126,MR3112200}), the invariant measures described by Theorem~\ref{T:1.2} do not fit this mold at all.  Indeed, we contend that the soliton gases described by Theorem~\ref{T:1.2} lie outside the traditional Gibbsian framework, which makes them very interesting indeed.

As described in \cite{KMV}, white noise corresponds to the lowest-regularity coercive conservation law for \eqref{kdv}, namely, $M(w):=\int \frac12w^2\,dx$.  The introduction of higher conservation laws into the Gibbs weight would make the samples smoother in contradiction to the regularity of the samples just described.  There is a second issue: with the exception of $M(w)$, the conservation laws associated to \eqref{kdv} are not coercive.  In particular, there is no Gibbs state analogous to \eqref{formal Gibbs} because the potential energy is unbounded below.  

Failure of coercivity of the Hamiltonian is a typical feature for focusing equations, such as \eqref{kdv}.  Following the lead of \cite{MR0939505}, it has become standard practice to impose a mass cutoff to overcome this.  Unfortunately, this approach is only applicable in finite volume; see \cite{KMV} for references and further discussion.


\subsection{Overview of the paper}
While the overall strategy we employ in this paper owes much to \cite{KMV}, the implementation is very different.   Although white noise is very irregular, it is probabilistically very simple: it is Gaussian and independent on disjoint intervals.  The law of Gibbs samples is much less explicit and, of course, carries correlations. Correspondingly, the probabilistic analysis here is essentially disjoint from that of \cite{KMV}.

Perhaps the most conspicuous metamorphosis relative to \cite{KMV} is the absence of any multiscale analysis (the most technical part of that work).  
This transformation of the method relies on our identification of a subtle difference between the two models in the earliest phase of the argument, which then initiates a cascade of improvements.  This difference is the relativistic speed limit inherent to the Lax operator associated to \eqref{mkdv}, which has the structure of a Dirac operator.   This manifests already in Proposition~\ref{P:Lax Op} and in the universality of the Combes--Thomas estimate \eqref{CT est}.

Section~\ref{S:Gibbs} is devoted to the rigorous construction of the Gibbs measure \eqref{formal Gibbs} and the elaboration of its key properties.  Our approach is via finite marginals and the Feynman--Kac formula. Although this approach is not new (cf. \cite{MR1277197,SimonInt,SXu}), it is significantly less well-known than its principal competitor.  Nowadays, the preeminent approach is to begin with a Gaussian measure, which is then modified by an appropriate Radon--Nikodym derivative.  This latter method, popularized by \cite{MR0939505}, is indeed very effective in finite volume.  It is also well-matched to the prevailing approach to the construction of solutions, based on treating the nonlinearity perturbatively.  Nevertheless, this approach is ill-suited to our problem because in infinite volume, Gibbs measure is mutually singular to the corresponding Gaussian field!

One other approach to the thermodynamic limit is to first construct finite-volume Gibbs measures via their Radon--Nikodym derivative and then use correlation inequalities to prove tightness of the resulting sequence of laws.   This style of argument underpins the study the Gibbs state for NLS in  \cite{MR1777342}, but relies on the chemical potential being negative.

Given the breakdown of the naive notion \eqref{formal Gibbs} described above, it is natural to ask if there is an intrinsic way of identifying the proper Gibbs state.  It was for this purpose that the DLR condition was introduced in \cite{MR0231434,MR0256687}.  It says that an infinite-volume measure is a Gibbs state if the conditional law on finite windows matches the Gibbs law given the state exterior to the window.  For particle systems, this has become the settled definition; however, for fields, the rule \eqref{formal Gibbs} breaks down even in finite volume.

An alternate approach to defining classical Gibbs states originates from the KMS (Kubo--Martin--Schwinger) condition in quantum statistical physics (cf. \cite{MR0219283}).  The classical analogue of this is discussed, for example, in \cite{GallVerb,Katz,Mermin}.  For particle systems, it was shown to be equivalent to the DLR condition in \cite{MR0443784}.

We will prove that our Gibbs states satisfy not only the traditional classical KMS condition (Theorem~\ref{T:KMS}), but also the natural synthesis of the DLR and KMS conditions; see Proposition~\ref{PROP:DLRKMS}.   At the beginning of subsection~\ref{SS:KMS}, we illustrate how the classical KMS condition is closely tied to the question of invariance in the finite-dimensional setting.  Our inspiration for considering it here and for regarding it as way-point on the road to proving dynamical invariance was the recent paper \cite{AS}, which considers classical fields in finite volume.

Our construction of infinite-volume dynamics rests on a sequence of finite-volume approximations.  For this purpose, it is crucial to have a strong \emph{quantitative} coupling of the finite-volume Gibbs laws and that in infinite volume.  This is provided by subsection~\ref{ss2.2}.

Subsections~\ref{SS:mix} and~\ref{SS:diffusion} demonstrate mixing of the Gibbs process and its relation to diffusions, respectively, showing how directly such properties may be deduced from our chosen construction of the Gibbs state.

Section~\ref{S:3} is devoted to the deterministic theory of the \eqref{mkdv} Lax operator under the very mild condition \eqref{det L hyp}.  A particular focus of this section is to obtain bounds on the diagonal Green's functions and to control their dependence on $q$.  Due to the wide scope afforded by our hypothesis \eqref{det L hyp}, we believe many of the results of this section will be useful in the study of \eqref{mkdv} more broadly.  For example, these results cover almost periodic and step-like initial data. 

In the latter part of Section~\ref{S:3}, we review some parts of the method of commuting flows as implemented in \cite{Forlano}.  The definition of the regularized Hamiltonians is presented in \eqref{Hk defn} and the key properties of the induced flows are covered by Proposition~\ref{P:HK well}.

In Section~\ref{S:invariance} we demonstrate the invariance of the finite-volume Gibbs measure in two cases: In Theorem~\ref{THM:HKinv} we treat the flows associated to the regularized Hamiltonians $H_\kappa$.   Then, we prove invariance under \eqref{mkdv} by sending $\kappa\to\infty$; see  Theorem~\ref{T:4.3}.

The goal of Section~\ref{S:5} is to construct regularized flows in infinite volume and to prove that they leave the Gibbs measure invariant; see Theorem~\ref{T:1601}.  Note that it is much easier to take an infinite volume limit for the regularized flows, as opposed to \eqref{mkdv}, because they have much weaker transportation properties.

In Section~\ref{S:6}, we send $\kappa\to\infty$ in the $H_\kappa$ flows in order to construct solutions to \eqref{mkdv} with Gibbsian initial data in infinite volume.   Amongst other things, we prove that these solutions have the group property, preserve Gibbs measure, and satisfy microscopic conservation laws; see Theorem~\ref{T:1936}.  The Poisson commutativity of the regularized Hamiltonians (for differing values of $\kappa$) plays an essential role in being able to control this $\kappa\to\infty$ limit: It allows us to control the difference of two such flows by just analyzing the flow under the difference Hamiltonian; see Proposition~\ref{P:1951}.

Section~\ref{S:7} begins by discussing the connection between \eqref{kdv} and \eqref{mkdv} at the level of the Lax operators.  The latter part is divided into two subsections:  The first describes the image of the Gibbs state under the Miura mapping, while the second demonstrates the existence of \eqref{kdv} solutions with such initial data and the invariance of measure.

\subsection*{Notations}

We reserve the norm notation  $\|\cdot\|$ for infinite-dimensional Banach spaces and use simple absolute value signs to represent the norm of vectors in finite-dimensional spaces such as $\C$, $\C^2$, and the space $\C^{2\times 2}$ of $2\times2$ matrices.

In addition to the traditional $W^{s,p}(\R)$ and $H^s(\R)$ Sobolev spaces and norms, we will also employ $H^s_\kappa$, which is equal to $H^s(\R)$ as a set, but has norm
$$
\| f \|_{H^s_\kappa}^2 = \int |\hat f(\xi)|^2 (\kappa^2 + \xi^2)^s \,d\xi .
$$ 

\subsection*{Acknowledgements}
J.F. was supported by the European Research Council (grant no. 864138 ``SingStochDispDyn'').  R.K. was supported by NSF grant DMS-2154022.  M.V. was supported by NSF grant DMS-2054194.

\section{Gibbs measures}\label{S:Gibbs}

By virtue of the scaling symmetry
\begin{equation}\label{scaling}
q(t,x) \mapsto \lambda q(\lambda^3 t,\lambda x)
\end{equation}
enjoyed by \eqref{mkdv}, we may reduce our consideration of the Gibbs states to the case $\beta=1$.  This rescaling alters the chemical potential $\mu$; however, its precise value will not play a role in our analysis and henceforth, we will regard it as fixed.  The dependence of implicit and explicit constants on $\mu$ will not be tracked, nor typically even noted.

We will first introduce the Gibbs measure, both on the line and on the circle, by simply declaring its marginal distributions at finitely many sample points.  This definition and certain fundamental properties of the resulting process are presented in Proposition~\ref{PROP:Gibbs}.

Later, we will justify that this measure is indeed the Gibbs measure by showing that it satisfies the classical KMS condition, which we find to be just the right tool for proving dynamical invariance.   We will also show that the whole-line Gibbs measure is the limit of those on ever larger tori; indeed, our chosen construction is specifically geared toward our need to prove quantitative estimates on the proximity of the whole-line distributions to those on large tori, as well as our need to couple together the random initial data following these laws.

Consider the one-dimensional Schr\"odinger operator
\begin{align}
\LL := -\tfrac12\partial_y^2 + V \qtq{with} V(y) := \tfrac{1}{2}y^4 - \tfrac\mu2 y^2 + V(0), \label{HGibbs}
\end{align}
where $V(0)\in\R$ is chosen so that the lowest eigenvalue of $\LL$ is zero.  The essential self-adjointness of this anharmonic oscillator on $C^\infty_c(\R)$ is shown in \cite{RS2} (by five different methods!).  The Rellich criterion (cf. \cite{RS4}) guarantees that $\LL$ has compact resolvent.  All eigenvalues of $\LL$ are necessarily simple as can be inferred by noting that the Wronskian of two eigenvectors is both constant and integrable.  In view of these properties, the spectral theory of $\LL$ is completely described by its eigenvalues
$$
0=\lambda_0<\lambda_1<\lambda_2<\cdots \qtq{and $L^2$-normalized eigenvectors} \psi_0,\psi_1,\psi_2,\ldots
$$
The eigenvector $\psi_0$ is necessarily sign-definite; we choose it to be positive.

As we will soon see, the finite marginal distributions of the Gibbs state may be expressed in terms of the heat kernel $e^{-t\LL}(y,y')$ associated to this Schr\"odinger operator.  Let us review some of its basic properties:

\begin{proposition}\label{P:heat}
The heat kernel $e^{-t\LL}(y,y')$ is non-negative and continuous as $(t,y,y')$ ranges over $(0,\infty)\times\R\times\R$; moreover, there exists $\eps>0$ so that 
\begin{alignat}{2}
 e^{-t\LL}(y,y')  &\lesssim t^{-1/2}\exp\bigl( - \tfrac1{2t}|y-y'|^2 \bigr) &&\qtq{uniformly for} 0< t\leq 1, \label{short heat} \\
 e^{-t\LL}(y,y')  &\lesssim \exp( -\eps|y|^3 - \eps|y'|^3 ) &&\qtq{uniformly for} 1\leq t,  \label{long heat} 
\end{alignat}
and we have the Hilbert--Schmidt bound
\begin{align}
\bigl\| e^{-t\LL}(y,y') - \psi_0(y)\psi_0(y') \bigr\|_{L^2_{y,y'}(\R^2)}^2 \lesssim \bigl(1 + \tfrac1t\bigr) e^{-2\eps t} \quad\text{uniformly for $t >0 .$}\!\! \label{HS heat}
\end{align}
Here $\psi_0$ denotes the $L^2$-normalized ground state of $\LL$, which itself satisfies
\begin{align}\label{Carmona}
\bigl| \psi_0(y) \bigr| &\lesssim \exp( -\eps|y|^3 ).
\end{align}
\end{proposition}

\begin{proof}
For a proof that the heat kernel is continuous, see \cite[Thm B.7.1]{MR0670130}.

The right-hand side of \eqref{short heat} is essentially the fundamental solution to the heat equation (a numerical prefactor has been omitted).  By the Feynman--Kac formula, this provides an upper bound for the heat kernel associated to any one-dimensional Schr\"odinger operator with positive potential.  Although the potential $V(y)$ is not positive, it is bounded from below. This energy shift produces an extra factor of $e^{Ct}$, which is clearly bounded for $t\in[0,1]$.

We may also compare the operator $\LL$ with the harmonic oscillator
\begin{align}\label{harmo}
\LL_0 = -\tfrac12\partial_y^2 + \tfrac{1}{2}y^2 - \tfrac12 .
\end{align}
Clearly, we may choose $C>0$ so that $\LL \geq \LL_0 - C$ as quadratic forms.   As the eigenvalues of $\LL_0$ are known to be precisely the non-negative integers, we infer that the eigenvalues of $\LL$ grow faster than some arithmetic progression.  Thus,
\begin{align}
\big\| e^{-t\LL}(y,y') - \psi_0(y)\psi_0(y') \bigr\|_{L^2_{y,y'}(\R^2)}^2
	= \sum_{n=1}^\infty e^{-2\lambda_n t} \lesssim (1 + t^{-1}) e^{-2\lambda_1 t} .
\end{align}
This proves \eqref{HS heat}.

We turn now to \eqref{long heat}. As $e^{-s''\LL}$ is an $L^2$-contraction for $s''\geq0$, so
\begin{align}\label{ss's''}
e^{-t\LL}(y,y') \leq \| e^{-s \LL} \delta_y \|_{L^2} \| e^{-s' \LL} \delta_{y'} \|_{L^2} 
	\qtq{provided} t\geq s+s'.
\end{align}
To estimate the terms on the right, we employ duality and the Feynman--Kac formula: Given $f\in L^2(\R)$,
\begin{align}\label{CFK}
\langle f, e^{-s \LL} \delta_y \rangle = \E\bigl\{ f\bigl(y+b(s)\bigr)\exp\bigl( - {\textstyle\int_0^s} V(y+b(\tau))\,d\tau \bigr)\bigr\}
\end{align}
where $s>0$ and $b(\tau)$ denotes a standard Brownian motion.  The next step is to apply Cauchy--Schwarz.  This produces two factors.  The first is easy to estimate:
\begin{align}\label{0821}
\E\Bigl\{ f\bigl(y+b(s)\bigr)^2 \Bigr\} &= \langle f^2, e^{s\Delta/2} \delta_y \rangle \lesssim s^{-1/2} \| f\|_{L^2}^2.
\end{align}

For the second factor, we partition the probability space on the basis of the event
$$
E = \bigl\{  |b(\tau)| < \tfrac12 |y|, \ \forall \tau\in[0,s]\bigr\}.
$$
By the reflection principle,
\begin{align}
\PP\bigl( E^c \bigr) \lesssim  \exp\bigl( - \tfrac{1}{4s} |y|^2 \bigr),
\end{align}
which guarantees that 
\begin{align}\label{0822}
\E\bigl\{  \exp\bigl( - 2 {\textstyle\int_0^s} V(y+b(\tau))\,d\tau \bigr) \cdot 1_{E^c} \bigr\}
		\lesssim \exp\bigl( - \tfrac{1}{4s} |y|^2 \bigr)
\end{align}
uniformly for $s\in(0,\frac12]$.  On the event $E$ we note that
\begin{align}\label{0823}
\E\bigl\{  \exp\bigl( - 2{\textstyle\int_0^s} V(y+b(\tau))\,d\tau \bigr) \cdot 1_{E} \bigr\}
		&\leq \exp\bigl( - 2 s \inf\{ V(y') : |y'| > \tfrac12 |y|\}\bigr) \\
&\lesssim  \exp\bigl( - \tfrac s{99} |y|^4 \bigr).
\end{align}
Again this estimate holds uniformly for $s\in(0,\frac12]$.

Combining all these ingredients in \eqref{CFK}, we deduce that for $s\in(0,\frac12]$,
\begin{align}\label{CFK'}
\| e^{-s \LL} \delta_y \|_{L^2} \lesssim s^{-1/2} \bigl[ \exp\bigl( - \tfrac{1}{4s} |y|^2 \bigr) + \exp\bigl( - \tfrac s{99} |y|^4  \bigr) \bigr].
\end{align}
Choosing $s=\min\{ \frac12, |y|^{-1}\}$, we may then deduce \eqref{long heat} from \eqref{ss's''}.

Lastly, \eqref{Carmona} follows from \eqref{long heat} and \eqref{HS heat} by sending $t\to\infty$.
\end{proof}

We are now ready to introduce the Gibbs distributions, both in finite and infinite volume.  These are described by measures on appropriate spaces of continuous functions.  In the finite-volume case, $(C(\T_L),\mathcal B)$ denotes the (separable) Banach space of bounded continuous real-valued functions endowed with the supremum norm and its associated Borel $\sigma$-algebra.   The case of infinite volume requires a small adjustment because Gibbs samples are almost surely unbounded functions.  We use $(C(\R),\mathcal B)$ to denote the Polish space of continuous real-valued functions endowed with the compact-open topology and its associated Borel $\sigma$-algebra.  The random process will be denoted $q(x)$; a subscript indicating the half-period will only be added when the distinction is needed.

\begin{proposition}[Gibbs measure] \label{PROP:Gibbs}
There is a unique probability measure defined on $(C(\R), \mathcal{B})$ satisfying
\begin{align}\label{FKmuL2}
\E\Biggl[ \, \prod_{j=1}^{n} f_{j}\bigl(q_{\infty} (x_j)\bigr) \Biggr]
	&= \bigl\langle \psi_0, f_1 e^{-(x_2-x_1) \LL}f_2 \cdots e^{-(x_n-x_{n-1}) \LL}f_n \psi_0 \bigr\rangle,
\end{align}
for every choice of $x_1<\cdots<x_n$ and bounded continuous functions $f_1,\ldots,f_n$.

Likewise, for each $L\in 2^\N$ there is a unique probability measure on $(C(\T_L), \mathcal{B})$ satisfying
\begin{align*}
\E\Biggl[  \,\prod_{j=1}^{n}f_{j}(q_{L}(x_j))  \Biggr]   
	& = \frac{
		\tr \Big\{  e^{-(x_1+L)\LL}f_1  e^{-(x_2-x_1)\LL}f_2 \cdots  e^{-(x_n-x_{n-1}) \LL}f_n e^{-(L-x_n) \LL}  \Big\} }
		{ \tr\{e^{-2L\LL} \} }, 
\end{align*}
for every choice of  $-L<x_1<\cdots<x_n\leq L$ and $f_1,\ldots,f_n$ as above.

These processes are stationary (i.e., the law is preserved by spatial translations) and admit the following bounds uniformly in the period $L\in 2^\N\cup\{\infty\}:$ \\
\textup{(i)}  For any $1\leq p <\infty$, $0\leq s<\frac12$, $b > \frac1p$, and $1\leq r<\infty$,
\begin{align}\label{Wsp moments}
\E \Bigl[ \bigl\| \langle x\rangle^{-b} q(x) \bigr\|_{W^{s,p}(\R)} ^r \Bigr] < \infty .
\end{align}
\textup{(ii)} For any finite interval $I\subset\R$ and any $\lambda >0$,
\begin{align}
\PP\bigg[ \sup_{x\in I} |q_{L}(x)| > \lambda \biggr] \lesssim \bigl( 1 + |I| \bigr) e^{-c \lambda ^{3}}  .\label{max on I}
\end{align}
In particular, for any $1\leq r<\infty$,
\begin{align}
\E \big[ \| q \|_{L^{\infty}(I)}^{r}\big] \lesssim_{r} \log\bigl(2+|I|\bigr)^{\frac{r}{3}}. \label{log13} 
\end{align}
\end{proposition}

\begin{proof}
Uniqueness boils down to this question: Do the measures of cylinder sets (which may be computed directly from the finite marginals) uniquely determine the measures of all Borel sets?  The answer is yes.  This is a standard application of the $\pi$-$\lambda$ Theorem.

As the finite marginal distributions are invariant under spatial translations, this uniqueness guarantees that any process with such marginals is automatically stationary.

Turning to the question of existence, let us first observe that the systems of finite marginals presented above are consistent and correspond to probability measures.  In the infinite-volume setting, this follows from the positivity and the semigroup property of the heat kernel together with the positivity of $\psi_0$, its $L^2$ normalization, and the fact that it is a zero-energy eigenfunction.   In the finite-volume case, one must also ensure that the operator traces are well defined; this follows from the Hilbert--Schmidt bound \eqref{HS heat}.

In order to prove the existence of a corresponding process, we wish to employ Kolmogorov's extension theorem in concert with his continuity criterion.  The latter requires adequate moment bounds; these are readily deduced from Proposition~\ref{P:heat}.
Indeed, this result guarantees that 
\begin{equation}\label{LD one}
\PP\{ |q(x)| \geq \lambda \}  \lesssim e^{-\eps \lambda^3}
\end{equation}
and similarly that 
\begin{equation}\label{LD two}
\PP\{ |q(x)-q(x')| \geq \lambda \} \lesssim \exp\bigl( - \tfrac{\lambda^2}{2|x-x'|} \bigr)
	\qtq{whenever} |x-x'|\leq 1. 
\end{equation}
In the finite-volume setting, $|x-x'|$ denotes the natural distance in $\R/2L\Z$, namely,
\begin{equation}\label{quotient metric}
|x-x'| = \min\{ |x - x'+2mL| : m\in\Z \} \qtq{for points} x,x' \in \R/2L\Z .
\end{equation}

Let us turn now to \eqref{Wsp moments}.  The driving idea is that for $0<s<1$, the $W^{s,p}(\R)$ norm takes the form
\begin{align*}
\| f \|_{{W}^{s,p}}^p   = \|f\|_{L^{p}(\R)}^p + \| f \|_{\dot{W}^{s,p}}^p
\qtq{with} \| f \|_{\dot{W}^{s,p}}^p = \iint_{\R^2} \frac{ |f(x)-f(x')|^p }{|x-x'|^{1+ps}} dx\,dx' ,
\end{align*}
which is amenable to analysis via \eqref{LD one} and \eqref{LD two}.  (This is also one approach to proving Kolmogorov's continuity criterion; see, for example, \cite{MR1267569}.) 

We begin with the $L^p$ norm.  From \eqref{LD one}, we see that
\begin{align*}
\bigl\| q(x) \bigr\|_{L^r(d\PP)} \lesssim r^{1/3} .
\end{align*}
Thus when $bp>1$ and $r\geq p$, the triangle inequality in $L^{r/p}(d\PP)$ shows
\begin{align}\label{Lp b}
\E\Bigl[ \bigl\| \langle x\rangle^{-b} q(x) \|_{L^p(\R)}^r\Bigr]
	\lesssim_r \Bigl[{\textstyle\int}\langle x\rangle^{-bp}\,dx \Bigr]^{r/p} < \infty.
\end{align}

Regarding differences, \eqref{LD one} and \eqref{LD two} guarantee that
\begin{align*}
\bigl\| \langle x\rangle^{-b} q(x)  - \langle x'\rangle^{-b} q(x') \bigr\|_{L^r(d\PP)}
	\lesssim_r  \bigl[ \langle x\rangle^{-b} + \langle x'\rangle^{-b} \bigr]  \biggl[{\frac{|x-x'|}{\langle x - x'\rangle}} \biggr]^{1/2}.
\end{align*}
Thus, when $r\geq p$, we have
\begin{align*}
\E\Bigl[ \bigl\| \langle x\rangle^{-b} q(x) \|_{\dot W^{s,p}(\R)}^r\Bigr]
	\lesssim_r \biggl[\iint_{\R^2} \frac{ \langle x\rangle^{-bp} + \langle x'\rangle^{-bp} }{|x-x'|^{1+ps}}
		\Bigl[{\tfrac{|x-x'|}{\langle x - x'\rangle}} \Bigr]^{\frac p2}\, dx\,dx' \biggr]^{r/p} 
\end{align*}
which is finite, so long as $bp>1$ and $0<s<\frac12$.  Together with \eqref{Lp b}, this proves \eqref{Wsp moments}.  The restriction to $r\geq p$ is inconsequential because higher moments automatically control lower moments via Jensen's inequality. 

We turn now to \eqref{max on I}, noting the following reduction:  It suffices to show that
\begin{align}
\PP\bigg[ \sup_{x\in J} |q(x)| > \lambda \biggr] \lesssim e^{-\tilde c \lambda ^{3}} \label{max on I'}
	\qtq{when} \lambda \geq 1 \qtq{and} |J| \leq \lambda^{-1}.
\end{align}
Indeed, if $\lambda \leq 1$ then \eqref{max on I} is inferior to the fact that all probabilities are bounded by one.  On the other hand, when $\lambda\geq 1$, \eqref{max on I'} and the union bound show
\begin{align*}
\PP\bigg[ \sup_{x\in I} |q(x)| > \lambda \biggr] \lesssim (1+|I|) \lambda e^{-\tilde c \lambda ^{3}}
	\lesssim (1+|I|) e^{- c \lambda ^{3}},
\end{align*}
for any $c<\tilde c$.

To prove \eqref{max on I'}, we may simply consider $J=[0,\lambda^{-1}]$.  We first note that
$$
\sup_{x\in J} |q(x)| \leq |q(0)| + \sum_{n=1}^\infty \ \sup_{1\leq k\leq 2^n} \ \relax
	\Bigl| q\bigl(\tfrac{k}{\lambda 2^{n}}\bigr) - q\bigl(\tfrac{k-1}{\lambda 2^{n}}\bigr)\Bigr| .
$$
Thus, choosing $M=0.8$, we may employ \eqref{LD one} and \eqref{LD two} to deduce that
\begin{align*}
\PP\Big[ \sup_{x\in J} |q(x)| > \tfrac{\lambda}{1-M} \Bigr]
	&\leq \PP\Big[ |q(0)| > \lambda \Bigr] + \sum_{n=1}^\infty \sum_{k=1}^{2^n}
		\PP \Bigl[ \Bigl| q\bigl(\tfrac{k}{\lambda 2^{n}}\bigr) - q\bigl(\tfrac{k-1}{\lambda 2^{n}}\bigr)\Bigr| > \lambda M^n \Bigr] \\
&\lesssim e^{-\eps \lambda^3} + \sum_{n=1}^\infty 2^n \exp\bigl\{ - \tfrac12 \lambda^3 M^{2n} 2^{n} \bigr\} \\
&\lesssim e^{-\eps \lambda^3} + e^{- \lambda^3/ 4}.
\end{align*}
This proves \eqref{max on I'} and so \eqref{max on I}.  The moment bound \eqref{log13} follows from \eqref{max on I}.
\end{proof}

The description of the Gibbs state in terms of finite marginals is very convenient for understanding certain conditional distributions.  We are most interested in the following situation: given the value of $q$ at the two ends of an arc on the circle, what is the conditional law of $q$ on this arc and the complementary arc.  (In infinite volume, our two observation points divide the line into three intervals.)  Let us first state the key observations; the physical significance will be addressed later.

Given points $x_0<x_1<\cdots < x_{n+1}$ and $y_0,\ldots,y_{n+1}\in\R$, we define
\begin{align}\label{CondFM}
P\Bigl(\,\begin{matrix} y_1, \ldots, y_n \\ x_1, \ldots, x_n \end{matrix}\Big| \begin{matrix} y_0,y_{n+1} \\ x_0,x_{n+1} \end{matrix} \Bigr)
&= \frac{ \prod_{j=0}^n e^{-(x_{j+1}-x_j)\LL}(y_j,y_{j+1}) }{ e^{-(x_{n+1}-x_{0})\LL}(y_0,y_{n+1}) } .
\end{align}
This expresses the sought after conditional laws in the following sense:

\begin{lemma}[Conditional laws]\label{L:condit}
Given $L\in 2^\N \cup\{\infty\}$,
\begin{align}\label{CL1}
\E\Biggl[  \,\prod_{j=1}^{n}f_{j}\bigl(q_L(x_j)\bigr)  \Bigg| q_L(a),q_L(b)\Biggr]
	& = \E\Biggl[  \,\prod_{j=1}^{n}f_{j}\bigl(q_L(x_j)\bigr)  \Bigg| q_L(a),q_L(b)\Biggr] \notag \\
	& \qquad\qquad  \times \E\Biggl[  \,\prod_{j=n+1}^{n+m}f_{j}\bigl(q_L(x_j)\bigr)  \Bigg| q_L(a),q_L(b) \Biggr]
\end{align}
for any $a< x_1<\cdots<x_n <  b < x_{n+1}<\cdots<x_{n+m} < a+2L$.  Moreover, 
\begin{align}\label{CL2}
\E\Biggl[  \,\prod_{j=1}^{n} & f_{j}\bigl(q_L(x_j)\bigr)  \Bigg| q_L(a),q_L(b)\Biggr] \\
	& = \int\!\cdots\!\int P\Bigl(\,\begin{matrix} y_1, \ldots, y_n \\ x_1, \ldots, x_n \end{matrix}\Big| \begin{matrix} q_L(a),&\!\!\!\!q_L(b) \\ a,&\!\!\!\!b \end{matrix} \Bigr) \prod_{j=1}^{n}f_{j}\bigl(y_j\bigr)  \,dy_1\cdots dy_n .\notag
\end{align}
\end{lemma}

\begin{proof}
The defining property of conditional expectation is the law of iterated expectation.  It is elementary to verify that the formulas presented satisfy this axiom.  
\end{proof}

The description of the conditional distribution by $P$ is not entirely unique --- the defining identity would still be obeyed if $P$ were modified on a Lebesgue null set of its arguments.  Nevertheless, the representative presented is the unique continuous one.  The identity \eqref{CL1} has a simple interpretation: the restrictions of $q_L$ to the two arcs joining $a$ and $b$ in $\R/2L\Z$ are conditionally independent given the values of $q_L$ at the two ends.

A more complete description of the infinite volume case is this: the restrictions of $q$ to the three intervals $(-\infty,a]$, $[a,b]$, and $[b,\infty)$ are conditionally independent given $q(a)$ and $q(b)$; moreover the law on the middle interval is described by \eqref{CL2}.

As a single point suffices to disconnect the line, one may consider a single observation point $b$ in the infinite-volume setting.  The analogue of \eqref{CL1} is simply the Markov property: the restrictions of $q$ to $(-\infty,b]$ and $[b,\infty)$ are conditionally independent given $q(b)$.

Nevertheless, there is a good reason to consider the process with two observation points, even in the line setting, namely, the DLR condition.  This condition was introduced to define the notion of Gibbs state in infinite-volume lattice models when the underlying Hamiltonian involves finite-range interactions.  For such models, finite volume Gibbs states can be constructed without difficulty; likewise, Gibbs states on a finite region with prescribed values on the boundary are easily defined.  The DLR condition is this: an infinite-volume measure is a Gibbs measure if and only if the finite-volume conditional laws conditioned on the boundary data are correct.   Lemma~\ref{L:condit} already verifies one facet of this by showing that the finite- and infinite-volume measures have identical conditional laws.  The question of the correctness of these conditional laws will be addressed by Proposition~\ref{PROP:DLRKMS}, where we demonstrate a version of the KMS condition.  For this purpose, we will need the following integration by parts identity:

\begin{lemma}\label{L:IBP}
Fix $L\in 2^\N \cup\{\infty\}$ and $a <  b < a+2L$.  Suppose $\phi,\psi \in C^\infty_c(\R)$ are both supported within $(a,b)$, then
\begin{align*}
\E\biggl[ &\biggl( i{\int} \phi(x) \psi(x) \,dx \biggr) \cdot \exp\biggl( i{\int} q(x) \psi(x) \,dx \biggr) \bigg| q(a),q(b)\biggr] \\
	&= \E\biggl[ \biggl( \int -\phi''(x) q(x) + \phi(x)V'(q(x)) \,dx \biggr)\cdot 
		\exp\biggl( i{\int} q(x) \psi(x) \,dx \biggr) \bigg| q(a),q(b)\biggr] .
\end{align*}
where $V$ is as in \eqref{HGibbs}.
\end{lemma}

\begin{proof}
Exploiting the operator identities
\begin{align*}
y\LL - \LL y = \tfrac{\partial\ }{\partial y} \qtq{and} y \LL^2 + \LL^2 y - 2 \LL y \LL = V'(y),
\end{align*}
and integrating by parts, we obtain
\begin{align}\label{cdc?}
\int_s^t e^{-(x-s)\LL} & [-\phi''(x) y + \phi(x)V'(y)] e^{-(t-x)\LL} \,dx \\
&= e^{-(t-s)\LL} \bigl[\phi(t)\tfrac{\partial\ }{\partial y}  - \phi'(t) y \bigr]
	- \bigl[\phi(s)\tfrac{\partial\ }{\partial y} - \phi'(s) y \bigr] e^{-(t-s)\LL} \notag.
\end{align}
Note here $y$ and $V'(y)$ denote multiplication operators: $y:f(y)\mapsto yf(y)$.

Combining this with \eqref{CL2} and \eqref{CondFM}, we find
\begin{align}\label{CL3}
& \E\Biggl[  \biggl( \int -\phi''(x) q(x) + \phi(x)V'(q(x)) \,dx \biggr) \cdot \prod_{j=1}^{n} e^{i\xi_j q(x_j)}  \Bigg| q(a),q(b)\Biggr] \\
	{} = {}& \E\biggl[  \biggl( i \sum_{j}\xi_j \phi(x_j)\biggr) \cdot \prod_{j=1}^{n} e^{i\xi_j q(x_j)}  \Bigg| q(a),q(b)\Biggr] \notag
\end{align}
for any choices of $\xi_j\in\R$ and $a=x_0<x_1<\cdots<x_n<x_{n+1}=b$.  To see this, one first breaks the interval $[a,b]$ into pieces $[x_j,x_{j+1}]$ and applies \eqref{cdc?} on each such smaller interval.

Choosing $\xi_j$ wisely we obtain a close variant of the sought-after identity.  The only difference is that all integrals involving $\psi(x)$ are replaced by Riemann sums.  A simple approximation argument using \eqref{Wsp moments} completes the proof.
\end{proof}

\subsection{The KMS condition}\label{SS:KMS}
Let us quickly motivate the KMS condition by considering finite dimensional Hamiltonian mechanics with position and momentum coordinates $\vec  x,\vec p \in\R^n$ and the traditional Poisson structure
\begin{align*}
\{G,F\} = \sum \frac{\partial G}{\partial x_i}\frac{\partial F}{\partial p_i} - \frac{\partial G}{\partial p_i} \frac{\partial F}{\partial x_i} .\end{align*}

For any pair of smooth functions $F,G\in C^\infty(\R^n\times\R^n)$, with $G$ being of compact support, integration by parts shows
\begin{align*}
\int \{G,F\} \,dx\,dp = 0.
\end{align*}
This expresses the divergence-free nature of the Hamiltonian vector field generated by $F$.  Inserting the Gibbs weight $e^{-\beta E}$ into this calculation, we find
\begin{align*}
 0 = \int \{F,G e^{-\beta E}\}  \,dx\,dp = \int \{ F,G \} e^{-\beta E} - \beta \{ F, E\}  G e^{-\beta E} \,dx\,dp 
\end{align*}
This leads us to the \emph{classical KMS condition}  for a measure to be a Gibbs measure for the free energy $E$ at inverse temperature $\beta$:
\begin{align}\label{naive KMS}
\E\bigl[ \{F,G\} \bigr]  = \beta \,\E\bigl[ \{ F, E \} G \bigr] .
\end{align}

In general, the free energy $E$ is not simply the Hamiltonian $H$ but a linear combination of (typically conserved) macroscopic physical quantities.  If $H$ and $E$ do indeed Poisson commute, then the KMS condition guarantees that
\begin{align*}
 \E\bigl[ \{H,G\} \bigr]  = \beta \,\E\bigl[ \{ H, E \} G \bigr] = 0  \qtq{for all} G\in C^\infty_c(\R^n\times\R^n).
\end{align*}
Thus KMS states are dynamical invariant whenever $\{ H, E \}=0$.  

We caution the reader that the preceding calculations took place in a finite-dimensional paradise, devoid of analytic subtleties.  Serious obstacles must be overcome to obtain dynamical invariance for interesting infinite-dimensional systems.

Let us consider what the KMS condition ought to say for \eqref{formal Gibbs} with $\beta=1$.  As testing functions, we employ
\begin{align}\label{FG chars}
F(q) = \exp\biggl( i{\int} q(x) \phi(x) \,dx \biggr)\qtq{and} G(q)=\exp\biggl( i{\int} q(x) \psi(x) \,dx \biggr)
\end{align}
with $\phi$ and $\psi$ belonging to $C^\infty_c(\R)$.  With these choices, 
\begin{align}\label{FG char PBs}
\{ F, G\}(q) &= - F(q) G(q) \int \phi(x)\psi'(x) \,dx,\\
\label{FE char PBs}
\{ F, H_\mKdV - \mu M\}(q)  &=  i F(q) \int \phi(x) \partial_x [-q''(x) + 2q(x)^3 - \mu q(x)] \,dx \\
&=  - i F(q) \int - \phi'''(x) q(x) + \phi'(x) V'(q(x)) \,dx . \notag
\end{align}
Combining these Poisson brackets with \eqref{naive KMS}, we are lead to the following form of KMS condition for \eqref{mkdv}.  It holds with any prescribed boundary conditions and so is stronger condition than the traditional KMS condition.  It arises naturally as a synthesis of the KMS and DLR conditions.
 
\begin{proposition}[Joint KMS/DLR condition] \label{PROP:DLRKMS} 
Fix $L\in 2^\N \cup\{\infty\}$.  The Gibbs distribution defined in Proposition~\ref{PROP:Gibbs} satisfies the following: Given $a <  b < a+2L$ and $\phi,\psi\in C^\infty_c(\R)$ that are supported within $(a,b)$, the functions
$F$ and $G$ defined in \eqref{FG chars} satisfy
\begin{align}\label{E:DLRKMS}
\E\Bigl[ - \langle &\phi,\psi'\rangle F(q) G(q)  \Big| q(a),q(b)\Bigr] \\
	&= \E\biggl[ \biggl( -i \int -\phi'''(x) q(x) + \phi'(x) V'(q(x)) \,dx \biggr)\cdot F(q) G(q) \bigg| q(a),q(b)\biggr] \notag
\end{align}
\end{proposition}

\begin{proof}
One simply applies Lemma~\ref{L:IBP} using the functions $\phi'$ and $\phi+\psi$.  The resulting identity reduces to that written above because $\int \phi' \phi\,dx = 0$.
\end{proof}

Taking the (unconditional) expectation of both sides of \eqref{E:DLRKMS} shows that our Gibbs distributions satisfy the traditional KMS condition.  

The specific test functions $F(q)$ and $G(q)$ permitted  in Proposition~\ref{PROP:DLRKMS} are not well suited to proving dynamical invariance of these Gibbs measures.  Our next result, Theorem~\ref{T:KMS}, addresses this.  We will restrict attention to the finite volume setting because that is what we need.  

Given a function $F: L^2(\T_L)\to \R$ and a frequency cutoff $N\in 2^\N$, we define a corresponding frequency localized version
\begin{align}\label{F_N on L^2}
F_N(q) := F(q_{\leq N}), \qtq{for which} \tfrac{\delta F_N}{\delta q} = P_{\leq N} \tfrac{\delta F}{\delta q}(q_{\leq N}).
\end{align}
Notice that if $F$ is $C^1$ and satisfies 
\begin{align}\label{C^1 on L^2}
\sup_q \Bigl[ |F(q)| + \bigl\| \tfrac{\delta F}{\delta q} \bigr\|_{L^2} \Bigr] < \infty,
\end{align}
then $\tfrac{\delta F_N}{\delta q}$ is uniformly bounded in $H^{100}(\T_L)$.  In particular, for any two functions $F$ and $G$ satisfying \eqref{C^1 on L^2}, the Poisson bracket of $F_N$ and $G_N$ is well defined.

\begin{theorem}[KMS condition]\label{T:KMS} 
Fix $L\in 2^\N$ and $N\in 2^\N$.  For any pair of $C^1$ functions $F,G: L^2(\T_L)\to \R$ satisfying \eqref{C^1 on L^2},
\begin{align}\label{E:C^1 KMS}
\E\Bigl[ \{ F_N, G_N\}  \Bigr] = \E\Bigl[  - \Bigl\langle \partial_x \tfrac{\delta F_N}{\delta q},\; -q''+V'(q) \Bigr\rangle G_N  \Bigr]  .
\end{align}
\end{theorem}

\begin{proof}
At this moment, we only know that 
\begin{align}\label{E:DLRKMS'}
\E\Bigl[ \langle\phi,\psi'\rangle F(q) G(q)  \Bigr] 
	&= \E\biggl[ i F(q) G(q) \int -\phi'''(x) q(x) + \phi'(x) V'(q(x)) \,dx \biggr] 
\end{align}
when $F$ and $G$ take the form \eqref{FG chars} with $\phi,\psi\in C^\infty(\T_L)$ both vanishing in the same small arc.  The arc in question contains the conditioning points $a$ and $b$ employed in Proposition~\ref{PROP:DLRKMS}.

The restriction that $\psi$ vanish is easily removed by approximation employing \eqref{Wsp moments}.  Such an argument does not extend easily to $\phi$ due to the presence of terms such as $\int \phi''' q \,dx$.

Here is the simple remedy:  First decompose $\phi=\phi_1+\phi_2$ such that each summand vanishes on some arc.  Now use \eqref{E:DLRKMS'} not with the original $F$ and $G$ but rather on each summand in the representation 
$$
F(q) G(q) = e^{i\int\phi_1 q} \cdot e^{i\int(\psi+\phi_2) q} + e^{i\int\phi_2 q} \cdot e^{i\int(\psi+\phi_1) q} .
$$

In this way, we deduce that \eqref{E:DLRKMS'} and so also \eqref{E:C^1 KMS} hold when $F$ and $G$ take the form \eqref{FG chars} with general $\phi,\psi\in C^\infty(\T_L)$.  Specializing to $\phi$ and $\psi$ being trigonometric polynomials, it follows that \eqref{E:C^1 KMS} holds whenever $F$ and $G$ are complex exponentials built from the Fourier coefficients of $q_{\leq N}$.

On the other hand,  $F_N=F(q_{\leq N})$ is ultimately a $C^1$ function of (finitely many) Fourier coefficients of $q_N$.  Thus it may be approximated by finite sums of complex exponentials of the form just described.  In fact, given any $C^1$ function $f:\R^d\to\R$ for which $f$ and $\nabla f$ are uniformly bounded, there is a sequence of trigonometric polynomials $f_n$ with the properties that
$$
\sup_{n,x} |f_n(x)| + |\nabla f_n(x)| <\infty \qtq{and} \bigl|f(x)-f_n(x)\bigr| + \bigl|\nabla f(x)-\nabla f_n(x)\bigr| \longrightarrow 0 
$$
uniformly on compact sets.  This is readily shown, for example, by truncation, periodization, and Fej\'er summation.  This quality of approximation suffices to deduce \eqref{E:C^1 KMS} in full generality.
\end{proof}

\subsection{Coupling Gibbs measures}\label{ss2.2}

It is not difficult to verify that finite-volume Gibbs distributions converge to that in infinite volume as $L\to\infty$.  Once one has convergence in distribution, the general Skorohod coupling allows one to realize all these Gibbs processes on a single probability space in such a way that samples from the finite-volume Gibbs laws converge almost surely to those of the infinite volume Gibbs state.  (Here convergence means uniformly on compact sets.)   Such a coupling is evidently necessary if we wish to transfer dynamical statements from finite to infinite volume.

Unfortunately, such generalities do not suffice to overcome the strong transportation present in the flows we study.  We need the following very strong quantitative version:

\begin{proposition}\label{P:coupled q_L}
There is a single probability space supporting Gibbs distributed random processes $q_L$ for all $L\in 2^\N\cup\{\infty\}$ with the following additional property: 
\begin{align}\label{E:coupled q_L}
\PP\Bigl( q_L(x) \not\equiv q_\infty(x) \text{ on the interval } \bigl[-\tfrac L2,\tfrac L2 \bigr] \Bigr) \lesssim e^{-c L}
\end{align}
uniformly in $L$ for some fixed $c>0$.
\end{proposition}

While we do not truly need an exponential rate of convergence, no polynomial rate would be fast enough; see, for example, the proof of Lemma~\ref{L:1603}.  We begin with a well-known coupling technique:

\begin{lemma}\label{R2 couple}
Let $f_X,f_Y \in L^1(\R^2)$ be probability density functions and suppose that $X:\Omega\to\R^2$ is a random variable with law $f_X$.  Then there is a product probability space $\Omega\times\Omega'$ endowed with product measure and a random variable $Y:\Omega\times\Omega'\to \R^2$ with law $f_Y$ so that
\begin{align}\label{almost same}
\PP ( X\neq Y) =  \tfrac12 \| f_X -f_Y \|_{L^1}  
\end{align}
and so that $Y$ is conditionally independent of every $Z:\Omega\to\R$ given $X$.
\end{lemma}

\begin{proof}
If RHS\eqref{almost same} vanishes then we may take $Y\equiv X$.  If $f_X$ and $f_Y$ are unequal, then we augment the probability space to include random variables $U$ and $W$ that are statistically independent of one another and of any random variable defined over $\Omega$.  We choose $U$ to be uniformly distributed over $[0,1]$ and $W$ to be $\R^2$-valued with probability density
$$
f_W(x) = 2  \|f_X-f_Y\|_{L^1}^{-1} \bigl[ f_Y(x) - f_X(x) \wedge f_Y(x) \bigr],
$$
where the wedge symbol denotes minimum.  We then define
$$
Y = \begin{cases} X &: U \cdot f_X(X) \leq f_X(X)\wedge f_Y(X) \\ W &: \text{otherwise.}\end{cases}
$$
The way that $U$ and $W$ are defined ensures that $Y$ is conditionally independent of every $Z:\Omega\to\R$ given $X$.

As $X$ and $W$ are independent and have absolutely continuous laws,
\begin{equation*}
\PP \bigl( X = Y\bigr) = \PP \bigl( U \cdot f_X(X) \leq f_X(X)\wedge f_Y(X) \bigr)  =  \int f_X(x) \wedge f_Y(x)\,dx
\end{equation*}
from which we deduce that
\begin{equation*}
\PP \bigl( X \neq Y\bigr) = \tfrac12 \int \bigl [ f_X(x) + f_Y(x) - 2 f_X(x) \wedge f_Y(x)\bigr]\,dx = \tfrac12 \| f_X -f_Y \|_{L^1}.
\end{equation*}

Lastly, dividing the expectation according to whether or not $X=Y$, we have
$$
\E\{ F(Y) \} = \int F(x) f_X(x)\wedge f_Y(x)\,dx + \int F(x) \bigl[ f_Y(x) - f_X(x) \wedge f_Y(x) \bigr]\,dx
$$
for any $F\in L^\infty(\R^2)$ and so $Y$ does indeed have law $f_Y$.
\end{proof}

Evidently, this argument would apply equally well to random vectors in $\R^d$ for any finite $d$.  However, we are specifically interested in the two-dimensional case because we wish to couple
\begin{align}\label{XY at L}
X_L = \begin{bmatrix} q_\infty(+L/2) \\ q_\infty(-L/2)\end{bmatrix} \qtq{and} Y_L =  \begin{bmatrix} \tilde q_L(+L/2) \\ \tilde q_L(-L/2)\end{bmatrix},
\end{align}
where $\tilde q_L$ and $q_\infty$ are Gibbs distributed and $L \in 2^\N$.  Although $\tilde q_L$ introduced here and $q_L$ appearing in Proposition~\ref{P:coupled q_L} have the same law, it will be notationally convenient for what follows to distinguish the two. 

\begin{lemma}\label{L:XY coup}
Given $L\in 2^\N$, let us write $f_{X_L}$ and $f_{Y_L}$ for the probability density functions of $X_L$ and $Y_L$ defined in \eqref{XY at L}.  Then 
\begin{align}\label{E:XY coup}
\| f_{X_L} - f_{Y_L} \|_{L^1(\R^2)} \lesssim e^{-\eps L},
\end{align}
for a universal $\eps>0$ (which may be chosen as in Proposition~\ref{P:heat}).
\end{lemma}

\begin{proof}
From the definition of the Gibbs processes, 
$$
[f_{Y_L} - f_{X_L}](y,y')=\Bigl[ \tfrac{1}{Z_L} e^{-L \LL}(y,y') - \psi_0(y)\psi_0(y') \Bigr]e^{-L \LL}(y',y).
$$
Here $Z_L = \tr( e^{-2L \LL})$ while $\psi_0$ is $L^2$-normalized and satisfies $\LL \psi_0 = 0$.

From Proposition~\ref{P:heat} and Cauchy--Schwartz, we have
\begin{align}\label{coup1}
\iint  e^{-L\LL}(y',y) \bigl| e^{-L\LL}(y,y') - \psi_0(y)\psi_0(y')\bigr| \,dy\,dy' \lesssim  e^{-\eps L}
\end{align}
uniformly for $L\geq 1$.  This leaves us to bound $|Z_L^{-1} -1|$.  

The spectral theory of $\LL$ shows  $Z_L = \sum e^{-2L\lambda_k} \geq 1$.  From this, we find
\begin{align*}
0 \leq 1 - Z_L^{-1} \leq Z_L - 1  = \iint  e^{-L\LL}(y',y) \bigl[ e^{-L\LL}(y,y') - \psi_0(y)\psi_0(y')\bigr] \,dy\,dy'  \lesssim  e^{-\eps L}.
\end{align*}
The last step is just \eqref{coup1}.  The conclusion \eqref{E:XY coup} now follows easily.
\end{proof}

By combining Lemmas~\ref{L:condit} and~\ref{L:XY coup}, we may deduce that as $L\to\infty$, all finite marginals of $\tilde q_L$ converge to those of $q_\infty$.  Concretely, one chooses $a=-L/2$ and $b=L/2$, then exploits the fact that the conditional marginals are identical.  This leads us to the means of coupling the processes: First couple $X_L$ to $Y_L$ exactly as in Lemma~\ref{L:XY coup}.  Because the conditional marginals are identical, we can actually make $q_L$ and $q_\infty$ identical on the interval $[-L/2,L/2]$ whenever $X_L=Y_L$.

Before we present the details, let us quickly review the important concept of extension by conditioning.

Concomitant with the notion of conditional expectation is that of disintegration.   The Disintegration Theorem guarantees that there is a map $y\mapsto \mu_y$ of points in $\R^2$ to Borel probability measures on $C(\T_L)$ with the property that
$$
\E\bigl[ F(\tilde q_L) \bigr] = \iint F\,d\mu_y f_{Y_L}(y)\,dy.
$$
The map $y\mapsto \mu_y$ is known as a probability kernel.

More generally, given a probability space $(\Omega_0,\PP_0)$ together with a random variable $Y:\Omega_0\to\R^2$ that has the same law as $Y_L$, we may define a probability kernel on $\Omega_0$ via $\omega\mapsto\nu_\omega:=\mu_{Y(\omega)}$.

This second probability kernel $\nu$ allows us to define an augmented probability space $\Omega_0\times C(\T_L)$ with expectation
$$
\E \bigl[ F(\omega,q) \bigr] = \iint F(\omega,q) \,d\nu_\omega(q) \,d\PP_0(\omega).
$$
By construction, the random variable $q$ is Gibbs distributed and it is conditionally independent of any random variable $Z:\Omega_0\to\R$ given $Y$.

In this example, we performed a single extension of a given probability space.  We would like to do this infinitely many times, once for each $L\in 2^\N$.  The fact that this is possible is know as the Ionescu-Tulcea Theorem; see, for example, \cite[\S8]{MR4226142}.

\begin{proof}[Proof of Proposition~\ref{P:coupled q_L}]Let us begin with a realization of $q_\infty$, which is Gibbs distributed on $\R$.  This probability space also realizes all random variables $X_L$  defined by \eqref{XY at L}.

By repeatedly applying Lemmas~\ref{R2 couple} and~\ref{L:XY coup}, it is possible to augment our original probability space to realize new random variables $Y_L$ whose laws match those given by \eqref{XY at L} and also satisfy 
$$
\PP ( X_L\neq Y_L)  \lesssim e^{-\eps L}
$$
uniformly in $L$.  Note that $q_\infty$ and $Y_L$ are conditionally independent given $X_L$.

Strictly speaking, the argument just presented applies only to finite collections of $L\in 2^\N$; however, this demonstrates the consistency of these joint distributions and so guarantees that one may realize the $Y_L$ for \emph{all} $L\in 2^\N$.

Combining the Disintegration and Ionescu-Tulcea Theorems, we may now realize processes $\tilde q_L$ that are Gibbs distributed on $\T_L$ for all $L\in 2^\N$ and satisfy the second identity in \eqref{XY at L}.

Our construction ensures that $q_\infty$ and $\tilde q_L$ are conditionally independent given $X_L$.  This is important because together with Lemma~\ref{L:condit} it guarantees that the four processes
\begin{align}\label{four processes}
q_\infty\big|_{[-L/2,L/2]},\quad q_\infty\big|_{[-L/2,L/2]^c}, \quad
	\tilde q_L\big|_{[-L/2,L/2]},\qtq{and} \tilde q_L\big|_{[L/2,3L/2]} 
\end{align}
are conditionally independent given $X_L$ and $Y_L$.  

We are now ready to define the random variables $q_L$: For $x\in[-L,L]$,
\begin{align}
q_L(x) = \begin{cases} q_{\infty} (x) &: x \in [-L/2,L/2]  \text{ and } Y_{L}=X_{L},\\
	\tilde q_L(x) &: \text{ otherwise.}\end{cases}
\end{align}
We continue this definition $2L$-periodically throughout $\R$.

It is clear that \eqref{E:coupled q_L} holds.  To verify that $q_L$ is indeed Gibbs distributed,  we just need to show that it has the same distribution as $\tilde q_L$.  This in turn follows from two observations: First, the values of $q_L$ on the complementary arcs $[-L/2,L/2]$ and $[L/2,3L/2]$ are conditionally independent given $X_L$ and $Y_L$.  Secondly, restricting attention to the arc $[-L/2,L/2]$, the conditional distributions of $q_L$ and $\tilde q_L$ given $X_L$ and $Y_L$ are identical.  This is a consequence of Lemma~\ref{L:condit}.
\end{proof}

\subsection{Mixing properties}\label{SS:mix}
It is well-known that the the whole-line Gibbs state is not only ergodic under translations, but even (strong) mixing.  Much later in the paper, we shall rely on ergodicity in order to understand the action of the Miura map; moreover, in Corollary~\ref{C:diff} we will use it to show that different chemical potentials and/or temperatures lead to mutually singular Gibbs measures.  Rather than settle for ergodicity, we will prove mixing because the following lemma, inspired by \cite{RosenSimon}, allows us to give an efficient treatment based on the preliminaries we have already presented.

\begin{lemma}\label{L:expo mix} Let $q$ be Gibbs distributed on the whole line and fix points $a<b$ in $\R$.  Given random variables $F,G \in L^2(d\PP)$ that are measurable with respect to the sigma algebras generated by $\{q(x):x\leq a\}$ and $\{q(x):x\geq b\}$, respectively, then
\begin{align}\label{expo mix}
\bigl| \E(FG) - \E(F)\E(G)\bigr| \leq e^{-\lambda_1(b-a)} \sqrt{ \E(F^2)\E(G^2) } \,.
\end{align}
\end{lemma}

\begin{proof}
By the Markov property, we may replace $F$ by $\E(F|q(a))$ and $G$ by $\E(G|q(b))$ without affecting the LHS\eqref{expo mix}.  This change may affect the RHS\eqref{expo mix}, but only by making it smaller.  These observations reduce the question to consideration of random variables $F=f(q(a))$ and $G=g(q(b))$.

For such random variables, the joint marginals \eqref{FKmuL2} show
\begin{align*}
\text{LHS\eqref{expo mix}} = \biggl|\iint f(y)\psi_0(y) \bigl[ e^{-(b-a)\LL}(y,y') - \psi_0(y)\psi_0(y')\bigr]\psi_0(y')g(y')\,dy\,dy'\biggr|.
\end{align*}
The spectral theory of the operator $\LL$ discussed earlier now guarantees that the quantity in square brackets is the integral kernel of an operator of norm $e^{-(b-a)\lambda_1}$ and consequently,
\begin{equation*}
\text{LHS\eqref{expo mix}} \leq e^{-\lambda_1(b-a)} \| f \psi_0\|_{L^2(\R)} \| g \psi_0\|_{L^2(\R)} = \text{RHS\eqref{expo mix}}.
	\qedhere
\end{equation*}
\end{proof}

This lemma already captures the intuitive notion of mixing that when observed on widely separated regions of space, a single sample from the Gibbs state yields seemingly independent observations.  The true definition of mixing, however, requires us to consider more general pairs of measurable events.  Recall that we are employing the Borel $\sigma$-algebra associated with the compact-open topology, which is also the $\sigma$-algebra generated by cylinder sets.

\begin{corollary}\label{C:mixing}
The infinite-volume Gibbs state is mixing under translations.
\end{corollary}

\begin{proof}
Suppose first that $A$ and $B$ are cylinder subsets of $C(\R)$, which means that they are determined by the value of the function at finitely many points.  Then Lemma~\ref{L:expo mix} shows that
\begin{equation}\label{ABmix}
\PP\Bigl( q(\cdot) \in A \ \&\ q(\cdot-x) \in B\Bigr) \longrightarrow 
	\PP\bigl( q(\cdot) \in A\bigr) \PP\bigl( q(\cdot) \in B\bigr)  \qtq{as} |x|\to\infty.
\end{equation}

Observe now that for fixed cylinder set $B$, the collection of sets $A$ for which \eqref{ABmix} holds is a $\lambda$~system which contains the cylinder sets (a $\pi$~system).  Thus by the $\pi$-$\lambda$~Theorem, convergence holds for all measurable $A$.  Likewise, we see that for each measurable $A$, convergence holds not only for cylinder sets $B$ but for all measurable $B$.  This shows that $q$ is mixing.
\end{proof}

Recall that mixing implies ergodicity.  Next we use this to show that no distinct pairs of temperature and chemical potential lead to absolutely continuous Gibbs states; rather, they are mutually singular.  Recall that we previously used the scaling \eqref{scaling} to normalize the temperature and so simplify the presentation.  From this perspective, we see that changes in chemical potential are not be reversed by rescaling.

\begin{corollary}\label{C:diff}
The Gibbs states corresponding to different chemical potentials and/or temperatures are mutually singular.
\end{corollary}

\begin{proof}
From our previous discussion, we know that the stationary distribution for the infinite-volume Gibbs state at inverse temperature $\beta$ and chemical potential $\mu$ is given by the measure $\psi(y)^2\,dy$ where $\psi$ is the $L^2$-normalized ground state of the Schr\"odinger operator
$$
-\tfrac12 \partial_y^2 + \tfrac12 \beta^2 y^4 - \tfrac\mu2 y^2 .
$$
As no two such ground states agree, nor do any two stationary distributions.

Given pairs $(\beta_1,\mu_1)$ and $(\beta_2,\mu_2)$ that are not identical, we may choose a function $f\in C^\infty_c(\R)$ that witnesses the distinction between the two stationary distributions in the following way:
$$
\E_{\beta_1,\mu_1} \bigl\{ f\bigl(q(0)\bigr)\bigr\} < 0 < \E_{\beta_2,\mu_2} \bigl\{ f\bigl(q(0)\bigr)\bigr\}.
$$
By the ergodic theorem, it follows that the  Gibbs$(\beta_1,\mu_1)$ measure gives full weight to the event
$$
\liminf_{L\to\infty} \frac1L \int_0^L f\bigl(q(x)\bigr)\,dx < 0
$$
and that this event is a Gibbs$(\beta_2,\mu_2)$ null set.  
\end{proof}

\subsection{The infinite-volume Gibbs state as a diffusion}\label{SS:diffusion}
In the introduction, we mentioned the connection between the infinite-volume Gibbs state and solutions of the stochastic differential equation \eqref{E:Ito}.  This is quite a departure from the manner in which we defined Gibbs measure.  In this short subsection, we explain this connection using the martingale characterization of diffusions championed in \cite{MR2190038}.  This characterization will help us better describe the regularity properties both of the Gibbs samples $q(x)$ and of their image under the Miura map. 

In the language of diffusions, \eqref{E:Ito} says that $q(x)$ has diffusivity $a(y)\equiv 1$ and drift $b(y)=\psi_0'(y)/\psi_0(y)$.  As $\psi_0$ is the ground state of Schr\"odinger operator \eqref{HGibbs}, we see that $b(y)$ is smooth and odd; moreover, $b(y)= -y^2 + \tfrac\mu2 + o(1)$ as $y\to+\infty$.  Such a diffusion is characterized by the fact that for every $f\in C^\infty_c(\R)$, the process
\begin{align}\label{SVtest}
f\bigl(q(x_2)\bigr) - f\bigl(q(x_0)\bigr) - \int_{x_0}^{x_2} \tfrac12 f''\bigl(q(x)\bigr)  + b\bigl(q(x)\bigr)f'\bigl(q(x)\bigr)\,dx
\end{align}
indexed by $x_2$ is a martingale for $x_2>x_0$.

As we already know that $q(x)$ is a Markov process, it is enough to verify the following:

\begin{lemma} If $q$ is distributed according to the infinite-volume Gibbs law, then 
\begin{align}\label{SVtest'}
\E\biggl\{ f\bigl(q(x_2)\bigr) - \int_{x_1}^{x_2} \tfrac12 f''\bigl(q(x)\bigr)  + b\bigl(q(x)\bigr)f'\bigl(q(x)\bigr)\,dx \bigg| q(x_1) \biggr\}
	=  f\bigl(q(x_1)\bigr) 
\end{align}
for all $f\in C^\infty_c(\R)$ and all pairs $x_2>x_1$.
\end{lemma}

\begin{proof}
We must show that
\begin{align}\label{SVtest''}
\E\biggl\{ \biggl[f\bigl(q(x_2)\bigr) - \int_{x_1}^{x_2} \bigl[\tfrac12 f'' + bf'\bigr]\bigl(q(x)\bigr)\,dx\biggr]
g\bigl(q(x_1)\bigr) \biggr\}
	=  \E\Bigl\{ \bigl[fg\bigr]\bigl(q(x_1)\bigr) \Bigr\}
\end{align}
for any $g\in C^\infty_c(\R)$.  This is a question directly accessible to analysis via \eqref{FKmuL2}.  The key observation is that
$\LL\psi_0=0$ implies that
$$
\bigl[\tfrac12 f'' + bf'\bigr](y) \cdot \psi_0(y) = - \LL [f\psi_0] (y).
$$
From this, we see that
\begin{align*}
\E\Bigl\{ \bigl[\tfrac12 f'' + bf'\bigr]\bigl(q(x)\bigr) \cdot g\bigl(q(x_1)\bigr) \Bigr\}
	&=  - \bigl\langle \LL(f \psi_0), \, e^{-(x-x_1)\LL} g\psi_0\bigr\rangle \\
&= \tfrac{d}{dx} \E\Bigl\{ f(q(x))\cdot g\bigl(q(x_1)\bigr) \Bigr\}
\end{align*}
and so \eqref{SVtest''} follows from the Fundamental Theorem of Calculus.
\end{proof}

We may now bring this part of the story to its natural conclusion:

\begin{corollary}\label{C:BM}
If $q$ is distributed according to the infinite-volume Gibbs law, then the process
\begin{align}\label{aBM}
B(x) := q(x) - q(0) - \int_0^x \tfrac{\psi_0'}{\psi_0}\bigl(q(x')\bigr)\,dx'  
\end{align}
is a Brownian motion and $q(x)$ solves \eqref{E:Ito} because
\begin{align}\label{aBM'}
q(x) = q(0) + B(x) +  \int_0^x \tfrac{\psi_0'}{\psi_0}\bigl(q(x')\bigr)\,dx'   .
\end{align}
\end{corollary}

\begin{proof}
In view of the bounds \eqref{max on I}, a simple limiting argument allows us to take $f(y)=y$ and $f(y)=y^2$ in \eqref{SVtest'}.

The choice $f(y)=y$ shows that $B(x)$ is a martingale.  Combining this with the result for $f(y)=y^2$, a short computation shows that $B(x)^2-x$ is also a martingale.  Noting that $B(x)$ has continuous sample paths, it follows from L\'evy's characterization (cf. \cite[\S7.1]{MR1267569}) that $B(x)$ is a Brownian motion.
\end{proof}

It follows from \eqref{aBM'} that $q(x)$ differs from Brownian motion by a (random) $C^1$ function, which justifies the remark made in the introduction that samples $q(x)$ have the same regularity as Brownian paths.  In particular, we see that $q(x)$ is not of bounded variation on any interval because Brownian motion is not of bounded variation on any interval.

\section{Lax operators, Green's functions, and approximating flows}\label{S:3}

We begin by developing a deterministic theory of the Lax operator
\begin{align}\label{LaxAKNS}
\Lax = \Lax_0  + Q \quad\text{with}\quad \Lax_0=\begin{bmatrix}
-\partial & 0\\
0 & \partial
\end{bmatrix} \quad\text{and}\quad Q=\begin{bmatrix}
0 & q\\
-q  & 0
\end{bmatrix} 
\end{align}
where $q$ is real-valued and satisfies
\begin{equation}\label{det L hyp}
\langle x \rangle^{-b} q(x) \in L^2(\R)
\end{equation}
for some $b>0$.  This includes the two regimes of greatest interest to us, namely, periodic and whole-line Gibbs distributed $q$; indeed, \eqref{Wsp moments} shows that \eqref{det L hyp} holds (almost surely) whenever $b>\frac12$.

The natural domain for $\Lax_0$ is $H^1(\R;\C^2)$.  Moreover, it is elementary to see that this yields an anti-selfadjoint operator, which is to say, $i\Lax_0$ is selfadjoint.  We define a corresponding free resolvent via
\begin{align}
\LaxR_{0}(\kappa)= (\Lax_0 + \kappa)^{-1} =
	\begin{bmatrix} (\kappa-\partial)^{-1} & 0\\ 0 &  (\kappa+\partial)^{-1} \end{bmatrix}; \label{R0}
\end{align}
here and in all that follows, we shall only need to consider those
\begin{equation}\label{kappa care}
\kappa\in\R \qtq{that satisfy} |\kappa|\geq 1.
\end{equation}

The free resolvent admits an integral kernel, namely, the free Green's function
\begin{align}
G_{0}(x,y;\kappa)=\sign(\kappa) e^{-|\kappa| |x-y|} 
\begin{bmatrix}
\ind_{(0,\infty)} (\kappa[y-x] ) & 0 \\[0.75ex]  0 & \ind_{[0,\infty)} (\kappa[x-y] )   \end{bmatrix}. \label{G0}
\end{align}
The two indicator functions are chosen not to overlap on the null set $\{x=y\}$ so that $\tr\{ G_0(x,y;\kappa)\}\equiv \sign(\kappa)$, irrespective of $x,y\in\R$.

\begin{proposition}\label{P:Lax Op}
Fix $b>0$ and $q$ satisfying \eqref{det L hyp}. Then \eqref{LaxAKNS} defines an anti-selfadjoint operator $\Lax$ on $L^2(\R;\C^2)$ with domain
\begin{equation}\label{D(L)}
\Bigl\{ \bigl[\begin{smallmatrix}\phi\\ \psi\end{smallmatrix} \bigr]\in L^2(\R;\C^2) : \phi,\psi\in H^1_\loc(\R)
	\text{ and } q\psi - \phi', \; \psi' - q\phi \in L^2(\R) \Bigr\}.
\end{equation}
Indeed, the restriction of $\,i\Lax$ to $C^\infty_c(\R;\C^2)$ is essentially selfadjoint.

For $\kappa$ as in \eqref{kappa care}, the Green's function $G(x,y;\kappa)$, which serves as the integral kernel for the resolvent $\LaxR(\kappa)=(\Lax +\kappa)^{-1}$, has the property that
\begin{align}\label{G-G0}
G(x,y;\kappa)-G_{0}(x,y;\kappa) 
\end{align}
is a continuous function of $(x,y)\in \R^2$, real-analytic in $\kappa$,  and 
\begin{align}\label{GinLinfty}
| G(y_1,y_2;\kappa) | \lesssim \langle y_1 \rangle^b \langle y_2 \rangle^b \bigl[ 1 + \kappa^{-1} \| \langle x\rangle^{-b} q \|_{L^2}^2  \bigr] .
\end{align}
\end{proposition}

\begin{proof}
Let us write $\mathfrak l$ for the operator given by \eqref{LaxAKNS} with domain $C^\infty_c(\R;\C^2)$.  It is not difficult to verify that $\mathfrak l$ is anti-symmetric and that the domain of its adjoint $\mathfrak l^*$ is the set given in \eqref{D(L)}.   In a moment, we will show that $\mathfrak l^*$ is anti-symmetric.  By general arguments, it then follows that $\mathfrak l^{*{\mkern -1mu}*}$, which is the closure of $\mathfrak l$, coincides with $-\mathfrak l^*$.  This verifies all the claims made in the first paragraph.

To show that $\mathfrak l^*$ is anti-symmetric, let $\bigl[\begin{smallmatrix}\phi\\ \psi\end{smallmatrix} \bigr]$ and $\bigl[\begin{smallmatrix}\varphi\\ \theta\end{smallmatrix} \bigr]$ belong to the domain of $\mathfrak l^*$, which is given by \eqref{D(L)}.  For every finite interval $[a,b]$, we may apply the fundamental theorem of calculus to obtain
\begin{align*}
\int_a^b \overline{\phi}[q\theta - \varphi'] + \overline{\psi} [\theta' - q\varphi] +
	[q\overline\psi - \overline\phi']\varphi + [\overline\psi' - q\overline\phi] \theta \, dx
		= \Bigl[ \overline\psi \theta - \overline{\phi} \varphi \Bigr]_a^b .
\end{align*}
The $H^1_\loc$ regularity of the four functions $\phi,\psi,\varphi,\theta$ was essential here; the final conditions in \eqref{D(L)} guarantee that the integrand on the left-hand side is actually $L^1$.  Moreover, as all four functions belong to $L^2$, we may send $a$ and $b$ to $\pm\infty$ and see that the integral over the full line vanishes.  This precisely shows that $\mathfrak l^*$ is anti-symmetric.

We turn now to the discussion of Green's functions.  From the explicit form \eqref{G0} and the elementary inequality
$\langle x\rangle \lesssim \langle y\rangle \langle x-y\rangle$, we find that for each $b\in\R$,
\begin{align}
| G_0(x,y;\kappa) | \lesssim_b  \langle y\rangle^b \langle x\rangle^{-b} \qtq{uniformly for} x,y\in\R \qtq{and} |\kappa|\geq 1.
\end{align}
From this, we deduce easily that 
\begin{align}
\| Q\LaxR_0(\kappa) \delta_y \|_{L^2} &\lesssim_b \langle y\rangle^{b} \| \langle x\rangle^{-b} q(x) \|_{L^2} 
\end{align}
and also that the map $y\mapsto  Q\LaxR_0(\kappa) \delta_y$ is $L^2$-continuous and depends analytically on $\kappa$.  Likewise, $y\mapsto\LaxR_0(\kappa) \delta_y$ is an analytic function of $\kappa$ that is $L^2$-continuous and
\begin{align}
\| \LaxR_0(\kappa) \delta_y \|_{L^2} &\lesssim \kappa^{-1/2} . 
\end{align}
In this way, the claims made about the function \eqref{G-G0}, including the bound \eqref{GinLinfty}, follow from the resolvent expansion
\begin{equation}\label{basicRI}
\LaxR(\kappa)-\LaxR_0(\kappa) = - \LaxR_0(\kappa) Q \LaxR_0(\kappa) + \LaxR_0(\kappa)Q\LaxR(\kappa)Q\LaxR_0(\kappa).  \qedhere 
\end{equation}
\end{proof}

From \eqref{max on I}, we see that periodic Gibbs samples are almost surely bounded.  For such $q\in L^\infty(\R)$, the domain of $\Lax$ simplifies to just $H^1(\R;\C^2)$.  In the infinite-volume case, however, Gibbs samples are almost surely unbounded (this follows from ergodicity) and correspondingly, the domain given in \eqref{D(L)} genuinely depends on $q$: it is random!

Our proof that $\Lax$ is anti-selfadjoint was inspired by arguments in \cite{MR1136037}.  It did not rely on the specific growth condition \eqref{det L hyp}, but merely that $q$ is square-integrable on every finite interval.  The fact that the Dirac-type operator $i\mathbf L$ can be essentially selfadjoint without growth restrictions on the potential marks a striking departure from the case of Schr\"odinger operators.  The physical origin of this distinction was elegantly articulated in \cite{Chernoff}: the relativistic speed limit inherent to Dirac operators ensures that the quantum particle cannot reach infinity in finite time.

For much of our analysis, the range of $\kappa$ presented in \eqref{kappa care} can be reduced to simply $\kappa\geq 1$ by virtue of the following symmetry: 
\begin{equation}\label{k symmetry}
-\Lax = \Lax^* = \sigma_1 \Lax \:\! \sigma_1
	\qtq{and so} \LaxR(-\kappa) = -\LaxR(\kappa)^* = -\sigma_1 \LaxR(\kappa)\sigma_1 .
\end{equation}
Here $\sigma_1$ denotes the first of the three Pauli matrices; indeed,
\begin{equation}\label{Lax Pauli}
\Lax = iq\sigma_2 - \sigma_3 \partial , \quad
	\sigma_1 := \bigl[\begin{smallmatrix} 0&1\\1&0	\end{smallmatrix}\bigr], \quad
	\sigma_2 := \bigl[\begin{smallmatrix} 0&-i\\ i&\,0	\end{smallmatrix}\bigr],  \qtq{and} 
	\sigma_3 := \bigl[\begin{smallmatrix} 1&0\\ 0&-1	\end{smallmatrix}\bigr].
\end{equation}

Pointwise bounds on the Green's function such as \eqref{GinLinfty} can also be interpreted as (weighted) $L^1\to L^\infty$ bounds on the resolvent; such bounds will be an important part of the subsequent analysis.

Another key ingredient in our analysis is off-diagonal decay of the Green's function.  There is a robust means of obtaining such decay, namely, the Combes--Thomas argument.  For the operators we study, this argument leads us to the particularly elegant result stated next.  Notice that the estimate \eqref{CT est} holds with a constant that is independent of the potential.  This plays a major role in allowing us to avoid the full complexity of a multiscale analysis in this paper.

\begin{proposition}
Given $\kappa\geq 1$ and disjoint intervals $I,J\subset\R$, we have
\begin{align}\label{CT est}
\bigl\| \chi_I \LaxR(\kappa) \chi_J \bigr\|_{\op} \leq \tfrac{2}{\kappa} \exp\bigl\{ -\tfrac\kappa2 \dist(I,J) \bigr\} .
\end{align}
Here $\chi_I$ and $\chi_J$ denote (multiplication by) the corresponding indicator functions.  Moreover, the Green's function satisfies
\begin{align}\label{CT G off}
|G(y_1,y_2;\kappa)| \lesssim_b \langle y_1\rangle^b \langle y_2\rangle^b \bigl[ 1 + \kappa^{-1} \| \langle x\rangle^{-b} q \|_{L^2}^2  \bigr]\exp\bigl\{ -\tfrac\kappa6 |y_1-y_2| \bigr\}.
\end{align}
\end{proposition}

\begin{proof}
Direct computation shows that
\begin{align*}
e^{-\kappa x/2} \Lax e^{+ \kappa x/2} = \Lax - \tfrac12 \kappa \sigma_3
	\ \ \text{and so}\ \ e^{-\kappa x/2} \LaxR(\kappa) e^{+ \kappa x/2} = \LaxR(\kappa)  \bigl[ 1 - \tfrac{\kappa}{2}\LaxR(\kappa) \sigma_3\bigr]^{-1}.
\end{align*}
The inverse operator appearing here really does exist: $\Lax$ is anti-selfadjoint and so $\kappa \LaxR(\kappa)$ has norm bounded by one. In this way, we deduce that
\begin{align*}
\| e^{-\kappa x/2} \LaxR(\kappa) e^{+ \kappa x/2} \|_{\op} \leq \tfrac{2}{\kappa}.
\end{align*}

When $I$ is to the left of $J$, we prove \eqref{CT est} by noting that
\begin{align*}
\bigl\| \chi_I \LaxR(\kappa) \chi_J \bigr\|_{\op}
	\leq \exp\bigl\{ -\tfrac12 \kappa \dist(I,J) \bigr\} \bigl\|  e^{-\kappa x/2} \LaxR(\kappa) e^{+\kappa x/2} \bigr\|_{\op}.
\end{align*}
When $I$ and $J$ are oppositely ordered, the proof runs precisely parallel but with the conjugating weights reversed.

To verify \eqref{CT G off}, we merely need to incorporate the Combes--Thomas estimate \eqref{CT est} into our earlier proof of \eqref{GinLinfty}.  The key observations are 
\begin{align}\label{893}
\| e^{-\kappa|x-y_2|/2} \LaxR(\kappa)  e^{-\kappa|x-y_1|/2}\|_{\op} \lesssim \kappa^{-1} e^{-\kappa|y_1-y_2|/6},
\end{align}
whose proof will be given below, and the elementary bound
\begin{align*}
\| e^{\kappa|x-y|/2} Q \LaxR_0(\kappa) \delta_y\|_{L^2_x} \lesssim_b \langle y\rangle^{b}\| \langle x\rangle^{-b} q \|_{L^2} .
\end{align*}

To prove \eqref{893}, we first define three intervals
$$
I_0= \{x: |x-\tfrac{y_1+y_2}{2}\bigr| \leq \tfrac{|y_1-y_2|}{6}\bigr\}
	\qtq{and} I_\pm = \{x: \pm\bigl[x-\tfrac{y_1+y_2}{2}\bigr] > \tfrac{|y_1-y_2|}{6}\bigr\}
$$
and then insert the partition of unity $\chi_{I_+}+\chi_{I_-}+\chi_{I_0}$ on either side of $\LaxR(\kappa)$.  This yields nine terms to estimate, many of which split into two cases depending on the ordering of $y_1$ and $y_2$.  Most cases can be treated just by exploiting the exterior weights.  As an example of the remaining cases, we note that \eqref{CT est} shows
$$
\| e^{-\kappa|x-y_2|/2} \chi_{I_-} \LaxR(\kappa) \chi_{I_+} e^{-\kappa|x-y_1|/2}\|_{\op} 
	\leq \tfrac{2}{\kappa} \exp\bigl\{ -\tfrac\kappa2 \cdot \tfrac{|y_1-y_2|}{3}\bigr\} ,
$$
which we need to use when $y_1\in I_+$ and $y_2\in I_-$.
\end{proof}

\begin{corollary}
For any $\alpha\in(0,1)$ and $\kappa\geq 1$,
\begin{align}\label{multcomm}
\|e^{\langle x\rangle^\alpha}\LaxR(\kappa)e^{-\langle x\rangle^\alpha} \|_{\op} + 
	\|e^{- \langle x\rangle^\alpha}\LaxR(\kappa)e^{\langle x\rangle^\alpha} \|_{\op} \lesssim_\alpha \kappa^{-1} .
\end{align}
\end{corollary}

\begin{proof}
We focus on $e^{\langle x\rangle^\alpha}\LaxR(\kappa)e^{-\langle x\rangle^\alpha}$; the other operator may be treated in a parallel manner.

We partition $\R$ into unit-length intervals $I_n:=[n-\frac12,n+\frac12)$, $n\in\Z$, with a view to using \eqref{CT est}.  For any $f,h\in L^2(\R)$, we may then estimate
\begin{align*}
\bigl| \langle f, e^{\langle x\rangle^\alpha}\LaxR(\kappa)e^{-\langle x\rangle^\alpha} h\rangle \bigr| 
	\lesssim \sum_{n,m\in \Z} \| f \|_{L^2(I_n)} e^{|n|^\alpha - |m|^\alpha}  \bigl\| \chi_{I_n} \LaxR(\kappa) \chi_{I_m} \bigr\|_{\op} \| h \|_{L^2(I_m)},
\end{align*}
by noting that $|\langle x\rangle^\alpha - |n|^\alpha| \leq 2$ for all $x\in I_n$ and $n\in\Z$.

Next we observe that as $\alpha\in(0,1)$, one may choose $C(\alpha)<\infty$ so that
$$
\bigl| |n|^\alpha - |m|^\alpha \bigr| \leq \tfrac14 | n-m| + C(\alpha) \quad \text{for all $n,m\in\Z$.}
$$
In this way, we may apply Schur's test and \eqref{CT est} to see that
\begin{align*}
\bigl| \langle f, e^{\langle x\rangle^\alpha}\LaxR(\kappa)e^{-\langle x\rangle^\alpha} h\rangle \bigr| 
	&\lesssim_\alpha  \kappa^{-1} \sum_{n,m\in \Z}  e^{ - \frac14 | n-m|}\| f \|_{L^2(I_n)} \| h \|_{L^2(I_m)} \\
	&\lesssim_\alpha \kappa^{-1} \| f \|_{L^2} \| h \|_{L^2},
\end{align*}
which proves \eqref{multcomm}.
\end{proof}

\subsection{Diagonal Green's functions} In this subsection, we consider the behaviour of the (matrix) Green's function $G(x,y)$ at the confluence of $x$ and $y$.   We continue to consider general $q$ satisfying \eqref{det L hyp}.

From \eqref{G0} we see that even when $q\equiv 0$, the Green's function has a jump discontinuity at the diagonal; nevertheless, the continuity demonstrated in Proposition~\ref{P:Lax Op} and the diagonal character of $G_0$ allow us to define continuous real-valued functions
\begin{align}
\gamma(x;\kappa)&:=[G-G_0]_{11}(x,x;\kappa)+[G-G_0]_{22}(x,x;\kappa), \label{gamma} \\
g_-(x;\kappa) & := G_{21}(x,x;\kappa)-G_{12}(x,x;\kappa),\label{r} \\
g_+(x;\kappa)& := G_{21}(x,x;\kappa)+G_{12}(x,x;\kappa).\label{p}
\end{align}
Here, the numerical subscripts indicate matrix entries.  By \eqref{k symmetry}, these functions satisfy the following symmetries
\begin{align}\label{922}
g_-(-\kappa) = g_-(\kappa), \quad g_+(-\kappa) = - g_+(\kappa), \qtq{and} \gamma(- \kappa)&= -\gamma(\kappa).
\end{align}

At first glance, it may seem we are omitting information: the Green's matrix has four entries and we have only defined three functions.  However, by \eqref{k symmetry},
$$
[G-G_0]_{11}(x,x;\kappa) = [G-G_0]_{22}(x,x;\kappa).
$$
In fact, there is one further constraint on the diagonal Green's function, namely, \eqref{qquad}.  This and other foundational facts about our three functions $\gamma$ and $g_\pm$ form the centerpiece of this subsection. 

\begin{proposition}[Properties of $\gamma$, $g_+$, and $g_-$]\label{P:det diag}
Fix $b>0$.  For $q$ satisfying \eqref{det L hyp}, we have
\begin{align}
\| \langle x \rangle^{-4b} \gamma\|_{H^{1}_{\kappa}(\R)}
	&\lesssim_b \kappa^{-\frac{1}{2}} \|\langle x\rangle^{-b}q\|_{L^2}^{2}+\kappa^{-\frac32}\|\langle x\rangle^{-b}q\|_{L^2}^{4}, \label{gbd'} \\
\|\langle x \rangle^{-3b} g_\pm \|_{H^{1}_{\kappa}(\R)}
	&\lesssim_b \|\langle x\rangle^{-b}q\|_{L^2} +\kappa^{-1}\|\langle x\rangle^{-b}q\|_{L^2}^{3}, \label{pbd'}\\
\|\langle x \rangle^{-3b} g_-\|_{H^{2}_{\kappa}(\R)}
	&\lesssim_b \kappa\|\langle x\rangle^{-b}q\|_{L^2} + \|\langle x\rangle^{-b}q\|_{L^2}^{3}, \label{rbd'}
\end{align}
uniformly for $\kappa\geq 1$.  Moreover, for such $\kappa$, we have the identities
\begin{align}
\gamma'(\kappa)& = 2q g_+(\kappa), \label{Gderiv} \\
g_+'(\kappa)& =-2\kappa g_-(\kappa)+2 q(1+\gamma(\kappa)), \label{pderiv}\\
g_-'(\kappa)&=-2\kappa g_+(\kappa), \label{rderiv} \\
\gamma+\tfrac{1}{2}\gamma^2 &= \tfrac{1}{2} (g_+^2-g_-^2), \label{qquad}
\end{align}
and the pointwise inequalities 
\begin{align}
|g_+(x;\kappa)| < 1+\gamma(x;\kappa)  \qtq{and}  |g_-(x;\kappa)| < 1 .  \label{gammap1}
\end{align}
\end{proposition}

Note that the first inequality in \eqref{gammap1} shows which of the two possible solutions of the quadratic relation \eqref{qquad} the function $\gamma$ must be:
\begin{align}\label{qquad'}
\gamma(\kappa) &= -1 + \sqrt{1+ g_+(\kappa)^2-g_-(\kappa)^2}  .
\end{align}

Before turning to the proof of this proposition, we prove the following lemma, which will be used to obtain bounds on $\gamma$ via duality.

\begin{lemma}\label{L:bRb}
Given $b\geq 0$ and $F:\R\to \C^{2\times 2}$, we have the Hilbert--Schmidt bound
\begin{align}\label{E:bRb}
\bigl\| \langle x\rangle^{2b} \LaxR_0(\kappa) F \langle x\rangle^{-4b} \LaxR_0(\kappa) \langle x\rangle^{2b}\bigr\|_{\HS}
	\lesssim_b \kappa^{-\frac12}\| F \|_{H^{-1}_\kappa} 
\end{align}
uniformly for $|\kappa|\geq 1$.
\end{lemma}

\begin{proof}
Direct computation of the commutator of $\Lax_0$ and $\langle x\rangle^{-2b}$ reveals
\begin{align}\label{9:37}
\langle x\rangle^{-2b} \LaxR_0(\kappa) \langle x\rangle^{2b}
	= \LaxR_0(\kappa) \Bigl[1 + 2b\sigma_3 x \langle x\rangle^{-2b-2} \LaxR_0(\kappa) \langle x\rangle^{2b} \Bigr] .
\end{align}
Moreover, Schur's test and the explicit form \eqref{G0} show that the operator in square brackets is $L^2$-bounded uniformly for $|\kappa|\geq 1$.  In this way, the proof of \eqref{E:bRb} is reduced to the case where $b=0$.

When $b=0$, one may simply compute the Hilbert--Schmidt norm:
\begin{equation}
\bigl\| (\kappa\mp \partial)^{-1} f(\kappa \pm \partial)^{-1} \bigr\|_{\HS}^2
	= \frac{1}{\kappa} \int \frac{|\hat F(\xi)|^2\,d\xi}{\xi^2+4\kappa^2}. \qedhere \label{gamma2}
\end{equation}
\end{proof}

\begin{proof}[Proof of Proposition~\ref{P:det diag}]
The fact that the Green's function inverts $\mathcal L+\kappa$ implies
\begin{align}
\begin{split}
\partial_x G(x,y;\kappa)=  \begin{bmatrix} \kappa & q(x) \\ q(x) & -\kappa    \end{bmatrix} G(x,y;\kappa)
+\begin{bmatrix}
-\delta(x-y) & 0\\
0& \delta(x-y)
\end{bmatrix},\\
\partial_{y} G(x,y;\kappa)= G(x,y;\kappa) \begin{bmatrix} -\kappa & q(y) \\ q(y) & \kappa    \end{bmatrix} 
+\begin{bmatrix}
\delta(x-y) & 0\\
0& -\delta(x-y)
\end{bmatrix},
\end{split} \label{Gidents}
\end{align}
in the sense of distributions.  The identities \eqref{Gderiv}, \eqref{pderiv}, and \eqref{rderiv} follow easily.

Combining \eqref{pderiv} and \eqref{rderiv}, we find that
\begin{align}
g_- &= \tfrac{4 \kappa }{4\kappa^2 -\partial^2}( q[1+\gamma]) \quad \text{and} \quad g_+ = -\tfrac{2 \partial}{4\kappa^2 -\partial^2}( q[1+\gamma]), \label{prformula}
\end{align}
which we will use to estimate $g_\pm$ once we have estimated $\gamma$.

Given $F:\R\to\R$, we employ \eqref{basicRI} to write
\begin{align}\label{983}
\int \frac{F \gamma \,dx}{\langle x\rangle^{4b}} &= \tr \bigl\{ \langle x\rangle^{-4b} F \bigl[\LaxR-\LaxR_0\bigr]\bigr\} 
= \tr \bigl\{ \langle x\rangle^{-4b} F \LaxR_0 Q \LaxR Q \LaxR_0\bigr]\bigr\}
\end{align}
which allows us to estimate $\gamma$ by duality.  Note that the $F\LaxR_0Q\LaxR_0$ term does not contribute to the trace because it is purely off-diagonal.  

To estimate RHS\eqref{983}, we first observe that by \eqref{GinLinfty}
\begin{align}\label{1013}
\bigl\| \langle x\rangle^{-2b} Q \LaxR Q \langle x\rangle^{-2b} \bigr\|_{\HS}
	\lesssim \| \langle x\rangle^{-b} q\|_{L^2}^2 \bigl[ 1 + \kappa^{-1} \| \langle x\rangle^{-b} q \|_{L^2}^2  \bigr]  .
\end{align}
Together with Lemma~\ref{L:bRb} this allows us to deduce that
\begin{align}\label{980}
\Bigl| \tr \bigl\{ \langle x\rangle^{-4b} F \LaxR_0 Q \LaxR Q \LaxR_0 \bigr\}  \Bigr|
\lesssim \kappa^{-\frac12} \| F\|_{H^{-1}_\kappa} \| \langle x\rangle^{-b} q\|_{L^2}^2 \bigl[ 1 + \kappa^{-1} \| \langle x\rangle^{-b} q \|_{L^2}^2  \bigr],
\end{align}
which completes the proof of \eqref{gbd'}.

From \eqref{CT G off} we see that
\begin{align}\label{1000}
\bigl\| \langle x\rangle^{-2b}  [1+\gamma]\bigr\|_{L^\infty} \lesssim 1+ \kappa^{-1} \|\langle x\rangle^{-b}q\|_{L^2}^{2}.
\end{align}
The only obstacle to deducing \eqref{pbd'} and \eqref{rbd'} from \eqref{prformula} and \eqref{1000} is commuting $\langle x\rangle^{-3b}$ through the Fourier multipliers.  However, the means for doing so was provided already in \eqref{9:37}. 

It remains to prove \eqref{qquad} and \eqref{gammap1}.  We begin with the former, which is elementary once we prove the following more general fact:
\begin{align} \label{detG}
\det G(x,y)=0 \qtq{for all} x\neq y .
\end{align}
To see this, observe that \eqref{Gidents} implies that this determinant must be constant in each connected component of $x\neq y$.  On the other hand, if one fixes $x$ and sends $y\to\pm\infty$, then \eqref{CT G off} shows that the matrix (and so also its determinant) converges to zero.

Lastly, we prove the first inequality in \eqref{gammap1}. As $\Lax$ is anti-selfadjoint, so
$$
\tfrac12 \bigl[ \LaxR(\kappa) + \LaxR(\kappa)^*\bigr] = \kappa \LaxR(\kappa)^*\LaxR(\kappa)
$$
is positive definite.  Testing this against delta-functions reveals that
$$
\tfrac12\begin{bmatrix} 1+\gamma(x) & g_+(x) \\ g_+(x) & 1+\gamma(x) \end{bmatrix}
$$
is also positive definite, which is precisely the first inequality in \eqref{gammap1}.  The second inequality follows from the first by rewriting \eqref{qquad} as $1-g_-^2 = (1+\gamma)^2-g_+^2$.
\end{proof}

Our next proposition demonstrates that the diagonal Green's functions have a mild dependence on the potential~$q$.  The arguments needed are similar to those just presented, albeit more complicated.

\begin{proposition}[Lipschitz bounds for $g_\pm$ and $\gamma$]\label{P:G Lip}
Fix $b>0$. Adopting the abbreviation
\begin{align}\label{1055}
C_b := 1 + \kappa^{-1}\| \langle x\rangle^{-b} q_1 \|_{L^2}^2 + \kappa^{-1} \|\langle x\rangle^{-b}  q_2 \|_{L^2}^2
\end{align}
we have
\begin{align}\label{G lip b}
\bigl|G(y_1,y_2;q_1) & - G(y_1,y_2;q_2)\bigr| \notag \\
&\lesssim  C_b^2 \langle y_1\rangle^{5b/2} \langle y_2\rangle^{5b/2} \kappa^{-\frac12}  \|\langle x\rangle^{-b} [q_1-q_2] \|_{L^2}
	\exp\bigl\{ -\tfrac\kappa7 |y_1-y_2| \bigr\} 
\end{align}
uniformly for $y_1,y_2\in \R$ and $\kappa\geq 1$.  The diagonal Green's functions satisfy
\begin{align}\label{p lip b}
&\| \langle x\rangle^{-6b}[g_\pm(q_1)-g_\pm(q_2)]\|_{H^1_\kappa} 
\lesssim C_b^{5/2} \|\langle x\rangle^{-b} [q_1-q_2 ] \|_{L^2}, \\
\label{r lip b}%
&\| \langle x\rangle^{-6b}[ g_-(q_1)-g_-(q_2)] \|_{H^2_\kappa}
\lesssim \kappa C_b^{5/2} \|\langle x\rangle^{-b} [q_1-q_2 ] \|_{L^2},\\
\label{g lip b}
&\|\langle x\rangle^{-7b}[\gamma(q_1)-{}\gamma(q_2)]\|_{H^1_{\kappa}} \notag \\ 
& \qquad \lesssim \kappa^{-\frac12} C_b^{5/2} \|\langle x\rangle^{-b} [q_1-q_2 ] \|_{L^2} \bigl[ \| \langle x\rangle^{-b} q_1 \|_{L^2(\R)} + \| \langle x\rangle^{-b} q_2 \|_{L^2(\R)} \bigr] .
\end{align}
\end{proposition}

\begin{proof}
To prove \eqref{G lip b}, we combine \eqref{CT G off} and the resolvent identity. This leads to
\begin{align*}
\text{LHS\eqref{G lip b}}\lesssim C_b^2 \langle y_1\rangle^{b} \langle y_2\rangle^{b}
	\! \int_\R \langle y \rangle^{2b} |q_1(y)-q_2(y)| e^{ -\kappa(|y_1-y| + |y-y_2|)/6} \,dy .
\end{align*}
The estimate \eqref{G lip b} now follows by employing the elementary inequality
$$
\tfrac{1}{6}|y_1-y| + \tfrac{1}{6} |y-y_2| \geq \tfrac{1}{7} |y_1-y_2| + \tfrac{1}{42} \bigl|y-y_1\bigr| + \tfrac{1}{42} \bigl|y-y_2\bigr| 
$$
and Cauchy--Schwarz.

Let us now jump ahead to \eqref{g lip b}.  Mimicking \eqref{983}, we have that
\begin{align}\label{1108a}
\int_\R [\gamma(x;q_1)-\gamma(x;q_2)] \frac{F(x) \, dx}{\langle x\rangle^{7b}}  &= \tr \bigl\{ \langle x\rangle^{-7b} F \LaxR_0 \bigl[Q_1 \LaxR_1 Q_1 - Q_2 \LaxR_2 Q_2\bigr] \LaxR_0\bigr]\bigr\} \!\!\!
\end{align}
where the subscript $0$ refers to $q\equiv 0$, while $1$ and $2$ refer to $q_1$ and $q_2$, respectively.

With a view to estimating RHS\eqref{1108a}, we write
\begin{align*}
Q_1 \LaxR_1 Q_1 - Q_2 \LaxR_2 Q_2 = Q_1 \LaxR_1 \bigl[ Q_1 - Q_2\bigr] + \bigl[ Q_1 - Q_2 \bigr] \LaxR_2  Q_2 
	+  Q_1 \bigl[\LaxR_1 - \LaxR_2\bigr] Q_2 .
\end{align*}
Employing \eqref{CT G off} and \eqref{G lip b}, we then obtain the Hilbert--Schmidt bound
\begin{align*}
\big\|\langle x&\rangle^{-7b/2} \bigl[ Q_1 \LaxR_1 Q_1 - Q_2 \LaxR_2 Q_2 \bigr]\langle x\rangle^{-7b/2} \bigr\|_\HS \\
&\lesssim C_b \|\langle x\rangle^{-b} [q_1-q_2] \|_{L^2}
	\bigl[ \| \langle x\rangle^{-b} q_1 \|_{L^2(\R)} + \| \langle x\rangle^{-b} q_2 \|_{L^2(\R)} \bigr] \\
&\qquad + C_b^2 \kappa^{-\frac12}\|\langle x\rangle^{-b} [q_1-q_2] \|_{L^2} \| \langle x\rangle^{-b} q_1 \|_{L^2(\R)} \| \langle x\rangle^{-b} q_2 \|_{L^2} .
\end{align*}

We are now ready to estimate RHS\eqref{1108a}.  To do so, we employ Lemma~\ref{L:bRb} (with $b$ set to $7b/4$) and the Hilbert--Schmidt bound we just obtained.  This yields
\begin{align*}
\text{RHS\eqref{1108a}}
	&\lesssim \kappa^{-\frac{1}{2}} C_b^{\frac52} \| F \|_{H^{-1}_\kappa} \|\langle x\rangle^{-b}[q_1-q_2]\|_{L^2} \bigl[ \| \langle x\rangle^{-b} q_1 \|_{L^2(\R)} + \| \langle x\rangle^{-b} q_2 \|_{L^2(\R)} \bigr],
\end{align*}
from which \eqref{g lip b} follows by duality.

Although \eqref{g lip b} implies an $L^\infty$ bound, we may employ \eqref{G lip b} directly to obtain
\begin{align*}
\| \langle x\rangle^{-5b} [\gamma_1 - \gamma_2] \|_{L^\infty} \lesssim \kappa^{-\frac12} C_b^2 \|\langle x\rangle^{-b} [q_1-q_2] \|_{L^2}
\end{align*}
and consequently,
\begin{align*}
\| \langle x\rangle^{-6b} [q_1(1+\gamma_1) - q_2(1+\gamma_2)] \|_{L^2} \lesssim C_b^{5/2} \|\langle x\rangle^{-b} [q_1-q_2] \|_{L^2}.
\end{align*}
The claims \eqref{p lip b} and \eqref{r lip b} then follow easily from the relations \eqref{prformula}.
\end{proof}

Although an $L^2$-based theory will suffice for us, we wish to make a few brief comments on how one may transfer additional regularity of $q$ to the diagonal Green's functions.  One approach is via \eqref{prformula}, noting that $\gamma$ is already significantly smoother than $q$; see \eqref{gbd'}.  An alternate approach is to use the bounds of Proposition~\ref{P:G Lip} with $q_2$ being a small translate of $q_1$.  This works well not only with H\"older spaces, but also in $W^{s,p}$ spaces through the Besov characterization.

\subsection{The periodic case} Thus far, our analysis of the Green's function has assumed the mild hypothesis \eqref{det L hyp}.   When $b>\frac12$, this covers the case of periodic potentials $q\in L^2(\T_L)$.

The central result of this subsection is Lipschitz bounds for $\gamma$ and $g_\pm$ in the periodic setting that will allow us to construct solutions to the $H_\kappa$ flow via contraction mapping.  Size will be measured using the one-period $L^2$ norm.  As it does not affect the arguments, we will allow $L=\infty$; however, such $q\in L^2(\R)$ will ultimately have no role in this paper.
Indeed, for samples from the Gibbs state, the $L^2(\T_L)$ norm diverges as $L\to\infty$.  Correspondingly, these periodic estimates do not suffice for the analysis of the Gibbs state in the  $L\to\infty$ limit.    A proper quantitative treatment of that case will be taken up in Section~\ref{S:5}.

As $g_+$, $g_-$, and $\gamma$ all vanish when $q\equiv 0$, specializing the Lipschitz estimates below to $q_2\equiv 0$ yields norm bounds.

\begin{proposition}[Periodic $g_\pm$ and $\gamma$]\label{P:PerioG}
Fix $1\leq L\leq \infty$.  Given $y_1,y_2\in \R$,
\begin{align}\label{1054}
(\kappa,q)\mapsto G(y_1,y_2;\kappa,q) - G_0(y_1,y_2;\kappa,q) 
\end{align}
is a jointly real-analytic function of $\kappa\in[1,\infty)$ and $q\in L^2(\T_L)$.  Adopting the abbreviation
\begin{align}\label{1155}
C_L := 1 + \kappa^{-1}\| q_1 \|_{L^2(\T_L)}^2 + \kappa^{-1} \| q_2 \|_{L^2(\T_L)}^2 ,
\end{align}
we have the Lipschitz bound
\begin{align}\label{CT G L}
\bigl|G(y_1,y_2;q_1) - G(y_1,y_2;q_2)\bigr| \lesssim  \kappa^{-\frac12} C_L^2 \|q_1-q_2\|_{L^2(\T_L)} \exp\bigl\{ -\tfrac\kappa8 |y_1-y_2| \bigr\}
\end{align}
uniformly for $y_1,y_2\in \R$, $\kappa\geq 1$, and $q_1,q_2\in L^2(\T_L)$.  Moreover,
\begin{align}\label{p lip}
\| g_\pm(q_1)-g_\pm(q_2)\|_{H^1_\kappa(\T_L)} 
&\lesssim C_L^{5/2}\|q_1-q_2\|_{L^2(\T_L)} ,\\
\label{r lip}%
\| g_-(q_1)-g_-(q_2)\|_{H^2_\kappa(\T_L)}
&\lesssim \kappa C_L^{5/2} \|q_1-q_2\|_{L^2(\T_L)} , \\
\label{g lip}
\|\gamma(q_1)-\gamma(q_2)\|_{H^1_{\kappa}(\T_L)}  
&\lesssim  C_L^{3} \|q_1-q_2\|_{L^2(\T_L)} .  
\end{align}
\end{proposition}

\begin{proof}
Consider first \eqref{CT G L}.  The Green's function has the following translation symmetry: for any $h\in\R$,
\begin{align}\label{G translate}
G\bigl(y_1+h,y_2+h;\kappa,q\bigr) = G\bigl(y_1,y_2;\kappa,q(\cdot+h)\bigr)
\end{align}
and correspondingly, it suffices to consider the case $y_1=0$ (with a translated potential).  We then see that \eqref{CT G L} follows by choosing  $b>\frac12$ in \eqref{G lip b} together with two further observations:  First, the slight deterioration in the exponential decay between \eqref{G lip b} and \eqref{CT G L} allows us to absorb the factor $\langle y_2\rangle^{5b/2}$.  Second, as $b>\frac12$ we see that
\begin{align}\label{1175}
\| \langle x\rangle^{-b} q_1 \|_{L^2(\R)} \lesssim \| q_1 \|_{L^2(\T_L)}
\end{align}
and similarly for $q_2$ and $q_1-q_2$.  In particular, this shows $C_b\lesssim C_L$.

To continue, we need the following improvement of \eqref{1175}:
\begin{align}\label{1182}
\| f \|_{L^2(\T_L)}^2 \approx_b \int_{-L}^L \int_\R |f(x+h)|^2 \langle x\rangle^{-2b} \, dx\,dh
\end{align}
valid for any $2L$-periodic $f$ and any $b>\frac12$.  This is easily verified by integrating first in $h$.  By the same method, one sees that for any $b>\frac12$,
\begin{align}\label{1186}
\| f \|_{H^1_\kappa(\T_L)}^2 \approx_b \int_{-L}^L \int_\R \bigl[ |f'(x+h)|^2 + \kappa^2 |f(x+h)|^2\bigr]\langle x\rangle^{-2b} \, dx\,dh 
\end{align}
and an analogous result for $H^2_\kappa(\T_L)$.  Using such relations, we can then derive \eqref{p lip}, \eqref{r lip}, and \eqref{g lip} from their corresponding counterparts in Proposition~\ref{P:det diag}.  We will illustrate the method using \eqref{p lip} as our example.

Let us temporarily adopt the notation $q^h(x)=q(x+h)$.  Note that
$$
C_b(q_1^h,q_2^h)\lesssim C_L(q_1^h,q_2^h)=C_L(q_1,q_2) \ \ \text{and}\ \  g_\pm^h(x;q):=g_\pm(x;q^h) = g_\pm(x+h;q);
$$
indeed, these follow from \eqref{1175} and \eqref{G translate}, respectively.  We now apply \eqref{p lip b} to the translated potentials $q_1^h$ and $q_2^h$.  After squaring, this yields
\begin{align}\label{p lip bL}
\| \langle x\rangle^{-6b}[g_\pm^h(q_1)-g_\pm^h(q_2)]\|_{H^1_\kappa(\R)}^2
&\lesssim C_L^{5} \|\langle x\rangle^{-b} [q_1^h-q_2^h ] \|_{L^2(\R)}^2,
\end{align}
which we integrate in $h$ over the interval $[-L,L]$.  In this way, \eqref{p lip} follows from \eqref{1182} and \eqref{1186}.

Let us turn finally to the question of analyticity.  Recall $\kappa \mapsto G(y_1,y_2;\kappa,q)$ was shown to be analytic in Proposition~\ref{P:Lax Op} using the resolvent identity \eqref{basicRI}.  In order to adapt that argument to joint analyticity, it suffices to show that $(\kappa,q)\mapsto \LaxR(\kappa,q)$ is jointly analytic (with values in the space of bounded operators).

The resolvent identity allows us to write
\begin{align*}
R(\kappa+\vk,q+f) = R(\kappa,q) + \sum_{\ell=1}^\infty  R(\kappa,q) (- B)^\ell	
	\qtq{with} B = \bigl[\begin{smallmatrix}\vk & f \\ -f & \vk\end{smallmatrix}\bigr] R(\kappa,q) .
\end{align*}
This is contingent, of course, on the convergence of this series, which we must demonstrate for small $\vk\in\R$ and $f\in L^2(\T_L)$ .  Happily,
\begin{align*}
\| B \|_{\op}^2 =  \| B^* \! B \|_{\op} \lesssim_q \vk^2 + \| f\|_{L^2(\T_L)} ^2
\end{align*}
as follows from Schur's test and the $q_2\equiv 0$ case of \eqref{CT G L}.
\end{proof}

\begin{proposition}[$L^2$ conservation laws]\label{P:Conserve}
Fix $1\leq L\leq \infty$.  Given  $q\in L^2(\T_L)$ and $\kappa\geq 1$, we define
\begin{align}\label{1208}
A(\kappa,q) := \int_{-L}^L \frac{q(x)g_-(x;\kappa,q)}{2+\gamma(x;\kappa,q)}\,dx.
\end{align}
This is a real-analytic function of $\kappa$ and of $q\in L^2(\T_L)$; moreover,
\begin{align}\label{1212}
\frac{\partial A}{\partial\kappa} = \int_{-L}^L \gamma(x)\,dx 
	\qtq{and} \frac{\delta A}{\delta q} = g_- .
\end{align}
The functionals $A(\kappa)$, $A(\vk)$, and $M$ all Poisson commute throughout $L^2(\T_L)$ for any $\vk,\kappa\geq 1$; moreover, $A(\kappa)$ commutes with the Hamiltonian $H_\mKdV$ throughout $H^1(\T_L)$.
\end{proposition}

\begin{proof}
The inequality \eqref{gammap1} shows that $2+\gamma\geq 1$.  Together with the other estimates of Proposition~\ref{P:PerioG}, this guarantees that the integrand in \eqref{1208} is indeed absolutely integrable.  The joint analyticity of $A(\kappa,q)$ follows from the analyticity of the Green's function shown there.

To proceed, we need the following important identity discovered in \cite{HGKV}:
\begin{align}\label{1237}
\gamma &= \frac{\partial\ }{\partial x} \Bigl(
	\frac{g_{+}\tfrac{\partial g_{-}}{\partial\kappa} - \tfrac{\partial g_{+}}{\partial\kappa} g_{-}}{2(2+\gamma)} \Bigr)
	+ \frac{\partial\ }{\partial\kappa} \Bigl(\frac{qg_-}{2+\gamma} \Bigr).
\end{align}
We do not repeat the proof.  It is involved; however, it relies only on algebraic manipulations of the identities presented in Proposition~\ref{P:det diag} and their $\kappa$-derivatives.

Integrating both sides over one period confirms the left identity in \eqref{1212}.  Our next goal is to use this identity to understand $\tfrac{\delta\ }{\delta q}\tfrac{\partial A}{\partial\kappa}$.

Consider the directional derivative in the direction $f\in L^2(\T_L)$.  Using the left identity in \eqref{1212}, \eqref{G translate}, and the periodicity of $f$ and $q$, we find that
\begin{align*}
\bigl\langle f, \tfrac{\delta\ }{\delta q} \tfrac{\partial A}{\partial\kappa} \bigr\rangle
	&= - \tr \int_{-L}^L \int_\R G(y,x) \bigl[\begin{smallmatrix}0 & f(x) \\ -f(x) & 0 \end{smallmatrix}\bigr]  G(x,y)\,dx\,dy\\
&= - \sum_{m\in\Z} \tr \int_{-L}^L \int_{-L}^L G(y,x+2mL) \bigl[\begin{smallmatrix}0 & f(x) \\ -f(x) & 0\end{smallmatrix}\bigr]  G(x+2mL,y)\,dx\,dy\\
&= - \sum_{m\in\Z} \tr \int_{-L}^L \int_{-L}^L G(y-2mL,x) \bigl[\begin{smallmatrix}0 & f(x) \\ -f(x) & 0\end{smallmatrix}\bigr]  G(x,y-2mL)\,dx\,dy\\
&= - \tr \int_\R \int_{-L}^L  G(y,x) \bigl[\begin{smallmatrix}0 & f(x) \\ -f(x) & 0 \end{smallmatrix}\bigr]  G(x,y)\,dx\,dy.
\end{align*}
Cycling the trace and performing the $y$ integral, causes the integral kernel of the operator $\LaxR^2=-\tfrac{\partial\ }{\partial\kappa} \LaxR$ to appear.  In this way, our computation reveals that
\begin{align}\label{1255}
\tfrac{\delta\ }{\delta q} \tfrac{\partial A}{\partial\kappa} = \tfrac{\partial\ }{\partial\kappa}  g_- .
\end{align}

Evidently, the second identity in \eqref{1212} should follow by integrating both sides of \eqref{1255} with respect to $\kappa$.  One need only show that there is no boundary contribution at infinity.  To this end, we observe that
\begin{align*}
\| g_-\|_{L^2} + \bigl\| \tfrac{\delta g_-}{\delta q}\bigr\|_{\op} + \bigl\| \tfrac{\delta \gamma}{\delta q}\bigr\|_{\op}
	 \to 0 \qtq{as} \kappa\to\infty
\end{align*}
as can be seen from \eqref{p lip} and \eqref{g lip}.  

Next we verify commutativity under the Poisson structure defined in \eqref{PB}.  As a first example, we combine \eqref{1212}, \eqref{Gderiv}, and \eqref{rderiv} to obtain
\begin{align}\label{1267}
\{ M, A(\kappa) \} = \int_{\T_L} q(x) g_-'(x;\kappa)\,dx  = -\kappa \int_{\T_L} \gamma'(x;\kappa) \,dx =0.
\end{align}

Next we verify the commutativity of $A(\kappa)$ and $A(\vk)$.  Integration by parts shows,
\begin{align*}
\bigl\{ A(\vk), A(\kappa) \bigr\} &=  \int_{\T_L}\frac{(\vk+\kappa)g_-(x;\vk) g_-'(x;\kappa)}{\vk+\kappa} \,dx \\
	&= \int_{\T_L} \frac{\vk g_-(x;\vk) g_-'(x;\kappa)- \kappa g_-'(x;\vk) g_-(x;\kappa)}{\vk+\kappa} \,dx.
\end{align*}
To see that this integral vanishes, we then employ the identity
\begin{align*}
(\vk-\kappa)\bigl[\vk & g_-(\vk) g_-'(\kappa) - \kappa g_-'(\vk) g_-(\kappa)\bigr] \\
	&= \vk\kappa \Bigl(g_+(\vk) g_+(\kappa) - g_-(\vk) g_-(\kappa) -[1+\gamma(\vk)][1+\gamma(\kappa)]\Bigr)',
\end{align*}
which is easily verified by combining \eqref{Gderiv}, \eqref{pderiv}, and \eqref{rderiv}.

In a similar fashion, \eqref{Gderiv}, \eqref{pderiv}, and \eqref{rderiv} imply
\begin{align}\label{1282}
[-q''+2q^3]g_-' = \partial_x\bigl[ 2\kappa q'g_+ + 4\kappa^2 qg_- + 4\kappa^3\gamma - 2\kappa q^2(1+\gamma)\bigr],
\end{align}
which then demonstrates that $\{H_\mKdV,A(\kappa)\}=0$.
\end{proof}

The Poisson commutativity of $A(\kappa;q)$ and $H_\mKdV$ shows that $A(\kappa;q)$ is conserved under \eqref{mkdv}.  In \cite{HGKV}, for example, this was crucial for controlling the low-regularity norms of solutions to \eqref{mkdv}.  As samples of the finite-volume Gibbs state are in $L^2(\T_L)$, we may exploit conservation of $M(q)$ for this role.

In this paper, the principal role of the Poisson commutativity of $A(\kappa;q)$ and $H_\mKdV$ is ensuring that the \eqref{mkdv} flow commutes with the approximate flow that we will be employing.  Our approximate flows are based on the Hamiltonians
\begin{align}\label{Hk defn}
H_\kappa(q) := 4\kappa^2 M(q) - 4\kappa^3 A(\kappa;q) .
\end{align}
Proposition~\ref{P:Conserve} ensures that these Poisson commute with $H_\mKdV$. Let us now present the basic properties of the flows generated by $H_\kappa(q)$:

\begin{proposition}\label{P:HK well}
Fix $1\leq L\leq \infty$ and $\kappa\geq 1$. Then the Hamiltonian flow induced by $H_{\kappa}$, namely,
\begin{align}\label{Hkflow}
\tfrac{d}{dt}q(x)= \partial_{x} \{ 4\kappa^2 q - 4\kappa^3 g_-(\kappa)\}= 4\kappa^2 q'(x)+8\kappa^4 g_+(x; \kappa), 
\end{align}
is globally well-posed on $L^2(\T_L)$.  This evolution preserves $M(q)$ and $A(\vk;q)$ for all $\vk\geq 1$ and commutes with the $H_\vk$ flow.  Moreover, under this flow, 
\begin{align}\label{+Hkflow}
\tfrac{d}{dt} g_+(\vk) & = 4\kappa^2 g_+'(\vk) + \tfrac{4\kappa^4}{\kappa^2 - \vk^2} \Big\{ g_+(\kappa)\bigl[ 1 + \gamma(\vk) \big]
	- \big[ 1 + \gamma(\kappa)\bigr] g_+(\vk) \Big\}' \\
\label{-Hkflow}
\tfrac{d}{dt} g_-(\vk) & = 4\kappa^2 g_-'(\vk) - \tfrac{8\kappa^4\vk}{\kappa^2-\vk^2} \Big\{ g_+(\kappa)\bigl[ 1 + \gamma(\vk) \big]
	- \big[ 1 + \gamma(\kappa)\bigr] g_+(\vk) \Big\} \\
\label{gHkflow}
\tfrac{d}{dt} \gamma(\vk) & = 4\kappa^2 \gamma'(\vk) - \tfrac{8\kappa^5\vk }{(\kappa^2-\vk^2)^{2}} \Bigl\{ g_-(\kappa)g_-(\vk) \Bigr\}' \\
& \qquad + \tfrac{4\kappa^4(\kappa^2+\vk^2) }{(\kappa^2-\vk^2)^{2}}
	\Bigl\{ g_+(\kappa)g_+(\vk) - \gamma(\kappa)- \gamma(\vk) - \gamma(\kappa)\gamma(\vk)\Bigr\}'  \notag
\end{align}
for any distinct  $\vk,\kappa \geq 1$.
\end{proposition}

\begin{proof}
The formulas for the $H_\kappa$ flow follow from \eqref{1212} and then \eqref{rderiv}.
To prove well-posedness of this flow, we first formulate the problem as the integral equation
\begin{align}
q(t,x)=q(0,x+4\kappa^2 t) +8 \kappa^4 \int_{0}^{t} g_+(x+4\kappa^2(t-t'); \kappa, q(t')) \, dt, \label{Hkduhamel}
\end{align}
which corresponds to treating \eqref{Hkflow} as a perturbation of the translation flow with Hamiltonian $4\kappa^2M(q)$.  From \eqref{p lip}, we know that $q\mapsto g_+$ is $L^2$-Lipschitz and so well-posedness follows easily from contraction mapping.  

Global well-posedness then follows from the conservation of $M(q)$,  which is a consequence of Proposition~\ref{P:Conserve} or can be deduced by direct computation using \eqref{Gderiv}.  Likewise, the conservation of $A(\vk)$ follows from Proposition~\ref{P:Conserve}.

The commutativity of the $H_\kappa$ and $H_\vk$ flows is easily checked for smooth data.  This extends to all of $L^2(\T_L)$ via well-posedness. 

The derivation of the evolution equations for the diagonal Green's functions from first principles is rather lengthy.  However, these results can be deduced directly from \cite[Corollary~4.8]{HGKV}.  We caution the reader that the regularized Hamiltonian used here and in \cite{HGKV} are different because the Poisson structures are different; however, the evolution \eqref{Hkflow} under discussion here is indeed the specialization of the flow studied in \cite{HGKV}  to real initial data. 
\end{proof}

As discussed in the Introduction, our approach to treating \eqref{mkdv} in the thermodynamic limit rests on first extending the $H_\kappa$ flow to the full line and then sending $\kappa\to\infty$.  Toward this goal, we will verify that the $H_\kappa$ flow preserves the finite-volume Gibbs state in the next section.  As a small detour, we will also show how our approach leads to a new proof of invariance of the finite-volume Gibbs state under \eqref{mkdv} dynamics.  The existence of solutions for such initial data follows easily from the well-posedness of \eqref{mkdv} in $L^2(\T_L)$, which was first proved in \cite{MR2131061}.  Our future arguments, however, rely on the fact that the $H_\kappa$ flows converge to \eqref{mkdv} as $\kappa\to\infty$.  This was proved in \cite{Forlano} as a key part of providing a new proof that \eqref{mkdv} is well-posed in $L^2(\T_L)$.

\begin{theorem}[\cite{Forlano,MR2131061}]\label{T:1404}
The \eqref{mkdv} flow is globally well-posed in $L^2(\T_L)$ and it commutes with the $H_\kappa$ flow. Moreover, for fixed $q^0\in L^2(\T_L)$ and $T>0$,
\begin{align}\label{L2 flow conv}
\lim_{\kappa \to \infty} \sup_{|t|\leq T}\bigl\|e^{tJ\nabla H_\kappa}q^0-e^{tJ\nabla H_\mKdV}q^0\bigr\|_{L^2 (\T_L)}=0.
\end{align}
\end{theorem}

A key ingredient in the argument in \cite{Forlano} is to make a diffeomorphic change of unknown (= gauge transformation).  This is to overcome the difficulty that the nonlinearity in \eqref{mkdv} does not make sense (as a distribution) pointwise in time for $C_t^{ } L^2_x$ solutions.  Convergence in \eqref{L2 flow conv} is actually deduced from convergence of the gauge variable $g_-$, whose evolution under the $H_\kappa$ flow was presented in \eqref{-Hkflow}.

Under the \eqref{mkdv} flow, the diagonal Green's functions evolve as follows:
\begin{align}\label{+flow}
\tfrac{d}{dt} g_+ & = -g_+''' + 6(q^2 g_+)' \\
\label{-flow}
\tfrac{d}{dt} g_- & = - g_-''' + 6 q^2 g_-' \\
\label{gflow}
\tfrac{d}{dt} \gamma & = - \gamma'''  + \bigl\{ 6q^2[1+\gamma] - 12\vk q g_- - 12\vk^2 \gamma \bigr\}'   .
\end{align}
As well as playing an important role in the proof of Theorem~\ref{T:1404}, these formulae also allow us to give an intrinsic definition of what it means for a wide class of low-regularity functions to be solutions of \eqref{mkdv}:

\begin{definition}\label{D:mkdv green}
Suppose $\langle x\rangle^{-b} q(t,x)$ is $L^2_x(\R)$-continuous in time for some $b\geq 0$.  We say that $q$ is a \emph{green solution} to \eqref{mkdv} if for all times $t_1<t_2$ and all $\vk\geq 1$,
\begin{gather}\label{Green-}
g_-(t_2) - g_-(t_1) = \int_{t_1}^{t_2} - g'''_-(t,x) + 6 q(t,x)^2 g_-'(t,x) \,dt \\
\label{Greeng}
 \gamma(t_2) - \gamma(t_1) = \int_{t_1}^{t_2} - \gamma'''  + \bigl\{ 6q^2[1+\gamma] - 12\vk q g_- - 12\vk^2 \gamma \bigr\}' \,dt
\end{gather}
in the sense of distributions.
\end{definition}

The estimate \eqref{rbd'} together with the well-known embedding $H^1\hookrightarrow L^\infty$ and its dual $L^1\hookrightarrow H^{-1}$ ensure that RHS\eqref{Green-} belongs to $H^{-1}_\loc(\R)$ and so is certainly a distribution.  Similarly, \eqref{gbd'} ensures that RHS\eqref{Greeng} belongs to $H^{-2}_\loc(\R)$.

There is no reason to demand that green solutions satisfy a parallel identity corresponding to \eqref{+flow}, since this can be derived from \eqref{Green-} by taking a spatial derivative and using \eqref{rderiv}.

We include \eqref{Greeng} in our definition for an important reason:  The function $\gamma(\vk)$ serves as a one-parameter family of microscopic conservation laws for \eqref{mkdv}, as \eqref{gflow} demonstrates; correspondingly, green solutions retain this important expression of complete integrability.  Indeed, the name of this type of solution was chosen to reflect this double meaning: green in the sense of Green's function, but also in the sense of conservationism.   For comparison, the notion of solution introduced in \cite{Christ'05} allows for non-conservative solutions of completely integrable systems.

In Section~\ref{S:7}, we will describe the notion of green solution appropriate to \eqref{kdv}, which was introduced already in \cite{KMV}.

\section{Invariance of the periodic Gibbs measure}\label{S:invariance}

At the conclusion of this section, we provide a new proof of the invariance of the Gibbs measure under \eqref{mkdv} on $\T_L$ for finite $L$.  Our proof relies on the invariance of this measure under the $H_\kappa$ flow; this also plays a central role in the infinite-volume analysis.

\begin{theorem}\label{THM:HKinv}
Fix $1\leq L<\infty$, $\kappa\geq 1$, and let $q^0$ be Gibbs distributed.\\
\slp{a} Almost surely, the $H_\kappa$ flow admits a global and unique $C_t(\R;L^2(\T_L))$ solution $q(t)$ with initial data $q^0$.\\
\slp{b} The solution $q(t)$ is Gibbs distributed at all times $t\in\R$.\\
\slp{c} For each $p<\infty$, $s<\frac12$, and $b>\frac1p$, there is an $\alpha>0$ so that
\begin{align}\label{bounds kL}
\sup_{\kappa\geq 1} \ \sup_{L\geq 1} \ \E \Bigl\{  \bigl\| \langle x\rangle^{-b} \, q(t,x) \|_{C^\alpha([-T,T];W^{s,p}(\T_L))}^r  \Bigr\}  < \infty
\end{align} 
for every choice of $T<\infty$ and $1\leq r<\infty$.
\end{theorem}

As the reader may well intuit, the uniformity in $L$ expressed in \eqref{bounds kL} will allow us to send $L\to\infty$ in the next section.  The limit will then inherit these bounds.  The uniformity in $\kappa$ will then be exploited to send $\kappa\to\infty$ in Section~\ref{S:6}.

We will break the proof of Theorem~\ref{THM:HKinv} into two parts, verifying parts (a) and (b) first and then returning to part (c) later, after Lemma~\ref{LEM:prob3}.  Note that part (c) is unnecessary for proving invariance of the Gibbs state under \eqref{mkdv} in finite volume; moreover, our insistence on uniformity in $\kappa$ and $L$ makes the proof of this part somewhat involved.

\begin{proof}[Proof of \slp{a} and \slp{b}]
Let $q^0$ be Gibbs distributed on $\T_L$.  As $q^0\in L^2$ almost surely, Proposition~\ref{P:HK well} guarantees the existence of a global $H_\kappa$ evolution $q(t)= e^{tJ\nabla H_{\kappa}} q^0$.

Suppose, towards a contradiction, that the Gibbs measure is not invariant under the $H_{\kappa}$ flow. Then, there is a $\phi\in C^\infty(\T_L)$, a time $T>0$, and a $B>0$ so that 
\begin{align}
\E\bigl\{ e^{i\jb{q(T),\phi}} \chi_B(q(T)) \bigr\} \neq \E\bigl\{ e^{i\jb{q^0,\phi}}  \chi_B(q^0) \bigr\}. \label{contra}
\end{align}
Here $\chi_B(q) = \chi ( M(q)/B )$ and $\chi\in C^\infty_c$ satisfies $\chi(x)=1$ for $|x|\leq 1$ and $\chi(x)=0$ for $|x|\geq 2$.  Notice that this construction guarantees
\begin{equation}\label{KMS chi}
\chi_B (q)  \cdot \chi_{2B}(q) \equiv \chi_B (q)  \qtq{and}
	\chi_B (q)  \cdot \tfrac{\delta\ }{\delta q}\chi_{2B}(q) \equiv 0 .
\end{equation}

We define $G:\R\times L^2(\T_L)\to \mathbb{C}$ by 
\begin{align}\label{5:02}
G(t,q^0) := \exp\bigl( i\bigl\langle\phi, e^{tJ\nabla H_{\kappa}} e^{-4\kappa^2 t\partial} q^0 \bigr\rangle\bigr)
	\cdot \chi_B\bigl( e^{tJ\nabla H_{\kappa}} e^{-4\kappa^2 t\partial} q^0 \bigr),
\end{align}
which combines the $H_{\kappa}$ flow map with translation.  Translation invariance of the Gibbs measure shows that \eqref{contra} is equivalent to the claim $\E[ G(0,q^0) ] \neq \E[G(T,q^0)]$.  We wish to reach a contradiction by proving that 
\begin{align}\label{5:03}
\partial_{t} \E\bigl[G(t,q^0) \bigr] = 0.
\end{align}

We consider the combined $H_\kappa$ and translation flows here for the same reason we did in the proof of Proposition~\ref{P:HK well}, namely, $L^2$ well-posedness is a direct application of the Cauchy--Peano contraction mapping argument.  This argument shows that $G$ is a $C^1$ function of time and of the initial data; moreover,
\begin{align}\label{G for inv}
\bigl| G(q) \bigr| +  \big\| \tfrac{\delta G}{\delta q}\big\|_{L^2(\T_L)} \lesssim_{B,|t|} 1
\end{align}
and so $G$ is $C^1$ in the sense \eqref{C^1 on L^2}.  In view of this differentiability, we may recast our earlier goal \eqref{5:03} as showing
\begin{align}\label{5:03'}
 \E\bigl[  \bigl\langle g_- ' (\kappa;q^0), \tfrac{\delta G}{\delta q}(t,q^0) \bigr\rangle \bigr] = 0.
\end{align}
Henceforth, $t$ will be regarded as fixed.

The combined $H_\kappa$ and translation flow has Hamiltonian $-4\kappa^3A(\kappa,q)$.  This is real-analytic and from \eqref{gammap1}, \eqref{1212}, and \eqref{p lip}, we see that
\begin{align}\label{A is C1}
\bigl| A \bigr|^2 + \kappa \bigl\|\tfrac{\delta A}{\delta q}\bigr\|_{L^2}^2
	\lesssim \kappa^{-1} \| q\|_{L^2(\T_L)}^2 \bigl( 1 + \kappa^{-1}\| q\|_{L^2(\T_L)}^2 \bigr)^6 .
\end{align}
In order to apply Theorem~\ref{T:KMS}, we must ensure that \eqref{C^1 on L^2} holds.  For this reason, we introduce a further cut-off, defining $F(q):=\chi_{2B}(q) A(\kappa;q)$.

Combining the above definitions with \eqref{KMS chi} and \eqref{1212}, we find that LHS\eqref{E:C^1 KMS} takes the form
\begin{align}\label{KMS1}
\E\Bigl[ \bigl\{ F_N , G_N \bigr\} \Bigr] = \E \Bigl[ \bigl\langle P_{\leq N} \bigl(g_{-,N}'\bigr) , \tfrac{\delta G_N}{\delta q} \bigr\rangle \Bigr]
	\ \text{ where }\ g_{-,N}(x) := g_-(x;\kappa,q^0_{\leq N}).
\end{align}
In view of \eqref{G for inv} and \eqref{p lip}, this quantity converges to LHS\eqref{5:03'} as $N\to\infty$.  Thus, our goal has become to show that this limit is zero, or equivalently, by Theorem~\ref{T:KMS}, to show that
\begin{align}\label{1410}
\lim_{N\to\infty} \E\Bigl[  \Bigl\langle P_{\leq N} \bigl(g_{-,N}'\bigr),\;
	-\partial_x^2 q^0 + V'(q^0) \Bigr\rangle    G_N(q^0)  \Bigr] =0.
\end{align}

Moving $P_{\leq N}$ across the inner product and employing the identities \eqref{1267} and \eqref{1282} with $q=q^0_{\leq N}$, we can further simplify \eqref{1410} to the claim that
\begin{align}\label{KMS2}
\lim_{N\to\infty}   \E\Bigl[ \Bigl\langle g_{-,N}' ,\; P_{\leq N}([q^0]^3) - [P_{\leq N} q^0]^3 \Bigr\rangle G_N(q)\Bigr] = 0.
\end{align}
This in turn follows from dominated convergence, using \eqref{Wsp moments}, \eqref{p lip}, and the fact that $P_{\leq N} f \to f$ in $L^p$ for any $f\in L^p$.
\end{proof}

To prove H\"older continuity in time, we exploit the Kolmogorov criterion in the following quantitative form.  This particular formulation appears as  \cite[Lemma~2.4]{KMV}, where a proof can be found.

\begin{lemma}\label{LEM:prob3}
 Given $T,\eps>0$, $\alpha\in(0,1)$, $1 \leq r<\infty$, a Banach space $X$, and a process $F:[-T,T]\to X$ that is almost surely continuous, 
\begin{align*}
\E\Big\{ \|F\|_{C_{t}^{\alpha}X}^{r}\Big\} \lesssim_{r,\eps, T}
	\E\Big\{ \|F(0)\|_{X}^{r}\Big\} +\sup_{-T \leq s<t\leq T} \E \bigg\{ \frac{\|F(t)-F(s)\|_{X}^{r}}{|t-s|^{1+\alpha r +\eps}} \bigg\}.
\end{align*}
\end{lemma}

\begin{proof}[Proof of Theorem~\ref{THM:HKinv}\slp{c}]
First we demonstrate that for any $1\leq r<\infty$,
\begin{align}\label{two pt kL}
\sup_{\kappa\geq 1} \ \sup_{L\geq 1} \ \E \Bigl\{
	\| \langle x\rangle^{-9}  [ q(t_2) - q(t_1) ] \|_{H^{-3}}^r \Bigr\}  \lesssim_r | t_2-t_1|^{r}
\end{align}
uniformly for $t_1,t_2\in\R$.

Combining \eqref{prformula} and \eqref{Hkflow} we may write the $H_\kappa$ flow as
\begin{align*}
\tfrac{d}{dt} q=  - \partial_x^3 \tfrac{4\kappa^2}{4\kappa^2-\partial^2}  q - \partial_x \tfrac{16\kappa^4}{4\kappa^2-\partial^2} \gamma q . 
\end{align*}
Using elementary commutations and the embedding $L^1\hookrightarrow H^{-1}$, we deduce that
\begin{align*}
\bigl\| \langle x\rangle^{-9}  [ q(t_2) - q(t_1) ] \bigr\|_{H^{-3}}  \lesssim  \int_{t_1}^{t_2}
	\bigl\| \langle x\rangle^{-1} q(t) \bigr\|_{L^2} + \kappa^2 \bigl\| \langle x\rangle^{-9} q(t)  \gamma(\kappa,q(t)) \bigr\|_{L^1} \, dt .
\end{align*}

Recall that $q(t)$ is Gibbs distributed at all times $t\in\R$.  Thus, by Minkowski's inequality, we then find
\begin{align*}
\E \Bigl\{ \| \langle x\rangle^{-9}  [ q(t_2) - q(t_1) ] \|_{H^{-3}}^r \Bigr\} \lesssim  |t_2 - t_1|^r \,
	\E \Bigl\{ \| \langle x\rangle^{-9} q \bigr\|_{L^2}^r + \kappa^{2r} \| \langle x\rangle^{-9} q \gamma \bigr\|_{L^1}^r \Bigr\},
\end{align*}
where $q$ is Gibbs distributed and $\gamma=\gamma(x;\kappa,q)$.  Notice that \eqref{two pt kL} will follow once we prove a bound on this expectation (for each $r$) uniformly in $\kappa$ and $L$.

Combining \eqref{qquad}, \eqref{gbd'}, and \eqref{pbd'} we obtain
\begin{align}\label{1558}
\| \langle x\rangle^{-8} \gamma \|_{L^1_x} &\lesssim  \| \langle x\rangle^{-4} g_+ \|_{L^2_x}^2 +
  	\| \langle x\rangle^{-4} g_- \|_{L^2_x}^2 + \| \langle x\rangle^{-4} \gamma(t) \|_{L^2_x}^2 \\
&\lesssim \kappa^{-2}  \|\langle x\rangle^{-1}q\|_{L^2}^{2} \Bigl( 1 +\kappa^{-1}\|\langle x\rangle^{-1}q\|_{L^2}^{2}\Bigr)^3, \notag
\end{align}
which, crucially, gains two powers of $\kappa$.  Thus, using Proposition~\ref{PROP:Gibbs} we find
$$
\sup_{\kappa\geq 1} \ \sup_{L\geq 1} \  \E \Bigl\{ \| \langle x\rangle^{-9} q \bigr\|_{L^2}^r
	+ \kappa^2 \| \langle x\rangle^{-1} q \bigr\|_{L^\infty}^r \| \langle x\rangle^{-8} \gamma \bigr\|_{L^1}^r \Bigr\} < \infty
$$
and so have justified \eqref{two pt kL}.

One may also get a bound on $q(t_2)-q(t_1)$ that is independent of $\kappa$ and $L$ by estimating the terms individually using \eqref{Wsp moments} and invariance of the Gibbs measure:
\begin{align*}
\sup_{\kappa\geq 1} \ \sup_{L\geq 1} \ \E \Bigl\{
	\| \langle x\rangle^{-\tilde b}  [ q(t_2) - q(t_1) ] \|_{W^{\tilde s,\tilde p}(\R)}^r \Bigr\}  < \infty
\end{align*}
for any $\tilde p<\infty$, $\tilde s<\frac12$, and $\tilde b > \frac1{\tilde p}$.
Interpolating between this and \eqref{two pt kL} yields
\begin{align}\label{two pt Wsp kL}
\sup_{\kappa\geq 1} \ \sup_{L\geq 1} \ \E \Bigl\{
	\| \langle x\rangle^{-b}  [ q(t_2) - q(t_1) ] \|_{W^{s,p}(\R)}^r \Bigr\}  \lesssim_r | t_1-t_2|^{c r}
\end{align}
for a small constant $c(p,s,b)>0$.

The estimate \eqref{bounds kL} now follows from Lemma~\ref{LEM:prob3}.  As $r$ can be made arbitrarily large, the only restriction on the H\"older exponent is that $\alpha<c(p,s,b)$.
\end{proof}

\begin{theorem}[Gibbs measure invariance under the defocusing periodic mKdV]
Fix $1\leq L<\infty$ and let $q^0$ be Gibbs distributed on $\T_L$. If $q$ is the global solution to \eqref{mkdv} with initial data $q^0$, then $q(t)$ is Gibbs distributed for every $t\in \R$.\label{T:4.3}
\end{theorem}

\begin{proof}
As $q^0 \in L^2(\T_L)$ almost surely, we obtain unique global solutions $q_{\kappa}(t)$ to the $H_{\kappa}$ flows and a unique global solution $q$ for the mKdV flow. By Theorem~\ref{THM:HKinv}, we know that the solutions $q_{\kappa}(t)$ are Gibbs distributed for all $t\in \R$, that is, 
\begin{align}\label{characHk}
\E\bigl[ e^{i \jb{q_{\kappa}(t),\phi}}\bigr] = \E \bigl[ e^{i \jb{q^0,\phi}}\bigr]
\end{align}
for every $t\in \R$ and $\phi\in C^{\infty}(\T_L)$. Sending $\kappa\to\infty$ using \eqref{L2 flow conv} we deduce that
\begin{align*}
\E\bigl[ e^{i \jb{q(t),\phi}}\bigr] = \E \bigl[ e^{i \jb{q^0,\phi}}\bigr]
\end{align*}
and so that Gibbs measure on $\T_L$  is invariant under \eqref{mkdv}.
\end{proof}

\section{$H_\kappa$ dynamics and invariance in infinite volume}\label{S:5}

Our goals in this section are to construct global solutions to the $H_\kappa$ flow \eqref{Hkflow} for almost every sample from the infinite-volume Gibbs distribution and to show that these dynamics have all the properties that we will need to send $\kappa\to\infty$ in the next section in order to construct \eqref{mkdv} dynamics.

We will construct infinite-volume solutions as limits of the finite-volume solutions described in Proposition~\ref{P:HK well} and Theorem~\ref{THM:HKinv}.  This requires the probabilistic set-up described in Section~\ref{ss2.2}; specifically, we work on a complete probability space $\Omega$ with random variables $q_L^0$ indexed by $L\in 2^\N\cup\{\infty\}$ that are Gibbs distributed on $\T_L$ and satisfy
\begin{align}\label{1562}
\PP\Bigl( q_L^0(x) \not\equiv q_\infty^0(x) \text{ on the interval } \bigl[-\tfrac L2,\tfrac L2 \bigr] \Bigr) \lesssim e^{-c L}
\end{align}
uniformly in $L$ for some fixed $c>0$.  Recall that Proposition~\ref{P:coupled q_L} shows that such coupling is possible.

Next, let us clarify what we will mean by a solution of the $H_\kappa$ flow in infinite volume.  To this end, we introduce some notations:
\begin{gather*}
\chi_L (x) \qtq{denotes the indicator function of} [-L,L] \\
L^p_b := \bigl\{ f:\R\to\R : \langle x\rangle^{-b} f(x) \in L^p(\R) \bigr\} \qtq{and} \| f \|_{L^p_b} := \| \langle x\rangle^{-b} f \|_{L^p(\R)} \\
W^{s,p}_b:=\bigl\{ f:\R\to\R : \langle x\rangle^{-b} f(x) \in W^{s,p}(\R) \bigr\} \qtq{and} \| f \|_{W^{s,p}_b} := \| \langle x\rangle^{-b} f \|_{W^{s,p}(\R)}  .
\end{gather*}
We shall only be employing $\chi_L$ when $L\in 2^\N\cup\{\infty\}$.  Similarly, we only consider
\begin{align}\label{spb}
(s,p,b) \qtq{with} 0\leq s<\tfrac12, \quad 1\leq p<\infty \qtq{and} b>\tfrac1p;
\end{align}
this ensures that Gibbs samples belong to $\Wspb$; see Proposition~\ref{PROP:Gibbs}.

\begin{definition}[Good solution]
Fix $\kappa\geq 1$. We say that $q:\R\to L^2_1$ is a \emph{good solution} of the $H_\kappa$ flow if it is continuous,
\begin{align}\label{1592}
q(t,x)=q(0,x+4\kappa^2 t) +8 \kappa^4 \int_{0}^{t} g_+(x+4\kappa^2(t-t'); \kappa, q(t'))\,dt,
\end{align}
for each $t\in\R$ and almost every $x$, and for each $T>0$ there is a $C_T>0$ so that
\begin{align}\label{1592'}
\int_{-T}^T \bigl\| \chi_\ell \LaxR(\kappa;q(t)) \bigr\|_{L^2\to L^\infty}^2 + \bigl\| \chi_\ell q(t) \bigr\|_{L^\infty}^2 \,dt \leq C_T \log^{C_T} (\ell) 
\end{align}
for all $\ell\geq 2$.
\end{definition}

Notice that good solutions are, by definition, \emph{global} solutions.  Also, if $q(t)$ is a good solution, then so too is $t\mapsto q(t+\tau)$ for any $\tau\in\R$.  As $q$ is $L^2_1$-continuous, \eqref{pbd'} guarantees that both sides of \eqref{1592} can equivalently be interpreted as elements of the space $L^2_3$ (rather than merely almost everywhere).

\begin{theorem}\label{T:1601}
Given $\kappa\in2^\N$, there is a subset $\mathcal G_\kappa\subseteq L^2_1(\R)$ of full Gibbs measure and a measure-preserving group $\Phi_\kappa:\R\times\mathcal G_\kappa\to\mathcal G_\kappa$ with the following properties:\\
\slp{a} For each $q^*\in\mathcal G_\kappa$, the orbit $q(t):=\Phi_\kappa(t,q^*)$ is a good solution of the $H_\kappa$ flow.\\
\slp{b} The diagonal Green functions satisfy \eqref{+Hkflow}, \eqref{-Hkflow}, and \eqref{gHkflow}.\\
\slp{c} For every triple \eqref{spb}, there is an $\alpha>0$ so that
\begin{align}\label{1610}
\sup_\kappa  \ \E\Bigl\{ \|\Phi_\kappa(t,q_\infty^0) \|_{C_{t}^{\alpha}([-T,T];\Wspb)}^{r} \Bigr\}<\infty
\end{align}
for each $1\leq r<\infty$ and every $T\in 2^\N$.
\end{theorem}

The proof of this theorem is given at the end of the section and rests on two principal pillars: (i) proving convergence of finite-volume solutions  and (ii) showing that good solutions are unique.

Uniqueness will be addressed in Proposition~\ref{P:unique}.  It is crucial not only for ensuring that $\Phi_\kappa$ is well-defined, but also for verifying the group property.  

We do not include a commutativity statement in Theorem~\ref{T:1601} because the precise manifestation of this property that we need for the developments in Section~\ref{S:6} has a rather technical character; see Lemma~\ref{L:commune}.  In order to prove this lemma, we must add an extra layer of complexity to the earlier parts of this section.

Our next three results relate to the existence part of the argument.  The following has a rather obvious connection to the property \eqref{1592'}.  It will also play an important role in constructing infinite-volume solutions as limits of finite-volume solutions, a topic taken up in the subsequent Lemma~\ref{L:1603} and Proposition~\ref{P:L limit}.  

\begin{lemma}\label{L:1613} 
For $1\leq r<\infty$,
\begin{align}\label{chi R}
\E\Bigl\{  \bigl\| \chi_\ell \LaxR(\kappa;q_L^0) \bigr\|_{L^2\to L^\infty}^r
	+ \bigl\| \LaxR(\kappa;q_L^0) \chi_\ell  \bigr\|_{L^1\to L^2}^r \Bigr\} \lesssim_r \kappa^{-r/2} \log^{r/3} ( \ell\wedge L )
\end{align}
uniformly for $L\in 2^\N\cup\{\infty\}$ and $2\leq\ell\leq \infty$.
\end{lemma}

\begin{proof}
We will treat the $L^2\to L^\infty$ bound.  Notice that for any $q$,
\begin{align}\label{chi R R chi}
\| \chi_\ell \LaxR(\kappa;q) \|_{L^2\to L^\infty} =  \| \LaxR(\kappa;q) \chi_\ell \|_{L^1\to L^2},
\end{align}
as may be seen by taking adjoints and using \eqref{k symmetry}.

The explicit formula \eqref{G0} and Cauchy--Schwarz shows that
\begin{align*}
\| \chi_\ell \LaxR_0(\kappa) \|_{L^2\to L^\infty} \lesssim \kappa^{-1/2}
\end{align*}
and that for $l\geq 2\ell$,
\begin{align*}
\| \chi_\ell \LaxR_0(\kappa) [\chi_{2l}-\chi_l] \|_{L^2\to L^\infty} \lesssim \kappa^{-1/2} e^{-\kappa|l-\ell|} .
\end{align*}
Combining this with the identity
$$
\chi_\ell\LaxR = \chi_\ell\LaxR_0 - \chi_\ell \LaxR_0 Q \LaxR = \chi_\ell\LaxR_0 - \chi_\ell \LaxR_0 \chi_\ell Q \LaxR
	- \sum \chi_\ell \LaxR_0(\kappa) [\chi_{2l}-\chi_l] Q \LaxR,
$$
where the sum is taken over $l\in \ell 2^\N$, we deduce that
\begin{align*}
\| \chi_\ell \LaxR(\kappa) \|_{L^2\to L^\infty}  \lesssim \kappa^{-1/2} \Bigl[ 1 +  \| \chi_\ell q \|_{L^\infty}
	+ \sum  e^{-\kappa|l-\ell|} \| \chi_l q \|_{L^\infty} \Bigr] .
\end{align*}

This analysis has been deterministic.  The lemma now follows by invoking \eqref{log13} and noting that, by periodicity, $q^0_L$ achieves its global maximum already on the interval $[-L,L]$.
\end{proof}

\begin{lemma}\label{L:1603}
Fix $\kappa\geq 1$, $\alpha\in(0,1)$, and $\eta>0$.  Suppose we have processes $q_L^1(x)$, for each $L\in 2^\N$, that are Gibbs distributed on $\T_L$ and are coupled in such a manner that
\begin{align}\label{1606}
\sup_{L\in 2^{\N}} e^{\eta L^{\alpha}} \bigl\|e^{-\langle x\rangle^{\alpha}}[ q_{2L}^1-q_{L}^1]\bigr\|_{L^2} <\infty
\end{align}
almost surely. Let $q_L(t)$ be defined as the \slp{global and unique} solution to the $H_\kappa$ flow with initial data $q^1_L$.  Then, there exists $\eps>0$ so that for any $T>0$,
\begin{align}\label{1607}
\sum_{L\in 2^{\N}}  e^{\eps L^{\alpha}}  \sup_{ t\in [-T,T]} \bigl\|e^{-\langle x\rangle^{\alpha}}[ q_{2L}(t)-q_{L}(t)]\bigr\|_{L^2} <\infty
\end{align}
almost surely.
\end{lemma}

As we will observe during the Proof of Theorem~\ref{T:1601}, the fact that $q_L^0$ satisfy \eqref{1562} ensures that they also satisfy \eqref{1606}.  However, we need the greater generality presented here in order to handle questions of commutativity.  Notice also that fixing some time $t\in\R$, the conclusion \eqref{1607} can be used to see that $q_{L}(t)$ verify the hypothesis \eqref{1606}, albeit for a different $\eta$.  In this way, $q_{L}(t)$ can be used as initial data for the $H_\vk$ flow (with $\vk\neq \kappa$). 

\begin{proof}
Setting $\eps =\frac14 \eta$, our overall ambition is to control 
\begin{align*}
d_{L}(t) =  \|e^{-\langle x\rangle^{\alpha}}[ q_{2L}(t)-q_{L}(t)]\|_{L^2}^{2}. 
\end{align*}
over the interval $[-T,T]$ via Gronwall's inequality.  From our hypothesis \eqref{1606},
\begin{align}\label{dL0}
d_L(0) \leq C e^{-8\eps L^\alpha\!} \qtq{uniformly for}  L \in 2^\N
\end{align}
for some random but almost surely finite $C$.

In order to run our Gronwall argument, we will need to control the difference between $g^L_+(t,x):=g_+(x;\kappa,q_L(t))$ and $g_+^{2L}$ in terms of $q_{2L}(t)-q_L(t)$.  We pause to make the necessary preparations; these culminate in \eqref{1658}.

For generic potentials $q$ and $\tilde q$ and $F:\R\to\C^{2\times2}$, we have 
\begin{align}\label{1655}
\tr\bigl\{e^{-\langle x\rangle^{\alpha}} F [\LaxR-\tilde\LaxR]\bigr\}
=&{} - \tr\bigl\{ F e^{-\langle x\rangle^{\alpha}} \LaxR_0 e^{\langle x\rangle^{\alpha}} \cdot e^{-\langle x\rangle^{\alpha}} [Q-\tilde Q] \tilde\LaxR\bigr\} \\
&{} + \tr\bigl\{  F \LaxR_0 \cdot Q \cdot \LaxR [Q-\tilde Q]e^{-\langle x\rangle^{\alpha}} \cdot
		e^{\langle x\rangle^{\alpha}}  \tilde\LaxR \;\!e^{-\langle x\rangle^{\alpha}}  \bigr\}, \notag
\end{align}
by the resolvent identity.  Our goal here is to obtain $L^2$ bounds by maximizing over $F\in L^2$.  To this end, we employ a variety of estimates, including,
\begin{gather}\label{1656a}
\bigl\|  F e^{-\langle x\rangle^{\alpha}} \LaxR_0 e^{\langle x\rangle^{\alpha}} \bigr\|_{\HS} \lesssim \kappa^{-1/2} \| F \|_{L^2}
	\qtq{and}  \bigl\|  F \LaxR_0 \bigr\|_{\HS} \lesssim \kappa^{-1/2} \| F \|_{L^2},
\end{gather}
which are easily deduced from the explicit formula \eqref{G0}.  We also use
\begin{align}\label{1656b}
\bigl\|  \LaxR [Q-\tilde Q]e^{-\langle x\rangle^{\alpha}} \bigr\|_{\HS}
		&\leq \bigl\| [Q-\tilde Q]e^{-\langle x\rangle^{\alpha}} \bigr\|_{L^2} \sup _y \bigl\|  G(x,y;\kappa) \bigr\|_{L^2_x}   \\
		&\leq \bigl\| e^{-\langle x\rangle^{\alpha}}[q-\tilde q] \bigr\|_{L^2} \| \LaxR \|_{L^1\to L^2}, \notag
\end{align}
as well as the parallel
\begin{align}\label{1657}
\bigl\|  e^{-\langle x\rangle^{\alpha}} [Q-\tilde Q] \tilde\LaxR \bigr\|_{\HS}
	\leq \bigl\| e^{-\langle x\rangle^{\alpha}} [q-\tilde q] \bigr\|_{L^2} \| \tilde\LaxR \|_{L^2\to L^\infty}.
\end{align} 
Returning to \eqref{1655} and using \eqref{multcomm} and the estimates just described, we find that
\begin{align}\label{1658}
\| e^{-\langle x\rangle^{\alpha}} [g_\pm - \tilde g_\pm]\|_{L^2} 
&\lesssim \Bigl[\kappa^{-\frac12}  \| \tilde\LaxR \|_{L^2\to L^\infty} + \kappa^{-\frac32} \| q \|_{L^\infty} \| \LaxR \|_{L^1\to L^2} \Bigr] \\
&\qquad\times \bigl\| e^{-\langle x\rangle^{\alpha}} [q-\tilde q] \bigr\|_{L^2} , \notag
\end{align}
for any pair of potentials $q$ and $\tilde q$.

From the definition \eqref{Hkflow} of the $H_\kappa$ flow and integration by parts, we have
\begin{align*}
\tfrac{d\ }{dt}  d_{L}(t) =& 8\kappa^{2} \alpha \bigl\langle \tfrac{x}{\langle x\rangle^{2-\alpha}}  e^{-\langle x\rangle^{\alpha}} [q_{2L}-q_{L}],
	 e^{-\langle x\rangle^{\alpha}} [q_{2L}-q_{L}]  \bigr\rangle \\
& +16\kappa^{4} \bigl\langle  e^{-\langle x\rangle^{\alpha}}[g_+^{2L}(\kappa)-g_+^{L}(\kappa)],
	e^{-\langle x\rangle^{\alpha}} [q_{2L}-q_{L}]  \bigr\rangle .
\end{align*}
Using \eqref{1658} and Gronwall's inequality, we deduce that
\begin{align}\label{1672}
\sup_{|t|\leq T}  d_{L}(t) \lesssim e^{X_L^T} d_L(0) 
\end{align}
where
\begin{align*}
 X_L^T \lesssim_\kappa T + \int_{-T}^T \Bigl[ \| \LaxR(\kappa;q_{2L}(t)) \|_{L^2\to L^\infty}
 	+ \| q_L(t) \|_{L^\infty} \| \LaxR(\kappa;q_{L}(t)) \|_{L^1\to L^2} \Bigr] \,dt.
\end{align*}

By Theorem~\ref{THM:HKinv}, $q_L(t)$ and $q_{2L}(t)$ are Gibbs distributed at all times.  Thus, we may combine \eqref{chi R} and \eqref{log13} to find that
\begin{align*}
 \E\bigl\{ X_L^T \bigr\} \lesssim_\kappa T \log^{2/3}( 2L )
 	\qtq{and thence that} \E\Bigl\{\sup_L \ \log^{-2}(L)\, X_L^T \Bigr\}< \infty.
\end{align*}
From this and \eqref{dL0}, our bound \eqref{1672} guarantees that
\begin{align*}
\sup_{|t|\leq T}  e^{2\eps L^\alpha} d_{L}(t) \leq C e^{2\eps L^\alpha} e^{C\log^2(L) - 8\eps L^\alpha} \qtq{uniformly for}  L \in 2^\N
\end{align*}
for some random but almost surely finite $C$.  Even after taking the square-root, this bound is summable over dyadic $L$ and so \eqref{1607} follows.
\end{proof}

Evidently, Lemma~\ref{L:1603} gives a very mild form of convergence of the $q_L(t)$ as $t\to\infty$.  We now substantially upgrade this:

\begin{proposition}[Infinite-volume limits]\label{P:L limit} 
Let $\eta>0$ be given, along with Gibbs-distributed processes $q_L^1(x)$, $L\in 2^\N$, that satisfy \eqref{1606}.  For each $\kappa\in 2^\N$, let $q_L(t;\kappa)$ denote the \slp{global and unique} solution of the $H_\kappa$ flow with initial data $q_L^1$.  Then, for every $\kappa,T\in 2^\N$, the sequence $q_L(t;\kappa)$ converges in $C_t([-T,T];L^2_1)$ almost surely.

The processes $q_\infty(t;\kappa)$ defined as these limits are Gibbs distributed at every $t\in\R$ and are almost surely good solutions of the $H_\kappa$ flow.  Moreover, for each triple \eqref{spb}, there is an $\alpha=\alpha(s,p,b)>0$ so that
\begin{align}\label{1776}
\sup_{\kappa\geq 1} \  \E \Bigl\{ \bigl\| q_\infty(t;\kappa) \|_{C^\alpha_t([-T,T];\Wspb)}^r  \Bigr\}  < \infty \qtq{for every} T\in 2^\N.
\end{align} 
\end{proposition}

\begin{proof}
Fix $\kappa,T\in 2^\N$. In view of \eqref{bounds kL}, we know that
\begin{align*}
\E\Bigl\{ \sum_L \log^{-3}(L)  \ \sup_{|t|\leq T}\  \| \langle x \rangle^{-\frac23}[q_{2L} - q_L] \|_{L^2} \Bigr\} < \infty .
\end{align*}
Combining this with the elementary
\begin{align*}
\langle x \rangle^{-1} \lesssim L \exp\{- \langle x\rangle^{1/9}\} + \log^{-3}(L) \langle x \rangle^{-\frac23}
	\qtq{uniformly for} L\in 2^\N
\end{align*}
and Lemma~\ref{L:1603}, we deduce that for each $T>0$,
\begin{align*}
\sum_L \ \sup_{|t|\leq T}\  \| q_{2L} - q_L \|_{L^2_1} < \infty \quad\text{almost surely.}
\end{align*}
This guarantees the almost sure convergence of $q_L(t;\kappa)$ in $C_t([-T,T];L^2_1)$.

As there are only countably many choices, we immediately deduce simultaneous convergence for \emph{all} $\kappa,T\in 2^\N$ also holds with probability one. 

By Fatou's Lemma and \eqref{bounds kL}, we know that
\begin{align}\label{1900}
\sup_{\kappa\geq 1} \  \E \Bigl\{ \; \liminf_{L\to\infty} \; \bigl\| q_L(t,x;\kappa) \|_{C^\alpha_t([-T,T];\Wspb)}^r  \Bigr\}  < \infty
\end{align} 
for every $1\leq r<\infty$, every $T\in 2^N$, and every triple \eqref{spb}.  Interpolating with convergence in $C_t([-T,T];L^2_1)$, we deduce convergence in $C_t([-T,T];\Wspb)$.  Then applying \eqref{1900} once again,  we obtain \eqref{1776}.

Recall that the finite-volume Gibbs distributions converge to the infinite volume Gibbs distribution as $L\to\infty$.  Moreover, from Theorem~\ref{THM:HKinv}, we know that $q_L(t;\kappa)$ are Gibbs distributed on $\T_L$ for every $\kappa,L\in 2^\N$.    As we have shown that these random variables converge almost surely in $L^2_1$ to $q_\infty(t;\kappa)$, it follows that $q_\infty(t;\kappa)$ is Gibbs distributed in infinite volume.

It remains to verify that $q_\infty(t;\kappa)$ is almost surely a good solution.  The finite-volume solutions $q_L(t;\kappa)$ satisfy the integral equation \eqref{1592} by construction.  Convergence in $C_t([-T,T];L^2_1)$ and \eqref{p lip b} extend this to the infinite-volume case.

As $q_\infty(t)$ is always Gibbs distributed, Proposition~\ref{PROP:Gibbs} and Lemma~\ref{L:1613} show that
\begin{align*}
\E \biggl\{ \int_{-T}^T \bigl\| \chi_\ell \LaxR(\kappa;q_\infty(t)) \bigr\|_{L^2\to L^\infty}^2 + \bigl\| \chi_\ell q_\infty(t) \bigr\|_{L^\infty}^2 \,dt \biggr\} 
	\lesssim_\kappa T \log^{2/3} (\ell) 
\end{align*}
uniformly for $\ell\in 2^\N$.  It follows that
$$
\sum_\ell \frac{1}{\log^2(\ell)} \int_{-T}^T \bigl\| \chi_\ell \LaxR(\kappa;q_\infty(t)) \bigr\|_{L^2\to L^\infty}^2
	+ \bigl\| \chi_\ell q_\infty(t) \bigr\|_{L^\infty}^2  \ dt
$$
is almost surely finite (it has finite expectation) and thus \eqref{1592'} holds almost surely for each (and thence every) $T\in 2^\N$.
\end{proof}

Let us now demonstrate the uniqueness of good solutions, building on ideas introduced already in Lemma~\ref{L:1603}.

\begin{proposition}[Uniqueness of good solutions]\label{P:unique}  Fix $\kappa\geq 1$.  For each initial data in $L^2_1$ there is at most one good solution to the $H_\kappa$ flow.
\end{proposition}

\begin{proof}
Let $q$ and $\tilde q$ denote two good solutions with the same initial data.  We seek to control 
\begin{align*}
d(t) =  \|e^{-\langle x\rangle^{\alpha}}[ q(t)-\tilde q(t)]\|_{L^2}^{2} 
\end{align*}
over any interval $[-T,T]$ via Gronwall's inequality.  Evidently, $d(0)=0$.

To control the difference between $g_+(t,x):=g_+(x;\kappa,q(t))$ and $\tilde g_+:=g_+(x;\kappa,\tilde q(t))$, we will generalize some of the arguments used to prove Lemma~\ref{L:1603}.  Recall \eqref{1655}:
\begin{align*}
\tr\bigl\{e^{-\langle x\rangle^{\alpha}} F [\LaxR-\tilde\LaxR]\bigr\}
={}& - \tr\bigl\{ F e^{-\langle x\rangle^{\alpha}} \LaxR_0 e^{\langle x\rangle^{\alpha}} \cdot e^{-\langle x\rangle^{\alpha}}
		[Q-\tilde Q] \tilde\LaxR\bigr\} \\
& + \tr\bigl\{  F \LaxR_0 \cdot Q \LaxR [Q-\tilde Q]e^{-\langle x\rangle^{\alpha}} \cdot
		e^{\langle x\rangle^{\alpha}}  \tilde\LaxR \;\!e^{-\langle x\rangle^{\alpha}}  \bigr\} .
\end{align*}

To estimate the right-hand side, we employ \eqref{1656a} and \eqref{multcomm} exactly as before, but need new estimates for the remaining two operators.  In order to obtain such estimates, we will use the following partition of unity:
\begin{align}\label{1952}
1 = \chi_L + \sum [\chi_{2\ell}-\chi_\ell] \qtq{where summation is over} \ell\in L 2^\N .
\end{align}  
Inserting this immediately before $\tilde\LaxR$ we see that
\begin{align}\label{1953}
\bigl\|  e^{-\langle x\rangle^{\alpha}} [Q-\tilde Q] \tilde\LaxR \bigr\|_{\HS}
	&\lesssim \bigl\| e^{-\langle x\rangle^{\alpha}} [q-\tilde q] \bigr\|_{L^2} \| \chi_L \tilde\LaxR \|_{L^2\to L^\infty} \\
& \qquad + \Bigl( \|q\|_{L^2_1} + \|\tilde q\|_{L^2_1}\Bigr) \sum \langle 2\ell\rangle e^{-\langle\ell\rangle^\alpha}
	\| \chi_\ell \tilde\LaxR \|_{L^2\to L^\infty}, \notag
\end{align}  
where the sum is taken over $\ell\in L 2^\N$.

Next we consider the operator $Q \LaxR [Q-\tilde Q]e^{-\langle x\rangle^{\alpha}}$, which we also need to bound in Hilbert--Schmidt norm.
We expand this operator by placing the partition of unity \eqref{1952} to the right of $\LaxR$ and to its left, we place the analogue starting at length scale $2L$.  We estimate the first such term thus:
\begin{align}\label{1954}
\bigl\| Q \chi_{2L} \LaxR \chi_L [Q-\tilde Q]e^{-\langle x\rangle^{\alpha}} \bigr\|_\HS 
&\lesssim \bigl\| \chi_{2L} q \bigr\|_{L^\infty} \bigl\| \LaxR \chi_L \bigr\|_{L^1\to L^2}
	\bigl\| e^{-\langle x\rangle^{\alpha}} [q-\tilde q] \bigr\|_{L^2} .
\end{align}

Using \eqref{CT G off}, we note that for $\ell'\in 2L2^\N$, we have
\begin{align*}
\bigl\| \langle x\rangle [\chi_{2\ell'}-\chi_{\ell'}] \LaxR \chi_L \bigr\|_{L^1\to L^\infty}
& \lesssim \Bigl( 1 + \| q \|_{L^2_1}^2 \Bigr)	\langle\ell'\rangle^2 \langle L \rangle e^{-\frac1{12} \ell'} .
\end{align*}
In this way, we deduce that
\begin{align*}
\sum_{\ell'\geq 2L} \bigl\| & Q [\chi_{2\ell'}-\chi_{\ell'}] \LaxR \chi_L [Q-\tilde Q]e^{-\langle x\rangle^{\alpha}} \bigr\|_\HS \notag \\
& \lesssim \| q \|_{L^2_1} \Bigl( 1 + \| q \|_{L^2_1}^2 \Bigr)	\bigl\| e^{-\langle x\rangle^{\alpha}} [q-\tilde q] \bigr\|_{L^2}
	\sum_{\ell'\geq 2L} \langle\ell'\rangle^2\langle L \rangle e^{-\frac1{12} \ell'} \\
&\lesssim \Bigl( 1 + \| q \|_{L^2_1}^2 + \| \tilde q \|_{L^2_1}^2 \Bigr)^2 e^{-\frac{1}{100} L^\alpha}. \notag
\end{align*}

Using \eqref{CT G off} in a similar manner, we find 
\begin{align*}
\sum_{\ell\geq L} \bigl\| Q \chi_{2L} & \LaxR [\chi_{2\ell}-\chi_{\ell}] [Q-\tilde Q]e^{-\langle x\rangle^{\alpha}} \bigr\|_\HS  \notag \\
&\lesssim \Bigl( 1 + \| q \|_{L^2_1}^2 + \| \tilde q \|_{L^2_1}^2 \Bigr)^2  \sum_{\ell\geq L}  \langle L \rangle^2 \langle\ell\rangle^2 
	  e^{- \langle \ell\rangle^{\alpha}} \\
&\lesssim \Bigl( 1 + \| q \|_{L^2_1}^2 + \| \tilde q \|_{L^2_1}^2 \Bigr)^2 e^{-\frac{1}{100} L^\alpha}. \notag
\end{align*}

For the last part of our decomposition of $Q \LaxR [Q-\tilde Q]e^{-\langle x\rangle^{\alpha}}$, we first use \eqref{CT G off} to see that
\begin{align*}
\bigl\| Q [\chi_{2\ell'}-\chi_{\ell'}] & \LaxR [\chi_{2\ell}-\chi_{\ell}] [Q-\tilde Q]e^{-\langle x\rangle^{\alpha}} \bigr\|_\HS  \\
&\lesssim \Bigl( 1 + \| q \|_{L^2_1}^2 + \| \tilde q \|_{L^2_1}^2 \Bigr)^2 \langle\ell\rangle^2 \langle\ell'\rangle^2 
	e^{ -\frac1{6} d(\ell,\ell')}  e^{-\frac12 \langle \ell\rangle^{\alpha}}  \notag
\end{align*}
where $d(\ell,\ell'):=\dist(\supp(\chi_{2\ell}-\chi_\ell),\supp(\chi_{2\ell}-\chi_{\ell'}))$.  This must then be summed over all $\ell'\geq 2L$ and $\ell\geq L$, which yields
\begin{align*}
&\sum_{\ell,\ell'} \bigl\| Q [\chi_{2\ell'}-\chi_{\ell'}] \LaxR [\chi_{2\ell}-\chi_{\ell}] [Q-\tilde Q]e^{-\langle x\rangle^{\alpha}} \bigr\|_\HS
\lesssim \Bigl( 1 + \| q \|_{L^2_1}^2 + \| \tilde q \|_{L^2_1}^2 \Bigr)^2e^{-\frac{1}{100} L^\alpha}  .
\end{align*}

Combining all these estimates as in the proof of Lemma~\ref{L:1603}, we find that
\begin{align*}
\|e^{-\langle x\rangle^{\alpha}}[ g_+(t)-\tilde g_+(t)]\|_{L^2}
\lesssim X_L(t) \bigl\| e^{-\langle x\rangle^{\alpha}} [q(t)-\tilde q(t)] \bigr\|_{L^2} + Y_L(t)  e^{-\frac1{100} L^\alpha}   
\end{align*}
where, after using \eqref{chi R R chi} on \eqref{1954}, we may choose
\begin{gather*}
X_L(t) = \| \chi_L \LaxR(\kappa;\tilde q(t)) \|_{L^2\to L^\infty}
	+ \|\chi_{2L} q(t)\|_{L^\infty}^2 + \| \chi_L \LaxR(\kappa;q(t)) \|_{L^2\to L^\infty}^2 \\
Y_L(t) = \Bigl( 1 + \|q(t)\|_{L^2_1} + \|\tilde q(t)\|_{L^2_1}\Bigr)^2 \Bigl( 1
	+ \sum \ell e^{-\frac1{100}\ell^\alpha}  \| \chi_\ell \LaxR(\kappa;\tilde q(t)) \|_{L^2\to L^\infty}  \Bigr).
\end{gather*}

The fact that $q$ and $\tilde q$ are both good solutions guarantees that
$$
8\kappa^2 T  +  16\kappa^4 \int_{-T}^T X_L(t) + Y_L(t) \,dt \leq C \log^C(L)
$$
for some $C$ depending on $q$, $\tilde q$, $\kappa$, and $T$, but not on $L\geq 2$.  Thus, applying Gronwall's inequality as in the proof of Lemma~\ref{L:1603},
we find that
$$
\sup_{|t|\leq T} d(t) \leq C \log^C(L) e^{ C \log^C(L) - \frac1{100} L^\alpha}.
$$
Sending $L\to\infty$ completes the proof.
\end{proof}

\begin{proof}[Proof of Theorem~\ref{T:1601}]
Fix $\kappa\in 2^\N$.  For each $L\in 2^\N$, let us write $\Phi_{\kappa,L}$ for the one-parameter group of homeomorphisms on $L^2(\T_L)$ defined by the $H_\kappa$ flow.  Recall that each of these preserves Gibbs measure on $\T_L$.

We define $\mathcal G_\kappa$ as the set of those $q^*_\infty\in L^2_1$ for which there is a sequence $q^*_L\in L^2(\T_L)$ meeting both of the following criteria: (i) For all $T\in 2^\N$, the sequence 
\begin{align}
\Phi_{\kappa,L}(t,q^*_L) \qtq{converges in}  C([-T,T];L^2_1)
\end{align}
as $L\to\infty$; (ii) The limit function is a good solution of the $H_\kappa$ flow. 
(We use the $*$ superscript to distinguish these initial data from the random variables $q^0_L$).

We then define $\Phi_\kappa:\R\times\mathcal G_\kappa\to L^2_1$ as the limiting function.  Proposition~\ref{P:unique} guarantees that this function is well-defined, which is to say, the trajectory is uniquely determined by the initial data (irrespective of which sequence lead to it).

In fact, $\Phi_\kappa$ takes values in $\mathcal G_\kappa$.  To see this, we note that all our criteria are invariant under time translation.  In a similar vein, we see that $\Phi_\kappa$ inherits the group property from the fact that it is true in finite volume.

The Lipschitz dependence demonstrated in Proposition~\ref{P:G Lip} shows that the evolution equations for the diagonal Green's functions presented in Proposition~\ref{P:HK well} carry over to the infinite volume setting.  This verifies part (b).

Turning to probabilistic matters, for $L\in 2^\N\cup\{\infty\}$, let $q_L^0$ be Gibbs distributed and suppose these are coupled together in the manner of Proposition~\ref{P:coupled q_L}.

Let $E_L$ denote the event that $q_L^0\equiv q_{2L}^0$ throughout $[-L/2,L/2]$.  By the coupling condition \eqref{E:coupled q_L}, we know that $\PP(E_L^c) \lesssim e^{-cL}$.  On the other hand, \eqref{Wsp moments} guarantees that
\begin{align*}
\E\Bigl\{ \bigl\|e^{-\langle x\rangle^{\alpha}} q_{L}^0 \bigr\|_{L^2}^2 \Bigr\}  \lesssim 1
	\qtq{and} \E\Bigl\{ \bigl\|e^{-\langle x\rangle^{\alpha}} [1-\chi_{L/2}]q_{L}^0 \bigr\|_{L^2}^2 \Bigr\}  \lesssim_\alpha e^{-L^\alpha\!/2},
\end{align*}
both uniformly in $L$.  By Cauchy--Schwarz, we then see that
\begin{align}\label{2052}
\E\Bigl\{ \sum_L e^{L^\alpha\!/8} \bigl\|e^{-\langle x\rangle^{\alpha}} [ q_{2L}^0-q_{L}^0] \bigr\|_{L^2} \Bigr\}
	\lesssim_\alpha \sum_L \bigl[e^{-c L/2} + e^{-L^\alpha\!/4}\bigr] e^{L^\alpha\!/8} < \infty
\end{align}
and consequently, the processes $q_L^0$ satisfy \eqref{1606} with $\eta=\frac18$.  

Let $q_L(t;\kappa)$ denote the (unique and global) solution to the $H_\kappa$ flow with initial data $q_L^0$.  According to Proposition~\ref{P:L limit}, there is an event $E$ of full probability on which all of the following happen: (i) The sequence $q_L(t;\kappa)$ converges in $C_t^{ } L^2_1$ uniformly on compact time intervals;  (ii) The limit $q_\infty(t;\kappa)$ is a good solution of the $H_\kappa$ flow;  (iii)  For every triple \eqref{spb}, the curves $q_\infty(t;\kappa)$ are $\Wspb$ continuous.

There is one small wrinkle in this argument: there are uncountably many triples \eqref{spb} so one cannot just take an intersection based on \eqref{1776}. However, by the imbedding properties of these spaces, we need only consider rational triples; this ensures the intersection has full measure.

Let $\mathcal G^0$ denote the image under $q_\infty^0$ of the event $E$.  Evidently, $\mathcal G^0 \subset \mathcal G_\kappa$ and $\mathcal G^0$ has full measure.  Thus $\mathcal G_\kappa$ has full measure.

By uniqueness of good solutions, $q_\infty(t;\kappa)\equiv \Phi_\kappa(t,q^0_\infty)$ on $E$.  In this way, we see that \eqref{1610} is a consequence of \eqref{1776} and that the measure-preserving property of $\Phi_\kappa$ follows from the fact that $q_\infty(t;\kappa)$ is Gibbs distributed for all $t\in \R$.
\end{proof}

\begin{lemma}\label{L:commune}
Fix $\kappa\in 2^\N$ and let $q_\infty^0$ be Gibbs distributed.  Then there is a random process $q:\R\to L^2_1$ that is almost surely continuous and satisfies:\\
\slp{a} For each $t\in\R$, $q(t) = \Phi_{2\kappa}\bigl(t,\Phi_\kappa(-t,q^0_\infty)\bigr)$, almost surely.\\
\slp{b} $q(t)$ is Gibbs distributed at each $t\in\R$.\\ 
\slp{c} $q(t)$ is a solution of the $H_{2\kappa}-H_\kappa$ flow in the sense that
\begin{gather}\label{ccmm}
q(t)=e^{12t\kappa^2\partial_x}q^0_\infty +8 \kappa^4 \int_{0}^{t} e^{12t\kappa^2\partial_x}[16g_+(2\kappa, q(s)) - g_+(\kappa, q(s))] ds. 
\end{gather}
\end{lemma}

\begin{proof}
By augmenting the probability space, if necessary, we may introduce Gibbs distributed processes $q_L^0$ for each $L\in 2^\N\cup\{\infty\}$ that are  coupled to $q_\infty^0$ as in Proposition~\ref{P:coupled q_L}.

The natural finite-volume analogue of \eqref{ccmm} is 
\begin{gather}\label{ccmm'}
q_L(t)=e^{12t\kappa^2\partial_x}q^0_L +8 \kappa^4 \int_{0}^{t} e^{12t\kappa^2\partial_x}[16g_+(2\kappa, q_L(s)) - g_+(\kappa, q_L(s))] ds. 
\end{gather}
For almost every sample $q^0_L$, this is easily solved (uniquely and globally) by the same contraction mapping argument used to prove Proposition~\ref{P:HK well}; indeed, one just employs the Lipschitz estimates of Proposition~\ref{P:PerioG}.  Clearly, $q_L:\R\to L^2_1$ is almost surely continuous; moreover, by commutativity of flows,
\begin{gather}\label{ccmm''}
q_L(t)= \Phi_{2\kappa,L}(t) \circ \Phi_{\kappa,L}(-t, q_L^0).
\end{gather}
This also shows, by Theorem~\ref{THM:HKinv}, that $q_L(t)$ is Gibbs distributed for every $t\in\R$.

One may then show that $q_L(t)$ converges, almost surely, in $C_t([-T,T];L^2_1)$ by repeating the first paragraph of the proof of Proposition~\ref{P:L limit}.  This requires the following mild form of \eqref{bounds kL}: 
\begin{align}\label{bounds kL''}
 \sup_{L\geq 1} \ \E \Bigl\{  \bigl\| \langle x\rangle^{-\frac23} \, q_L(t,x) \|_{C^\alpha([-T,T];L^2(\T_L))}  \Bigr\}  < \infty .
\end{align} 
All the details necessary to establish \eqref{bounds kL''} appear already in the Proof of Theorem~\ref{THM:HKinv}(c); indeed, one may just estimate the vector field for the $H_{2\kappa}-H_\kappa$ flow via the triangle inequality and the estimates given there for a single $H_\kappa$ flow.

We define $q(t)$ as the (almost sure) limit of $q_L(t)$.  Continuity, as well as claims (b) and (c) follow immediately.  It remains to verify (a).  

Recall that $\mathcal G_\kappa$ and $\Phi_\kappa$ were defined via finite-volume limits.  In particular, after verifying in \eqref{2052} that $q_L^0$ satisfy \eqref{1606} with $\eta=\frac18$, we applied Proposition~\ref{P:L limit} to deduce that
$$
q_\infty^0 \in \mathcal G_\kappa \qtq{and} \lim_{L\to\infty} \Phi_{\kappa,L}(-t,q_L^0)=\Phi_\kappa(-t,q_\infty^0)  \quad\text{in $L^2_1$ sense}
$$
both hold with probability one.  To derive (a), we apply the same argument to the sequence $\Phi_{\kappa,L}(-t,q_L^0)$, which converges to $\Phi_\kappa(-t,q_\infty^0)$:  By \eqref{2052} and Lemma~\ref{L:1603}, the sequence $\Phi_{\kappa,L}(-t,q_L^0)$ also satisfies \eqref{1606}, with some $\eta>0$.  Thus, applying Proposition~\ref{P:L limit} in the same manner, we conclude that
$$
\Phi_\kappa(-t,q_\infty^0) \in \mathcal G_{2\kappa} \qtq{and}
	q(t) = \lim_{L\to\infty} q_L(t) = \Phi_{2\kappa}(t) \circ \Phi_\kappa(-t,q_\infty^0)
$$
both hold with probability one.
\end{proof}

\section{Dynamics and invariance for mKdV in infinite volume}\label{S:6}

As we will be working solely in infinite volume, we may revert to a much simpler probabilistic set-up, with just one process $q^0_\infty$ which is Gibbs distributed on $\R$.

Dynamics with initial data $q_\infty^0$ will be constructed as the $\kappa\to\infty$ limit of the infinite-volume $H_\kappa$ flow maps $\Phi_\kappa$ described by Theorem~\ref{T:1601}.  As there, we shall only consider $\kappa\in 2^\N$:

\begin{theorem}[Dynamics of the infinite-volume Gibbs state]\label{T:1936}\leavevmode\\
With probability one, the sequence of trajectories $\Phi_\kappa(t,q_\infty^0)$ converges, as $\kappa\to\infty$, in $C_t([-T,T];\Wspb)$ for every $T\in 2^\N$ and every triple \eqref{spb}.  This almost sure limit defines a measure-preserving group, whose orbits $q(t):=\Phi(t,q^0_\infty)$ satisfy\\
\slp{a} Almost surely, $q(t)$ is a distributional solution to \eqref{mkdv}.\\
\slp{b} Almost surely, $q(t)$ is a green solution to \eqref{mkdv}; see Definition~\ref{D:mkdv green}.\\
\slp{c} For every triple \eqref{spb}, there is an $\alpha>0$ so that
\begin{align}\label{1910}
 \E\Bigl\{ \| q(t) \|_{C_{t}^{\alpha}([-T,T];\Wspb)}^{r} \Bigr\}<\infty
\end{align}
for each $1\leq r<\infty$ and every $T\in 2^\N$.\\
\end{theorem}

The group property of the transformations $\Phi$ constructed here is this:
\begin{equation}\label{aegp}
\Phi(t+s,q) = \Phi(t)\circ \Phi(s,q) \quad\text{almost surely.}
\end{equation}
In Theorem~\ref{T:1601}, by contrast, there were no exceptional null sets.  For all practical purposes, the two notions of measure-preserving group are equivalent; see \cite{MR0143874} and the earlier work referenced therein.  This minor nuisance arises precisely from the uncountablitity of the set of times $t$.  Our first lemma shows that the almost-sure limit of measure-preserving groups (in the sense \eqref{aegp}) is likewise a measure preserving group.

\begin{lemma}\label{L:aegp}
Let $Q$ be a Polish space which is endowed with a Borel probability measure.  For each $\kappa\in 2^\N$, let $\Phi_\kappa:\R\times Q\to Q$ be a one-parameter group of measure preserving transformations.  Suppose that for each $t\in\R$, the sequence $\Phi_\kappa(t,q)$ converges almost surely \slp{as $\kappa\to\infty$} and let $\Phi(t,q)$ denote the limiting function.  Then $\Phi:\R\times Q\to Q$ is a one-parameter group of measure-preserving transformations of $Q$.
\end{lemma}

\begin{proof}
The fact that $\Phi(t)$ must also be measure-preserving is elementary: For any bounded continuous function $F:Q\to\R$, the dominated convergence theorem shows
\begin{align*}
\E\bigl\{ F \circ \Phi(t,q) \bigr\} = \lim_{\kappa\to\infty} \E\bigl\{ F \circ \Phi_\kappa(t,q) \bigr\}
	= \E\bigl\{ F (q) \bigr\}.
\end{align*}

In a similar vein, we may verify \eqref{aegp} by proving that 
\begin{align}\label{aegp'}
\E\bigl\{ \bigl| F \circ \Phi(t)\circ\Phi(s,q) - F \circ \Phi(t+s,q) \bigr|\bigr\} = 0
\end{align}
for any bounded continuous function $F:Q\to\R$.  Evidently, for any $\kappa\in 2^\N$,
\begin{align}
\text{LHS\eqref{aegp'}}
&\leq\E\bigl\{ \bigl| F\circ\Phi(t)\circ\Phi(s,q) - F\circ\Phi(t)\circ\Phi_\kappa(s,q) \bigr| \bigr\}\notag\\
&\quad + \E\bigl\{ \bigl| F\circ\Phi(t)\circ\Phi_\kappa(s,q) - F\circ\Phi_\kappa(t)\circ\Phi_\kappa(s,q) \bigr|\bigr\} \label{aegp''}\\
&\quad + \E\bigl\{ \bigl| F\circ\Phi_\kappa(t+s,q) - F\circ\Phi(t+s,q)  \bigr|\bigr\}. \notag
\end{align}
Our goal is to show that each term here vanishes in the limit $\kappa\to\infty$.  For the last term this is immediate from the dominated convergence theorem.  The middle term may be treated likewise after first exploiting the measure-preserving property of $\Phi_\kappa(s)$.  The first term in RHS\eqref{aegp''} requires a little more discussion.

We apply Lusin's theorem to the bounded measurable function $F\circ\Phi(t,q)$:  For every $\eps>0$, there is a bounded continuous function $F_\eps:Q\to\R$ so that
$$
\E \{|  F\circ\Phi(t,q) - F_\eps(q)|\} < \eps.
$$
From this observation, the measure preserving property of both $\Phi$ and $\Phi_\kappa$, as well as the dominated convergence theorem, we find that
\begin{align*}
\limsup_{\kappa\to\infty}\E\bigl\{ & \bigl| F\circ\Phi(t)\circ\Phi(s,q) - F\circ\Phi(t)\circ\Phi_\kappa(s,q) \bigr| \bigr\} \\
&\leq 2\eps +  \limsup_{\kappa\to\infty} \E\bigl\{ \bigl| F_\eps\circ\Phi(s,q) - F_\eps\circ\Phi_\kappa(s,q) \bigr| \bigr\}
	\leq 2\eps .
\end{align*}
As $\eps>0$ was arbitrary this completes the proof of \eqref{aegp'} and so of the lemma.
\end{proof}

The next proposition compares the flows under $H_{2\kappa}$ and $H_\kappa$.  The rather poor metric employed here will be upgraded in Corollary~\ref{C:1951}.

\begin{proposition}\label{P:1951}
Fix $\kappa\in 2^{\mathbb{N}}$ and $T>0$. Then, for any $1\leq r<\infty$,
\begin{align}\label{1961}
\sup_{|t|\leq T} \  \E \Big\{ \bigl\| \langle x\rangle^{-9} \bigl[ \Phi_{2\kappa} (t,q_\infty^0) - \Phi_{\kappa} (t,q_\infty^0)\bigl]
	\bigr\|_{H^{-5}_{x}}^{r} \Big\} 	\lesssim \kappa^{-r/8}, 
\end{align}
with an implicit constant that is independent of $\kappa$.
\end{proposition}

\begin{proof}
Fix $t$.  As $\Phi_\kappa(t,q_\infty^0)$ is Gibbs distributed, Lemma~\ref{L:commune} guarantees that, almost surely, one may find a solution $q(s)$ to \eqref{ccmm} with
\begin{align*}
q(t) := \Phi_{2\kappa} (t) \circ \Phi_{\kappa} (-t,q(0)) \qtq{and} q(0) = \Phi_\kappa(t,q_\infty^0). 
\end{align*}
Evidently,
\begin{align*}
\Phi_{2\kappa} (t,q_\infty^0) - \Phi_{\kappa} (t,q_\infty^0) = q(t) - q(0).
\end{align*}
Thus we wish to estimate how far this trajectory moves from its initial data; this will be accomplished by estimating the time derivative.

Recall that $q(s)$ is Gibbs distributed for all $s\in\R$.  Moreover, we may rewrite \eqref{ccmm} as
\begin{align}\label{1982}
\tfrac{d}{ds} q(s) &= 4\bigl\{ 3\kappa^2 q + \kappa^3 g_-(\kappa;q(s))-8\kappa^3 g_-(2\kappa;q(s))\bigr\}' .
\end{align}
Commuting the derivative outermost, this shows that our goal of proving \eqref{1961} has been reduced to showing that
\begin{align}\label{1985}
\E \Big\{  \| \jb{x}^{-9} \bigl[ 3\kappa^2 q + \kappa^3 g_-(\kappa;q)-8\kappa^3 g_-(2\kappa;q) \bigr] \bigr \|_{H^{-4}_{x}}^{r} \Bigr\}
	\lesssim \kappa^{-r/8}
\end{align}
for any Gibbs-distributed random variable $q$.

To proceed, it is convenient to split the central quantity into three parts.  Using \eqref{prformula} and combining terms, we may write 
\begin{align*}
3\kappa^2 q + \kappa^3 g_-(\kappa;q)-8\kappa^3 g_-(2\kappa;q)  = \mathsf{Lin}(q) + \mathsf{Nonlin}_1(q) + \mathsf{Nonlin}_2(q),
\end{align*}
where 
\begin{gather*}
\mathsf{Lin}(q) := \tfrac{3\kappa^2 \partial^4}{(16\kappa^2-\partial^2)(4\kappa^2-\partial^2)} \, q,
	\qquad \mathsf{Nonlin}_1(q) : =\tfrac{\kappa^2 \partial^2}{4\kappa^2-\partial^2} [q\gamma(\kappa)]
	- \tfrac{4 \kappa^2 \partial^2}{16\kappa^2-\partial^2} [q\gamma(2\kappa)],\\
\text{and} \qquad\qquad
	\mathsf{Nonlin}_2(q):=q\bigl[\kappa^2\gamma(\kappa) - 4\kappa^2\gamma(2\kappa)\bigr].
\end{gather*}

The first two terms here can be satisfactorily estimated without exhibiting any further cancellation.   Performing commutation in the style of \eqref{9:37} and then using the moment bound \eqref{Wsp moments}, we find
\begin{align*}
\E \Big\{  \|\jb{x}^{-9}  \, \mathsf{Lin}(q) \|_{H^{-4}_{x}}^r \Bigr\}
&\lesssim \kappa^{-2r} \E \Big\{  \|\jb{x}^{-9}q\|_{L^2}^r \Bigr\} \lesssim_r \kappa^{-2r}.
\end{align*} 
Proceeding similarly and employing \eqref{gbd'} and H\"older's inequality (in probability space), we find 
\begin{align*}
\E\Bigl\{ \bigl\|\jb{x}^{-9} \mathsf{Nonlin}_{1}(q)\bigr\|_{H^{-4}_{x}}^{r} \Bigr\}
	& \lesssim  \E \Bigl\{ \|\jb{x}^{-1} q \|_{L^\infty}^{r}\bigl[\|\jb{x}^{-4} \gamma(\kappa) \|_{L^2}^{r}
		+ \|\jb{x}^{-4} \gamma(2\kappa) \|_{L^2}^{r} \bigr] \Bigr\} \\
	&\lesssim \kappa^{-3r/2}.
\end{align*}
These bounds are clearly satisfactory for verifying \eqref{1985}.

We now move on to the task of estimating the contribution of $\mathsf{Nonlin}_{2}(q)$ in a satisfactory manner.  The key cancellation comes from the leading behaviour of $\gamma(\kappa)$ as $\kappa\to\infty$.  Concretely,
\begin{align}\label{2015}
\E \Big\{  \bigl\|\jb{x}^{-8}  \bigl[\kappa^2 \gamma(\kappa) + \tfrac12 q^2 \bigr] \bigr\|_{L^2}^{r} \Bigr\} \lesssim \kappa^{-r/8}
	\qtq{for any} 1\leq r < \infty
\end{align}
as we will show below.  Using H\"older's inequality (and the freedom to modify both $\kappa$ and $r$), we easily see that the bound \eqref{2015} suffices to complete our task: 
\begin{align*}
\E\Bigl\{ \bigl\|\jb{x}^{-9} \mathsf{Nonlin}_{2}(q)\bigr\|_{L^1}^{r} \Bigr\}
	& \lesssim  \E \Bigl\{ \|\jb{x}^{-1} q \|_{L^2}^{r}	
		\bigl\|\jb{x}^{-8} \bigl[\kappa^2\gamma(\kappa) - 4\kappa^2\gamma(2\kappa)\bigr] \bigr\|_{L^2}^{r}  \Bigr\} \\
	&\lesssim \kappa^{-r/8}.
\end{align*}

To verify \eqref{2015}, we first use the quadratic identity \eqref{qquad} to write 
\begin{align*}
\kappa^2 \gamma(\kappa) = -\tfrac12 \kappa^2 \bigl[g_-(\kappa) + g_+(\kappa) \bigr]\bigl[g_-(\kappa) - g_+(\kappa)\bigr]
	- \tfrac12 \kappa^2 \gamma(\kappa)^2 
\end{align*}
and then \eqref{prformula} to write
\begin{align*}
\kappa\bigl[ g_-(\kappa) \pm g_+(\kappa) \bigr] = q \mp \tfrac{\partial}{2\kappa\pm\partial} q + \tfrac{2\kappa}{2\kappa\pm\partial} q\gamma(\kappa).
\end{align*}
Combining these identities gives a (rather expansive) expression for $\kappa^2 \gamma(\kappa) + \tfrac12 q^2$ that we must estimate in $L^2(\R)$ and then in $L^r(d\PP)$.  Next we explain how to estimate the resulting terms.

Using \eqref{gbd'} and Proposition~\ref{PROP:Gibbs}, we find
\begin{gather*}
\E \Big\{  \|\jb{x}^{-4} q \|_{L^4}^r \Bigr\}  \lesssim 1 \\
\E \Big\{  \|\jb{x}^{-4} \tfrac{2\kappa}{2\kappa\pm\partial} q\gamma(\kappa) \|_{L^4}^r \Bigr\}
	\lesssim \E \Bigl\{ \|\jb{x}^{-1} q \|_{L^\infty}^{4} \|\jb{x}^{-3} \gamma(\kappa) \|_{H^1}^{r} \Bigr\} \lesssim \kappa^{-r/2} \\
\E \Big\{  \|\jb{x}^{-8} \kappa^2 \gamma(\kappa)^2 \|_{L^2}^r \Bigr\}
	\lesssim \E \Bigl\{ \kappa^{2r} \|\jb{x}^{-4} \gamma(\kappa) \|_{L^2}^{3r/2} \|\jb{x}^{-4} \gamma(\kappa) \|_{H^1}^{r/2} \Bigr\}
	\lesssim \kappa^{-r/2}.
\end{gather*}
This leaves only $\tfrac{\partial}{2\kappa\pm\partial} q$, which we estimate as follows:
\begin{align*}
\E \Big\{  \|\jb{x}^{-4} \tfrac{\partial}{2\kappa\pm\partial} q \|_{L^4}^r \Bigr\}
	&\lesssim \E \Big\{  \| \tfrac{1}{2\kappa\pm\partial} \jb{x}^{-4}  q \|_{H^{5/4}}^r \Bigr\} \\
	& \lesssim \kappa^{-r/8} \E \Big\{  \| \jb{x}^{-4}  q \|_{H^{3/8}}^r \Bigr\}  \lesssim \kappa^{-r/8},
\end{align*}
by using \eqref{Wsp moments} with $p=2$ and $s=\frac38$.
\end{proof}

\begin{corollary}\label{C:1951}
For any $T>0$ and triple \eqref{spb},
\begin{align}\label{2053}
\E \Big\{ \; \sum_{\kappa\in 2^{\mathbb{N}}} \;
	\bigl\| \bigl[ \Phi_{2\kappa} (t,q_\infty^0) - \Phi_{\kappa} (t,q_\infty^0)\bigl]
	\bigr\|_{C_t([-T,T];\Wspb)}  \; \Big\} 	< \infty .
\end{align}
\end{corollary}

\begin{proof}
In Proposition~\ref{P:1951}, the supremum in time is outside the expectation.  We remedy this first.

For a general function $f$ on $[-T,T]$ and any $N\in\N$, $\alpha\in(0,1)$ we have
\begin{align*}
\| f \|_{C_t([-T,T];L^2)} \lesssim \sum_{|n|\leq N} \| f(\tfrac nN T)\|_{H^{1/4}}^{\frac{20}{21}}\| f(\tfrac nN T)\|_{H^{-5}}^{\frac{1}{21}}
	+ T^\alpha N^{-\alpha} \bigl\| f \bigr\|_{C_t^\alpha ([-T,T];L^2)} .
\end{align*}
Using \eqref{1776} and \eqref{1961} and choosing $N\approx \kappa^\phi$ with $0<\phi<\tfrac1{168}$, we deduce that
\begin{align}\label{2069}
\E\Bigl\{ \bigl\| \langle x\rangle^{-9} \bigl[ q_{2\kappa}(t) - q_{\kappa}(t)\bigl] \bigr\|_{C_t([-T,T];L^2)} \Bigr\}
	\lesssim_T  N \kappa^{-\frac{1}{168}} + N^{-\alpha}  \lesssim_T \kappa^{-\theta} 
\end{align}
for some $\alpha>0$ and some $\theta>0$.

Evidently, \eqref{2069} is summable and yields one particular case of \eqref{2053}.  By interpolation with \eqref{1776}, we see that any $C_t\Wspb$ norm can likewise be bounded by a negative power of $\kappa$ and so can also be summed.
\end{proof}

Before we turn to the proof of Theorem~\ref{T:1936}, it is convenient to make some further preparations for the treatment of part~(b).  Concretely, we wish to make a clearer connection between the evolution of the diagonal Green's functions under the $H_\kappa$ flow and under \eqref{mkdv}.  Notice that the four quantities discussed in our next lemma come from \eqref{-flow}, \eqref{gflow}, \eqref{-Hkflow}, and \eqref{gHkflow}, respectively.

\begin{lemma}\label{L:2094}
Given $\vk\geq 1$ and $\kappa\geq2\vk$, let
\begin{align*}
X(\vk,q) &:= - g_-'''(\vk) + 6 q^2 g_-'(\vk), \\
Y(\vk,q) &:= - \gamma'''(\vk)  + \bigl\{ 6q^2[1+\gamma(\vk)] - 12\vk q g_-(\vk) - 12\vk^2 \gamma(\vk) \bigr\}', \\
X_\kappa(\vk,q) &:= 4\kappa^2 g_-'(\vk) - \tfrac{8\kappa^4\vk}{\kappa^2-\vk^2} \Big\{ g_+(\kappa)\bigl[ 1 + \gamma(\vk) \big]
	- \big[ 1 + \gamma(\kappa)\bigr] g_+(\vk) \Big\}, \\
Y_\kappa(\vk,q) &:= 4\kappa^2 \gamma'(\vk) - \tfrac{8\kappa^5\vk }{(\kappa^2-\vk^2)^{2}} \Bigl\{ g_-(\kappa)g_-(\vk) \Bigr\}' \\
& \qquad + \tfrac{4\kappa^4(\kappa^2+\vk^2) }{(\kappa^2-\vk^2)^{2}}
	\Bigl\{ g_+(\kappa)g_+(\vk) - \gamma(\kappa)- \gamma(\vk) - \gamma(\kappa)\gamma(\vk)\Bigr\}' .
\end{align*}
Then 
\begin{align*}
\bigl\| \jb{x}^{-12} & [X_\kappa-X]\bigr\|_{H^{-1}}  + \bigl\| \jb{x}^{-12} [Y_\kappa-Y]\bigr\|_{H^{-1}} \\
&\lesssim_\vk  \bigl( 1 + \|q\|_{L^2_1}^2 \bigr)^{4}
	\Bigl[\kappa^{-\frac14} \| \jb{x}^{-1} q \|_{H^{\frac14}} + \bigl\| \jb{x}^{-8}  [ 2\kappa^2 \gamma(\kappa) + q^2 ] \bigr\|_{L^2}  \Bigr] .
\end{align*}
\end{lemma}

\begin{proof}
We begin with some exact identities based on \eqref{pderiv}, \eqref{rderiv}, and \eqref{prformula}:
\begin{align*}
4\kappa^2 g_-'(\vk) +  \tfrac{8\kappa^4\vk}{\kappa^2-\vk^2} g_+(\vk) + g_-'''&(\vk) - 4 q^2 g_-'(\vk) \\
	&=  - \tfrac{4\vk^4}{\kappa^2-\vk^2} g_-'(\vk)  - 4\vk \bigl\{ q [1+\gamma(\vk)] \bigr\}' - 4 q^2 g_-'(\vk)\\
&= - \tfrac{4\vk^4}{\kappa^2-\vk^2} g_-'(\vk)  - 4\vk [1+\gamma(\vk)] q'   \\
\tfrac{8\kappa^4\vk}{\kappa^2-\vk^2} \bigl\{ \gamma(\kappa)  g_+(\vk) \bigr\} - 2q^2 g_-'(\vk) 
	&= - \tfrac{\kappa^2}{\kappa^2-\vk^2} 4\kappa^2 \gamma(\kappa) g_-'(\vk) - 2q^2 g_-'(\vk) \\
&= -  \tfrac{\kappa^2}{\kappa^2-\vk^2} [ 4\kappa^2 \gamma(\kappa) + 2q^2 ] g_-'(\vk) + \tfrac{2\vk^2}{\kappa^2-\vk^2} q^2 g_-'(\vk) \\
- \tfrac{8\kappa^4\vk}{\kappa^2-\vk^2} g_+(\kappa)\bigl[ 1 + \gamma(\vk) \big] &= \bigl[ 1 + \gamma(\vk) \big] \cdot
	\tfrac{16\kappa^4\vk}{(\kappa^2-\vk^2)(4\kappa^2-\partial^2)} [ q + \gamma(\kappa) q]' .
\end{align*}
Summing these identities gives a new way of writing $X_\kappa-X$, namely,
\begin{align}\label{XX}
X_\kappa-X = &{} - \tfrac{4\vk^4}{\kappa^2-\vk^2} g_-'(\vk) -  \tfrac{\kappa^2}{\kappa^2-\vk^2} [ 4\kappa^2 \gamma(\kappa) + 2q^2 ] g_-'(\vk) \\
&{} + \tfrac{2\vk^2}{\kappa^2-\vk^2} q^2 g_-'(\vk)
	+ 4\vk [1+\gamma(\vk)] \cdot \tfrac{(\kappa^2 - \vk^2)\partial^2 + 4\kappa^2\vk^2}{(\kappa^2-\vk^2)(4\kappa^2-\partial^2)} q' \notag \\
&{}	+ \bigl[ 1 + \gamma(\vk) \big] \cdot \tfrac{16\kappa^4\vk \partial}{(\kappa^2-\vk^2)(4\kappa^2-\partial^2)} \gamma(\kappa) q . \notag
\end{align}
We estimate these terms one-by-one.  In doing so, we note that $L^1\hookrightarrow H^{-1}$ and permit implicit constants to depend on $\vk$, but not $\kappa$.

From \eqref{pbd'}, we estimate the contribution of the first summand by 
\begin{align*}
 \kappa^{-2} \| \jb{x}^{-3} g_-'(\vk) \|_{H^{-1}} \lesssim \kappa^{-2}\bigl( 1 + \|q\|_{L^2_1}^2 \bigr) \| \jb{x}^{-1} q \|_{L^2} ,
\end{align*}
which is clearly acceptable.

We estimate the second term in \eqref{XX} by
\begin{align*}
  \bigl\| \jb{x}^{-12}  [ 4\kappa^2 \gamma(\kappa) + 2q^2 ] g_-'(\vk) \bigr\|_{L^1}
  	\lesssim \bigl\| \jb{x}^{-8}  [ 2\kappa^2 \gamma(\kappa) + q^2 ] \bigr\|_{L^2} \bigl\| \jb{x}^{-4} g_-'(\vk) \bigr\|_{L^2},
\end{align*}
which \eqref{pbd'} demonstrates to be acceptable.

The contribution of the third term in \eqref{XX} may be bounded by
\begin{align*}
 \kappa^{-2}  \bigl\| \jb{x}^{-12} q^2 g_-'(\vk) \bigr\|_{L^1}
 	\lesssim \kappa^{-2} \bigl\| \jb{x}^{-1} q \bigr\|_{L^2}^2 \bigl\| \jb{x}^{-9} g_-(\vk) \bigr\|_{H^2},
\end{align*}
which \eqref{rbd'} shows is acceptable.

As a stepping stone to estimating the full fourth term in \eqref{XX}, we note
\begin{align*}
\bigl\| \jb{x}^{-8} \tfrac{(\kappa^2 - \vk^2)\partial^2 + 4\kappa^2\vk^2}{(\kappa^2-\vk^2)(4\kappa^2-\partial^2)} q'  \bigr\|_{H^{-1}}
&\lesssim \bigl\| \jb{x}^{-8} \tfrac{\partial^2}{4\kappa^2-\partial^2} q'  \bigr\|_{H^{-1}}
	+ \kappa^{-2} \bigl\| \jb{x}^{-1} q  \bigr\|_{L^2} \\
&\lesssim \kappa^{-1/4} \| \jb{x}^{-1} q \|_{H^{1/4}}.
\end{align*}
From \eqref{gbd'}, we see that $\jb{x}^{-4} \gamma(\vk) \in H^1$.  By the duality and the algebra property of $H^1$, one sees that multiplication by an $H^1$ function maps $H^{-1}$ boundedly to itself.  Thus, the full fourth term may be estimated by  
\begin{align*}
\bigl\| \jb{x}^{-12} [1+\gamma(\vk)]  \tfrac{(\kappa^2 - \vk^2)\partial^2 + 4\kappa^2\vk^2}{(\kappa^2-\vk^2)(4\kappa^2-\partial^2)} & q'  \bigr\|_{H^{-1}} \\
&\quad\lesssim \bigl[ 1 + \| \jb{x}^{-4} \gamma(\vk)\|_{H^1} \bigr] \kappa^{-\frac14} \| \jb{x}^{-1} q \|_{H^{1/4}},
\end{align*}
which is acceptable; see \eqref{gbd'}.

Let us now consider the fifth and final term in \eqref{XX}.  Mirroring the preceding argument, we begin by noting that
\begin{align*}
\bigl\| \jb{x}^{-12} [ 1 + \gamma(\vk)] \tfrac{16\kappa^4\vk \partial}{(\kappa^2-\vk^2)(4\kappa^2-\partial^2)} & \gamma(\kappa) q\bigr\|_{H^{-1}} \\
&\qquad\lesssim \bigl[ 1 + \| \jb{x}^{-4} \gamma(\vk)\|_{H^1} \bigr] \| \jb{x}^{-8} \gamma(\kappa) q \|_{L^2}.
\end{align*}
To proceed, we use Gagliardo--Nirenberg, Sobolev,  and \eqref{gbd'} to see that
\begin{align}\label{2130}
\| \jb{x}^{-8} \gamma(\kappa) q \|_{L^2} &\lesssim \bigl\| \jb{x}^{-4} \gamma(\kappa) \bigr\|_{L^4} \bigl\| \jb{x}^{-4} q \bigr\|_{L^4} \notag \\
&\lesssim \bigl\| \jb{x}^{-4} \gamma(\kappa) \bigr\|_{L^2}^{1/2}\bigl\| \jb{x}^{-4} \gamma(\kappa) \bigr\|_{H^1}^{1/2}
		\bigl\| \jb{x}^{-4} q \bigr\|_{H^{1/4}}\\
&\lesssim \kappa^{-1} \bigl( 1 + \|q\|_{L^2_1}^2 \bigr)^2 \bigl\| \jb{x}^{-1} q \bigr\|_{H^{1/4}}. \notag
\end{align}
Together these estimates complete our treatment of \eqref{XX}.

We turn now to $Y_\kappa-Y$. From \eqref{pderiv} and \eqref{Gderiv}, we have
\begin{align*}
\gamma''(\vk) = 2q' g_+(\vk) - 4 \vk q g_-(\vk) + 4 q^2 [1+\gamma(\vk)] 
\end{align*}
and consequently,
\begin{align}\label{2140}
Y &= \bigl\{ - 2q' g_+(\vk)  + 2q^2[1+\gamma(\vk)] - 8\vk q g_-(\vk) - 12\vk^2 \gamma(\vk) \bigr\}' . 
\end{align}
Anticipating the cancellations we need,  we regroup and reorder terms in $Y_\kappa$ thus:
\begin{align}\label{2141}
Y_\kappa &=\Bigl\{ \tfrac{4\kappa^4(\kappa^2+\vk^2) }{(\kappa^2-\vk^2)^{2}} g_+(\kappa)g_+(\vk)
	  - \tfrac{4\kappa^4(\kappa^2+\vk^2) }{(\kappa^2-\vk^2)^{2}} \gamma(\kappa)\bigl[ 1 +\gamma(\vk)\bigr] \\
& \qquad\qquad  - \tfrac{8\kappa^5\vk }{(\kappa^2-\vk^2)^{2}}  g_-(\kappa)g_-(\vk)  
	+ \Bigl[ 4\kappa^2 - \tfrac{4\kappa^4(\kappa^2+\vk^2) }{(\kappa^2-\vk^2)^{2}} \Bigr] \gamma(\vk) \Bigr\}'. \notag
\end{align}

To proceed, we estimate the differences of corresponding pairs of terms taken from \eqref{2140} and \eqref{2141}.

Setting aside the common factor of $g_+(\vk)$, the difference of the first pair of terms produces the following, which we then rewrite using
\eqref{prformula}:
\begin{align*}
\tfrac{4\kappa^4(\kappa^2+\vk^2) }{(\kappa^2-\vk^2)^{2}} g_+(\kappa) + 2q'
	=  \tfrac{12\vk^2\kappa^4-4\kappa^2\vk^4}{(\kappa^2-\vk^2)^{2}} g_+(\kappa) -2 \tfrac{\partial^2}{4\kappa^2-\partial^2} q'
		- 2  \tfrac{4\kappa^2}{4\kappa^2-\partial^2} [\gamma(\kappa)q]'  .
\end{align*}
We estimate the first summand using the $L^2$ bound \eqref{pbd'}:
\begin{align*}
\bigl\| \jb{x}^{-9} \tfrac{12\vk^2\kappa^4 - 4\kappa^2\vk^4}{(\kappa^2-\vk^2)^{2}} g_+(\kappa) \bigr\|_{H^{-1}}
	\lesssim \kappa^{-1} \bigl( 1 + \|q\|_{L^2_1}^2 \bigr) \| \jb{x}^{-1} q \|_{L^2}.
\end{align*}
For the second summand, we use
\begin{align*}
\bigl\| \jb{x}^{-9} \tfrac{\partial^2}{4\kappa^2-\partial^2} q' \bigr\|_{H^{-1}} \lesssim
		\bigl\| \tfrac{\partial^2}{4\kappa^2-\partial^2}  \jb{x}^{-1}  q \bigr\|_{L^2} \lesssim \kappa^{-1/4} \| \jb{x}^{-1} q \|_{H^{1/4}}.
\end{align*}
For the third summand we use \eqref{2130} to obtain
\begin{align*}
\bigl\| \jb{x}^{-9} \tfrac{4\kappa^2}{4\kappa^2-\partial^2} [\gamma(\kappa)q]' \bigr\|_{H^{-1}}
	&\lesssim \kappa^{-1} \bigl( 1 + \|q\|_{L^2_1}^2 \bigr)^2 \bigl\| \jb{x}^{-1} q \bigr\|_{H^{1/4}}.
\end{align*}
In this way, we see that the contribution of the first pair of terms is bounded by
\begin{align*}
\kappa^{-1/4} \bigl( 1 + \|q\|_{L^2_1}^2 \bigr)^2 \bigl\| \jb{x}^{-1} q \bigr\|_{H^{1/4}} \| \jb{x}^{-3} g_+(\vk) \|_{H^1},
\end{align*}
which \eqref{pbd'} shows is acceptable.

The second pair of terms leads us to consider 
\begin{align*}
	  \tfrac{4\kappa^4(\kappa^2+\vk^2) }{(\kappa^2-\vk^2)^{2}} \gamma(\kappa) + 2q^2 
	  	= \tfrac{12\vk^2\kappa^4 - 4\kappa^2\vk^4}{(\kappa^2-\vk^2)^{2}}  \gamma(\kappa) + \bigl[4\kappa^2 \gamma(\kappa) + 2q^2 \bigr] ,
\end{align*}
which we estimate thus:
\begin{align*}
	  \bigl\| \jb{x}^{-8} \bigl[ \tfrac{4\kappa^4(\kappa^2+\vk^2) }{(\kappa^2-\vk^2)^{2}} \gamma(\kappa) + 2q^2 \bigr] \bigr\|_{H^{-1}}
	  	\lesssim \| \jb{x}^{-8}  \gamma(\kappa)\|_{L^2} + \bigl\| \jb{x}^{-8}  [ 2\kappa^2 \gamma(\kappa) + q^2 ] \bigr\|_{L^2} .
\end{align*}
Combining this with \eqref{gbd'} and
\begin{align*}
	  \| \jb{x}^{-4} [ 1 +  \gamma(\kappa)] \|_{L^\infty} \lesssim 1 + \bigl\| \jb{x}^{-4} \gamma(\kappa) \bigr\|_{H^1},
\end{align*}
we deduce that this pair makes an acceptable contribution.

Regarding the third pair of terms, we first use \eqref{prformula} to write
\begin{align}\label{2237}
q  - \tfrac{\kappa^5}{(\kappa^2-\vk^2)^{2}}  g_-(\kappa) = - \tfrac{\partial^2}{4\kappa^2-\partial^2} q
	- \tfrac{4\kappa^2}{4\kappa^2-\partial^2} [\gamma(\kappa)q] - \tfrac{\kappa^4-(\kappa^2-\vk^2)^{2}}{(\kappa^2-\vk^2)^{2}}  \kappa g_-(\kappa)
\end{align}
from which we deduce via \eqref{gbd'} and \eqref{pbd'} that
\begin{align*}
\bigl\| \jb{x}^{-8} \cdot \text{LHS\eqref{2237}}  \bigr\|_{H^{-1}}
	&\lesssim \kappa^{-1} \| \jb{x}^{-1} q  \|_{L^2} + \bigl\| \jb{x}^{-8} \gamma(\kappa)q \bigr\|_{L^1}
		+ \bigl\| \jb{x}^{-8} g_-(\kappa) \bigr\|_{L^2} \\
	&\lesssim \kappa^{-1} \bigl( 1 + \|q\|_{L^2_1}^2 \bigr)^2 \bigl\| \jb{x}^{-1} q \bigr\|_{L^2}   .
\end{align*}
Estimating the remaining factor $\jb{x}^{-3} g_-(\vk)$ using \eqref{pbd'} yields a satisfactory bound.

Combining the final terms in \eqref{2141} and \eqref{2140} yields
\begin{align*}
\Bigl[4\kappa^2 - \tfrac{4\kappa^4(\kappa^2+\vk^2) }{(\kappa^2-\vk^2)^{2}} + 12\vk^2\Bigr] \gamma(\vk)
	=\tfrac{12\vk^6 - 20 \kappa^2 \vk^4}{(\kappa^2-\vk^2)^{2}}\gamma(\vk) ,
\end{align*}
which \eqref{gbd'} shows is acceptable.
\end{proof}

\begin{proof}[Proof of Theorem~\ref{T:1936}]
Let us adopt the abbreviation $q_\kappa(t):=\Phi_\kappa(t,q_\infty^0)$.

Corollary~\ref{C:1951} shows that $q_\kappa(t)$ almost surely converges in $C_t([-T,T];\Wspb)$ for any $T>0$ and any one triple \eqref{spb}.  Thus, such convergence holds for all $T\in 2^\N$ and all rational triples and thence (by imbedding properties of these spaces) for all triples.

This convergence together with Theorem~\ref{T:1601} and Lemma~\ref{L:aegp} show that the limiting family of transformations $\Phi(t,q)$ forms a one-parameter group of transformations preserving Gibbs measure.  In particular, $q(t):=\Phi(t,q_\infty^0)$ is Gibbs distributed at every time.

Combining the convergence shown earlier, \eqref{1776}, and Fatou's lemma proves~\eqref{1910}.

Next we show that $q(t)$ is almost surely a distributional solution.  Using \eqref{prformula}, we may rewrite the $H_\kappa$ evolution equation as
\begin{align*}
 \tfrac{d}{dt} q_\kappa &= \bigl\{ 4\kappa^2 q_\kappa - 4\kappa^3 g_-(\kappa;q_\kappa) \bigr\}'
 	= \bigl\{  - \tfrac{4\kappa^2}{4\kappa^2-\partial^2} q_\kappa''
 	- \tfrac{4\kappa^2}{4\kappa^2-\partial^2} \bigl[ 4\kappa^2 \gamma(\kappa;q_\kappa) \, q_\kappa \bigr] \bigr\}' .
\end{align*}
On the other hand, $q_\kappa(t)\to q(t)$ in $C_t^{ }L^3_1$ and so
\begin{align*}
\bigl\{ - \tfrac{4\kappa^2}{4\kappa^2-\partial^2} q_\kappa''  + \tfrac{4\kappa^2}{4\kappa^2-\partial^2} \bigl[ 2 q_\kappa^3 \bigr]\bigr\}'
	\to \bigl\{ -q''+ 2q^3  \bigr\}'
 \end{align*}
almost surely as spacetime distributions.  To finish, it suffices to show that
\begin{align*}
\lim_{\kappa\to\infty} \int_{-T}^T \|\jb{x}^{-9} \bigl[ 4\kappa^2 \gamma(\kappa;q_\kappa) \, q_\kappa + 2 q_\kappa^3 \bigr] \|_{L^1} \,dt = 0	
\end{align*}
almost surely.   Recalling that $q_\kappa(t)$ is Gibbs distributed at every time, we may deduce this from \eqref{2015}; indeed, together with \eqref{log13}, this shows
\begin{align*}
\E \Bigl\{ \sum_{\kappa\in 2^\N} &\int_{-T}^T \bigl\|\jb{x}^{-9} \bigl[ 4\kappa^2 \gamma(\kappa;q_\kappa) \, q_\kappa + 2 q_\kappa^3 \bigr]
	\bigr\|_{L^1} \,dt \Bigr\}  \\
&\lesssim  T \sum_{\kappa\in 2^\N} \E \Bigl\{ \bigl\|\jb{x}^{-1} q_\kappa \bigr\|_{L^\infty} \bigl\|\jb{x}^{-8} \bigl[ \kappa^2 \gamma(\kappa;q_\kappa) + \tfrac12 q_\kappa^2 \bigr] \bigr\|_{L^1} \Bigr\}< \infty 	.
\end{align*}

It remains so verify that $q(t)$ is almost surely a green solution.  From Theorem~\ref{T:1601}(b), we know that
\begin{gather}\label{2293}
g_-(\vk;q_\kappa(t_2)) - g_-(\vk;q_\kappa(t_1)) = \int_{t_1}^{t_2} X_\kappa(\vk,q_\kappa(t)) \,dt, \\
	\label{2294}
\gamma(\vk;q_\kappa(t_2)) - \gamma(\vk;q_\kappa(t_1)) = \int_{t_1}^{t_2} Y_\kappa(\vk,q_\kappa(t)) \,dt .
\end{gather}
Here and below we employ the notations of Lemma~\ref{L:2094}.  Our approach is to deduce that $g_-(\vk;q(t))$ satisfies \eqref{Green-} and that $\gamma(\vk;q(t))$ satisfies \eqref{Greeng}, almost surely, by sending $\kappa\to\infty$ in \eqref{2293} and \eqref{2294}, respectively.

As $q_\kappa(t) \to q(t)$, Proposition~\ref{P:G Lip} guarantees convergence of the left-hand sides.  It also implies that
$$
X(\vk,q_\kappa(t)) \to X(\vk,q (t)) \qtq{and} Y(\vk,q_\kappa(t)) \to Y(\vk,q(t))
$$
as spacetime distributions, almost surely.  This leaves us to show that as $\kappa\to\infty$,
\begin{align}\label{2305}
X(\vk,q_\kappa(t)) - X_\kappa(\vk,q_\kappa(t)) \to 0 \qtq{and} Y_\kappa (\vk,q_\kappa(t)) - Y(\vk,q_\kappa(t)) \to 0
\end{align}
in the same sense.  Noting that $q_\kappa(t)$ is Gibbs distributed at all times and recalling \eqref{Wsp moments} and \eqref{2015}, Lemma~\ref{L:2094} implies that
\begin{align*}
\sum_\kappa \ \E \Bigl\{\ \int_{t_1}^{t_2} \bigl\| \jb{x}^{-12} \bigl[ X(\vk,q_\kappa(t)) - X_\kappa(\vk,q_\kappa(t)) \bigr]\bigr\|_{H^{-1}}\,dt \Bigr\} < \infty
\end{align*}
and analogously for $Y_\kappa (\vk,q_\kappa(t)) - Y(\vk,q_\kappa(t))$.  This certainly implies \eqref{2305}.
\end{proof}

\section{New invariant measures for KdV in infinite volume}\label{S:7}

\newcommand{\Hs}{\mathsf{H}}
\newcommand{\Rs}{\mathsf{R}}
\newcommand{\cj}{\overline}
\newcommand{\g}{\gamma}
\newcommand{\al}{\alpha}

The goal of this section is to show that applying the Miura map \eqref{Miura} to the Gibbs-distributed solutions to \eqref{mkdv} described in Theorem~\ref{T:1936}, we may construct new invariant measures for \eqref{kdv}.  We begin our discussion with some of the relevant operator theory.

The Lax operators pertinent to the \eqref{mkdv} equation take the form of a one-dimensional Dirac operator.  Dirac famously arrived at his model by factoring the three-dimensional Schr\"odinger equation in a manner inspired by the energy-momentum relation of Einstein.  After a minor change of basis, we will likewise see that the square of $\Lax$ yields a pair of one-dimensional Schr\"odinger operators; see \eqref{D2}.  The idea of factoring Sturm--Liouville operators as a product of first-order differential operators is considerably older, beginning in the late nineteenth century; see \cite{MR10757} for historical references.

The minor change of basis alluded to above is based on the orthogonal matrix 
\begin{align}\label{mcO}
\mathcal{O} := \frac{1}{\sqrt{2}} 
\begin{bmatrix}
1 & 1\\
1  & -1
\end{bmatrix} .
\end{align} 
Evidently, $\mathcal{O}^{-1}=\mathcal O^T = \mathcal{O}$. Moreover, 
\begin{align*}
 i \mathcal{O} \Lax \mathcal{O}= 
i \begin{bmatrix}
 0 & -\partial - q \\
-\partial + q & 0
\end{bmatrix}
\end{align*}
and consequently,
\begin{align}\label{D2}
-\mathcal{O} \Lax^{2}\mathcal{O}= 
\begin{bmatrix}
-\partial^{2}+q'+q^2 & 0\\
0  & -\partial^2-q'+q^2
\end{bmatrix}.
\end{align}

Recall that $\Lax$ was defined as an anti-selfadjoint operator in Proposition~\ref{P:Lax Op} under the mild condition \eqref{det L hyp}.  This gives meaning to \eqref{D2} as a self-adjoint operator.  We shall develop the theory of this operator only to the extent that we need and do so in a manner dictated by the means we have at our disposal.  For alternative, more thorough discussions, see \cite{MR2189502} and the references therein.

\begin{proposition}
For any $q$ satisfying \eqref{det L hyp}, the functionals
\begin{align}\label{Schrod forms}
\mathsf q_\pm(\psi) := \int |\psi' \mp q\psi|^2\,dx = \int |\psi'|^2 + [q^2 \pm q'] |\psi|^2 \,dx,
\end{align}
with domains $\mathsf D_\pm:=\{ \psi\in L^2(\R) : \psi\in H^1_\loc(\R) \text{ and } \psi' \mp q\psi\in L^2\}$, are the quadratic forms of positive semidefinite selfadjoint operators $\mathsf H_\pm$.  Moreover,
\begin{align}\label{D2'}
-\mathcal{O} \Lax^{2}\mathcal{O} = \mathsf H_+ \oplus \mathsf H_-
\end{align}
and the resolvents $\mathsf R_\pm(\kappa):=(\mathsf H_\pm + \kappa^2)^{-1}$ of\/ $\mathsf H_\pm$ satisfy
\begin{align}\label{Reskdvtomkdv}
\mathsf R_\pm(\kappa)=\tfrac{1}{2\kappa}\bigl[ \LaxR_{11}(\kappa)+\LaxR_{22}(\kappa) \pm \LaxR_{12}(\kappa) \pm \LaxR_{21}(\kappa)\bigr],
\end{align}
where $\LaxR_{jk}$ are the component operators of\/ $\LaxR(\kappa)=(\Lax+\kappa)^{-1}$.
\end{proposition}

\begin{proof}
Proposition~\ref{P:Lax Op} shows that the domain of the self-adjoint operator $i \mathcal{O} \Lax \mathcal{O}$ is precisely $\mathsf D_+ \oplus \mathsf D_-$.  It is then an elementary matter to define selfadjoint operators $\mathsf H_\pm$ through \eqref{D2'} and verify that \eqref{Schrod forms} are the associated quadratic forms; see \cite{RS1} for further background.

From the functional calculus of the selfadjoint operator $ i \mathcal{O} \Lax \mathcal{O}$, we have
\begin{align*}
\begin{bmatrix}
\Rs_{+}(\kappa) & 0\\
0  & \Rs_{-}(\kappa)
\end{bmatrix} =\mathcal{O}(\kappa^2-\Lax^{2})^{-1}\mathcal{O}
	= \tfrac{1}{2\kappa}\mathcal{O}\Big[ (\kappa+\Lax)^{-1} + (\kappa-\Lax)^{-1} \Big]\mathcal{O} .
\end{align*}
The claim \eqref{Reskdvtomkdv} follows from this and the symmetries \eqref{k symmetry}.
\end{proof}

In Proposition~\ref{P:Lax Op}, we also described the Green's functions associated to the Lax operator $\Lax$. Recall that those Green's functions were not continuous; they contain jumps at the diagonal of the same magnitude as the free Green's function \eqref{G0}.  These jumps cancel out in the combinations \eqref{Reskdvtomkdv}:

\begin{corollary}\label{C:2417}
For any $q$ satisfying \eqref{det L hyp} with $b>0$, the Schr\"odinger operators $\mathsf H_\pm$ admit Green's functions $\mathsf G(x,y;\kappa,q)$ that are continuous in $(x,y)$, real-analytic in $\kappa>0$, and satisfy 
\begin{align}\label{2470}
| \mathsf G_\pm(y_1,y_2;\kappa) | \lesssim \langle y_1 \rangle^b \langle y_2 \rangle^b \bigl[ 1 + \kappa^{-1} \| \langle x\rangle^{-b} q \|_{L^2}^2  \bigr] .
\end{align}
Their restrictions $\mathsf g_\pm(x;\kappa):=\mathsf G_\pm(x,x;\kappa)$ to the diagonal satisfy
\begin{align}\label{gestimate}
\bigl\| \langle x\rangle^{-4b} \bigl[ \mathsf g_\pm(x;\kappa) - \tfrac1{2\kappa} \bigr] \bigr\|_{H^{1}_\kappa}
	\lesssim_b \kappa^{-1} \|\langle x\rangle^{-b}q\|_{L^2} \Bigl[ 1+\kappa^{-1}\|\langle x\rangle^{-b}q\|_{L^2}^{2} \Bigr]^{3/2} 
\end{align}
and may also be written as
\begin{align}\label{gtogp0}
\mathsf g_\pm(x;\kappa) = \tfrac{1}{2\kappa} \big[ 1+ \gamma(x;\kappa,q) \pm g_+(x;\kappa,q)\big], 
\end{align}
using the diagonal Green's functions associated to $\Lax$.
\end{corollary}

\begin{proof} The claims regarding $\mathsf G_\pm$ follow from \eqref{Reskdvtomkdv}, \eqref{G0},  and Proposition~\ref{P:Lax Op}.  To obtain \eqref{gtogp0}, we simply combine \eqref{Reskdvtomkdv} with \eqref{gamma} and \eqref{p}.  Lastly, \eqref{gestimate} follows from \eqref{gtogp0}, \eqref{gbd'}, and \eqref{pbd'}.
\end{proof}

The Miura transformations $w_\pm(x;q) := q(x)^2 \pm q'(x)$  were introduced in \cite{MR0252825} and marked the beginning of the discovery of the intimate connection between \eqref{kdv} and Sturm--Liouville problems.  Miura's observation was that if $q$ is a smooth solution to \eqref{mkdv}, then both $w_\pm$ are smooth solutions of \eqref{kdv}.

As the transformation $q\mapsto -q$ preserves the Gibbs law, there is no advantage to considering both Miura maps.  We will confine our attention to just
\begin{align}\label{Miura}
w(x;q) := q'(x) + q(x)^2 .
\end{align}

For very general functions $q$, we may interpret the derivative in \eqref{Miura} distributionally.  In subsection~\ref{SS:7.1}, we will examine the resulting family of random distributions.  In particular, we will give meaning to our claim that samples `have the regularity of white noise', while also affirming that the image measures are mutually singular to white noise and to each other (as $\mu$ varies).  

In order to view $w$ as a solution to \eqref{kdv}, even distributionally, we must understand the \emph{square} of $w$.   As observed in subsection~\ref{SS:diffusion}, samples $q$ from the Gibbs state are not absolutely continuous on finite intervals or indeed, even of bounded variation.  In this way, $q'$ may not be interpreted as a function and so cannot be squared.  (The other terms in $w^2$ do make sense as distributions.)    It was precisely to give an intrinsic notion of what it means for such irregular objects to be a solution of \eqref{kdv} that the notion of green solutions was first introduced in \cite{KMV} building on the use of diagonal Green's functions in \cite{KV}.  In subsection~\ref{SS:7.2}, we will demonstrate that the solutions of \eqref{mkdv} described in Theorem~\ref{T:1936} are indeed transformed by \eqref{Miura} into green solutions of \eqref{kdv}; see \eqref{GreensolnKdV}.  In choosing \eqref{GreensolnKdV} as the defining notion of green solution for \eqref{kdv}, we have deviated slightly from \cite{KMV}.  In that paper, the authors selected (3.6) from \cite{KV} as the basis for defining green solutions.  Here we have adopted (3.10) from \cite{KV}  instead.

In order to interpret these solutions to \eqref{kdv} as a measure-preserving system, it is crucial that each initial data leads to a single solution.  We will achieve this goal by showing that the Miura map is almost surely injective, building on the analysis of \cite{MR2189502}.

\subsection{The Miura mapping of Gibbs states}~\label{SS:7.1}
First we wish to understand the smoothness properties of the distributions $w=q'+q^2$ by connecting them to the regularity of Brownian motion.  
For our purposes, the most convenient description of the law of white noise is as the derivative of Brownian motion.  More precisely, given a realization $B(x)$ of Brownian motion indexed by $x\in\R$, the random tempered distribution $\mathring w(x)$ defined by
\begin{align}\label{2505}
\langle \phi, \mathring w\rangle := -\int \phi'(x) B(x)\,dx \qtq{for all} \phi \in \mathcal S(\R), 
\end{align}
is said to be white noise distributed.  Observe that
\begin{align}\label{2506}
\E\bigl\{ \exp( i \langle \phi, \mathring w\rangle )\bigr\} = \exp\bigl\{ - \tfrac12 \|\phi\|^2_{L^2} \bigr\} ,
\end{align}
which, by the Minlos Theorem, may also serve as a definition of white noise.

\begin{proposition}\label{P:7.3}
Fix $\mu\in\R$ and suppose that the process $q$, defined on some probability space, follows the infinite-volume Gibbs law with chemical potential $\mu$.  The random distribution $w=q'+q^2$ has the same regularity as white noise in the following sense:  One may define a realization $\mathring w$ of white noise over the same probability space so that the random distribution
$$
\phi \mapsto \bigl\langle\phi, w - \mathring w \bigr\rangle
$$
may be represented, almost surely, by a continuous function.  Nevertheless, the law of $w$ is not Gaussian and is mutually singular to that of white noise.
\end{proposition}

\begin{proof}
Recall that we can construct a Brownian motion on our underlying probability space via \eqref{aBM}.  We then use this to construct our realization $\mathring w$ of white noise via \eqref{2505}.  By construction,
$$
\bigl\langle\phi, w - \mathring w \bigr\rangle
	= \int \phi(x) \bigl[ \tfrac{\psi_0'}{\psi_0}\bigl(q(x)\bigr) + q(x)^2\bigr]\,dx.
$$
Thus we see that this distribution may be represented, almost surely, by a continuous function.

For any $\phi\in C^\infty_c(\R)$, we have   
\begin{align}
\langle \phi^2, w \rangle =  \langle \phi, \mathsf H_+ \phi \rangle  - \langle \phi', \phi' \rangle \geq - \|\phi'\|_{L^2}^2 .
\end{align}
This already shows that $w$ is not Gaussian.

The same inequality holds for the translated function $x\mapsto \phi(x-y)$.  It follows that for any $\lambda>0$, the event
\begin{align}\label{wevent}
\limsup_{L\to\infty} \frac{1}{L} \int_0^L \exp\bigl\{ - \lambda \langle \phi^2, w(\cdot + y) \rangle\bigr\} \,dy
	\leq \exp\bigl\{ \lambda \|\phi'\|_{L^2}^2 \bigr\}
\end{align}
has full probability.  For white noise, however, ergodicity and \eqref{2506} imply
\begin{align}\label{w0event}
\lim_{L\to\infty} \frac{1}{L} \int_0^L \exp\bigl\{ - \lambda \langle \phi^2, \mathring w(\cdot + y) \rangle\bigr\} \,dy
	=  \exp\bigl\{ \tfrac12 \lambda^2 \|\phi\|_{L^4}^4 \bigr\}
\end{align}
almost surely.  Provided $\phi\not\equiv 0$, we may choose $\lambda$ large enough so that \eqref{wevent} and \eqref{w0event} are mutually exclusive conditions on a distribution.  It follows that the laws of $w$ and of white noise are mutually singular.
\end{proof}

Next we prove that the Miura map is almost surely injective by combining the deterministic analysis of \cite{MR2189502} with probabilistic tools.

\begin{proposition}\label{P:inject}
The Miura map is almost surely injective: The set
\begin{align}\label{not inj}
 \bigl\{ q\in C(\R) : \exists \tilde q\in C(\R) \text{ with } q\neq\tilde q \text{ but }
 	q'+q^2 = \tilde q' + q^2 \text{ as distributions}\bigr\}
\end{align}
is a null set for the infinite-volume Gibbs law, irrespective of $\mu\in\R$. 
\end{proposition}

\begin{proof}
Suppose $q$ belongs to the set \eqref{not inj}, as witnessed by $\tilde q\in C(\R)$.  Then both
\begin{align}\label{2498}
 y(x) = \exp\bigl[ \textstyle{\int_0^x} q(s)\,ds \bigr]
 	\qtq{and} \tilde y(x) = \exp\bigl[ \textstyle{\int_0^x} \tilde q(s)\,ds \bigr]
\end{align}
are positive distributional solutions to $y''=[q'+q^2]y$.  Moreover, $y(0)=\tilde y(0) =1$.  Our goal is to show that $y= \tilde y$, almost surely, and consequently, $q= \tilde q$.

The classical reduction of order method of d'Alembert--Lagrange provides the general solution of a Sturm--Liouville problem once one particular solution is known.  The paper \cite{MR2189502} realized this idea for potentials that are merely $H^{-1}_\loc(\R)$ and from their analysis, we conclude that 
\begin{align}\label{2499}
 \tilde y(x) = y(x) \biggl[  1 + c  \int_0^x \frac{ds}{y(s)^2} \biggr]
 	\qtq{for some} c\in \R.
\end{align}

In what follows, we will show that
\begin{align}\label{2500}
\limsup_{x\to\infty} \int_0^x \frac{ds}{y(s)^2}  =\infty \qtq{and}
\limsup_{x\to-\infty} \int_x^0 \frac{ds}{y(s)^2}  =\infty
\end{align}
almost surely.  Recalling that $y(x)$ and $\tilde y(x)$ are both positive, the former shows that $c\geq 0$ and the latter that $c\leq 0$.  Thus, together they show that $c=0$ and so the injectivity claim will be verified.

By virtue of its continuity, $q(x)$  is integrable on every finite interval. Combining this with the definition \eqref{2498} of $y(x)$, we see that the two events described in \eqref{2500} are invariant under translations $q(x)\mapsto q(x+h)$.  As our process is ergodic under such translations (see \S\ref{SS:mix}), it follows that the probability of each of these events is either zero or one.

As the transformation $q(x)\mapsto q(-x)$ preserves Gibbs measure, the two events in \eqref{2500} have the same probability.  We will now use a similar symmetry to show that the probability of the former is at least $\frac12$ and so must be one.

By the Cauchy--Schwarz inequality,
\begin{align*}
x^2 \leq \int_0^x e^{- 2\int_0^s q(s')\,ds'}\,ds \int_0^x e^{2\int_0^s q(s')\,ds'}\,ds .
\end{align*}
It follows that if one of these integrals does not diverge as $x\to\infty$, then the other must diverge.  The first of these integrals is an exact recapitulation of the first integral in \eqref{2500}; the second corresponds to case where the Gibbs sample $q(x)$ is replaced by $-q(x)$.  However, the mapping $q\mapsto -q$ preserves the Gibbs measure and so the first event in \eqref{2500} must have a probability of $\frac12$ or greater.
\end{proof}

\begin{corollary}\label{C:mutsing}
Distinct values of the chemical potential lead to mutually singular laws for $w=q'+q^2$.
\end{corollary}

\begin{proof}
Given distinct chemical potentials, $\mu_1<\mu_2$,  Corollary~\ref{C:diff} shows that there are disjoint measurable subsets $A_1$ and $A_2$ of $C(\R)$ that have full Gibbs measure.   Proposition~\ref{P:inject} shows that although the images of these sets under the Miura map might intersect, any such intersection is a null set for the pushforward of both Gibbs measures.  
\end{proof}

\subsection{Invariant measures for KdV}~\label{SS:7.2}
By virtue of the almost-sure injectivity of the Miura map demonstrated in Proposition~\ref{P:inject}, the measure-preserving group of transformations $\Phi$ described in Theorem~\ref{T:1936} pushes forward to a measure-preserving group of transformations on the space of tempered distributions.  Our final theorem encapsulates the strong senses in which these orbits solve \eqref{kdv}:

\begin{theorem}\label{T:7.6}
Let $q(t)$ denote the statistical ensemble of solutions to \eqref{mkdv} constructed in Theorem~\ref{T:1936} and consider the process $w(t):= q'(t)+q(t)^2$.\\
\slp{a} Given $s>\frac12$ there exists $\alpha>0$ so that
\begin{align}\label{1610''}
 \E\Bigl\{ \| \langle x\rangle^{-2} w(t,x) \|_{C_{t}^{\alpha}([-T,T];H^{-s})}^{r} \Bigr\}<\infty
\end{align}
for every $T>0$ and every $1\leq r <\infty$.\\
\slp{b} $w(t)$ follows the same law for each $t\in \R$.\\
\slp{c} Almost surely, $w(t)$ is a green solution to \eqref{kdv} in the sense that
\begin{align}\label{GreensolnKdV}
\mathsf g(t_2;\kappa) = \mathsf g(t_1;\kappa) + \int_{t_1}^{t_2} \big\{ 2\mathsf g(t;\kappa)'' -6w(t)\mathsf g(t;\kappa)-12\kappa^2 \mathsf g(t;\kappa)\big\}'dt, 
\end{align}
as distributions.  Here $\kappa>0$ and $\mathsf g(t;\kappa)$ denotes the diagonal Green's function associated to the potential $w(t)$ by Corollary~\ref{C:2417}.
\end{theorem}

\begin{proof}
It suffices to consider $\tfrac12<s\leq 1$.  Choosing $b=1$ in \eqref{1910}, we see that $\langle x\rangle^{-1} q(t,x)$ is H\"older continuous as an $H^{1-s}(\R)$-valued function with all moments finite.  The claim \eqref{1610''} then follows easily.  To treat $q^2$, for example, one may just use the embedding $L^1(\R)\hookrightarrow H^{-s}(\R)$, which is dual to $H^s(\R)\hookrightarrow L^\infty(\R)$.

By Theorem~\ref{T:1936}, the process $q(t,x)$ follows the same Gibbs distribution for all $t\in\R$.  It follows that the law of $w(t)$ is likewise independent of $t\in\R$.

Recall Theorem~\ref{T:1936} also showed that $q(t)$ is almost surely a green solution to \eqref{mkdv}, which is to say, we showed that the identities \eqref{Green-} and \eqref{Greeng} both hold.  Taking a spatial derivative of the former and applying \eqref{Gderiv}, \eqref{pderiv}, and  \eqref{rderiv}, we deduce that
\begin{align}\label{Greener+}
g_+(t_2) &= g_+(t_1) + \int_{t_1}^{t_2} \bigl\{ 2g''_+ - 6q' [1+\gamma] - 6q^2 g_+ - 12\kappa^2 g_+ \bigr\}' \,dt \\
\label{Greenerg}
 \gamma(t_2) &= \gamma(t_1) + \int_{t_1}^{t_2} \bigl\{ 2\gamma'' - 6q' g_+ - 6q^2[1+\gamma] - 12\kappa^2 \gamma \bigr\}' \,dt
\end{align}
as distributions, for every choice of $t_1<t_2$.  Summing these identities and using \eqref{gtogp0}, we deduce \eqref{GreensolnKdV}.
\end{proof}

\end{document}